\documentclass[a4paper,12pt]{article}

\usepackage[T1]{fontenc}
\usepackage[utf8]{inputenc}

\usepackage[english]{babel}

\usepackage{amsmath}
\usepackage{amsfonts}
\usepackage{amssymb}
\usepackage{amsthm}
\usepackage{graphicx}

\usepackage{authblk}

\usepackage[colorlinks,
            pdftex,
            pdfauthor={},
            pdftitle={asymptotics CV},
            pdfsubject={},
            pdfkeywords={}]{hyperref}

\DeclareMathOperator{\argmin}{argmin}
\def\leq{\leqslant}
\def\geq{\geqslant}

\newtheorem{theorem}{Theorem}[section]

\newtheorem{lemma}[theorem]{Lemma}
\newtheorem{proposition}[theorem]{Proposition}

\theoremstyle{definition}

\newtheorem{remark}[theorem]{Remark}

\newcommand{\ensnombre}[1]{\mathbb{#1}}
\newcommand{\N}{\ensnombre{N}}

\newcommand{\Q}{\ensnombre{Q}}

\newcommand{\defeq}{\mathrel{\mathop:}=}
\newcommand{\eqdef}{\mathrel{=}:}

\def \P{\mathbb{P}}

\def \E{\mathbb{E}}
\DeclareMathOperator{\Var}{Var}

\newcommand{\fact}[1]{#1\mathpunct{}!}

\newcommand{\beq} {\begin{eqnarray*}}
\newcommand{\eeq} {\end{eqnarray*}}
\newcommand{\bei}{\begin{itemize}}
\newcommand{\eei}{\end{itemize}}

\def \R {\boldsymbol{R}}
\def \Q {\boldsymbol{Q}}
\def \r {\boldsymbol{r}}

\def \z {\boldsymbol{z}}
\def \e {\boldsymbol{e}}
\def \y {\boldsymbol{y}}
\def \f {\boldsymbol{f}}
\def \F {\boldsymbol{F}}
\def \bbeta {\boldsymbol{\beta}}
\def \W {\boldsymbol{W}}
\def \I {\boldsymbol{I}}
\def \g {\boldsymbol{g}}
\def \h {\boldsymbol{h}}

\usepackage{a4wide}

\title{Cross-validation estimation of covariance parameters under fixed-domain asymptotics }

\author[1]{F. Bachoc\footnote{Corresponding author. francois.bachoc@math.univ-toulouse.fr, Institut de Mathématiques de Toulouse, Université Paul Sabatier, 118 route de Narbonne, 31062 TOULOUSE Cedex 9 }}
\author[1]{A. Lagnoux}
\author[1]{T.M.N. Nguyen}

\affil[1]{Institut de Mathématiques de Toulouse, Université Paul Sabatier, France 
118 route de Narbonne, 31062 TOULOUSE Cedex 9 
\{francois.bachoc,agnes.lagnoux,thi\_mong\_ngoc.nguyen\}@math.univ-toulouse.fr
}

\begin{document}

\maketitle

\begin{abstract}

We consider a one-dimensional Gaussian process having exponential covariance function. Under fixed-domain asymptotics, we prove the strong consistency and asymptotic normality of a cross validation estimator of the microergodic covariance parameter.
In this setting, Ying \cite{Yin1991} proved the same asymptotic properties for the maximum likelihood estimator. Our proof includes several original or more involved components, compared to that of Ying. Also, while the asymptotic variance of maximum likelihood does not depend on the triangular array of observation points under consideration, that of cross validation does, and is shown to be lower and upper bounded. The lower bound coincides with the asymptotic variance of maximum likelihood. We provide examples of triangular arrays of observation points achieving the lower and upper bounds. We illustrate our asymptotic results with simulations, and provide extensions to the case of an unknown mean function. To our knowledge, this work constitutes the first fixed-domain asymptotic analysis of cross validation.
\end{abstract}


\section{Introduction}\label{sec:intro}

Kriging \cite{stein99interpolation,rasmussen06gaussian} consists in inferring the values of a Gaussian random field given observations at a finite set of observation points.
It has become a popular method for a
large range of applications, such as geostatistics \cite{matheron70theorie}, numerical code approximation \cite{sacks89design,santner03design,bachoc16improvement} and calibration \cite{paulo12calibration,bachoc14calibration} or global optimization \cite{jones98efficient}.

Before Kriging can be applied, a covariance function must be chosen.
The most common practice is to statistically estimate the covariance function, from a set
of observations of the Gaussian process, and to plug \cite[Ch.6.8]{stein99interpolation} the estimate in the Kriging equations.
Usually, it is assumed that the covariance function belongs to a given parametric family (see \cite{abrahamsen97review} for a review of classical families).
In this case, the estimation boils down to estimating the corresponding covariance parameters. For covariance parameter estimation, maximum likelihood (ML) is the most studied and used method, while cross validation (CV) \cite{sundararajan01predictive,zhang10kriging,bachoc13cross} is an alternative technique. CV has been shown to have attractive properties, compared to ML, when the parametric family of covariance functions is misspecified \cite{bachoc13cross,bachoc16asymptotic}.

There is a fair amount of literature on the asymptotic properties of ML. In this regard, the two main frameworks are increasing-domain and fixed-domain asymptotics \cite[p.62]{stein99interpolation}. Under increasing-domain asymptotics, the average density of observation points is bounded, so that the infinite sequence of observation points is unbounded. Under fixed-domain asymptotics, this sequence is dense in a bounded domain.

Consider first increasing-domain asymptotics. Generally speaking, for all (identifiable) covariance parameters, the ML estimator is consistent and asymptotically normal under some mild regularity conditions. The asymptotic covariance matrix is equal to the inverse of the (asymptotic) Fisher information matrix. This result was first shown in \cite{MarMar1984}, and then extended in different directions in \cite{cressie93asymptotic,cressie96asymptotics,
shaby12tapered,bachoc14asymptotic,furrer16asymptotic}. 

The situation is significantly different under fixed-domain asymptotics. Indeed, two types of covariance parameters can be distinguished: microergodic and non-microergodic parameters \cite{ibragimov78gaussian,stein99interpolation}. A covariance parameter is microergodic if, for two different values of it, the two corresponding Gaussian measures are orthogonal, see \cite{ibragimov78gaussian,stein99interpolation}. It is non-microergodic if, even for two different values of it, the two corresponding Gaussian measures are equivalent. Non-microergodic parameters cannot be estimated consistently, but have an asymptotically negligible impact on prediction \cite{AEPRFMCF,BELPUICF,UAOLPRFUISOS,zhang04inconsistent}. On the other hand, it is at least possible to consistently estimate microergodic covariance parameters, and misspecifying them can have a strong negative impact on prediction.

Under fixed-domain asymptotics, there exist results indicating which covariance parameters are microergodic, and providing the asymptotic properties of the corresponding ML estimator. Most of these available results are specific to particular covariance models. In dimension $d=1$ when the covariance model is exponential, only a reparameterized quantity obtained from the variance and scale parameters is microergodic. It is shown in \cite{Yin1991} that the ML estimator of this microergodic parameter is strongly consistent and asymptotically normal. These results are extended in \cite{CheSimYin2000}, by taking into account measurement errors, and in \cite{chang2017mixed}, by taking into account both measurement errors and an unknown mean function.
When $d>1$ and for a separable exponential covariance function, all the covariance parameters are microergodic, and the asymptotic normality of the ML estimator is proved in \cite{ying93maximum}. Other results in this case are also given in \cite{vdV2010,AbtWel1998}. Consistency of ML is shown as well in \cite{LohLam2000} for the scale parameters of the Gaussian covariance function and in \cite{Loh2005} for all the covariance parameters of the separable Mat\'ern $3/2$ covariance function. Finally, for the entire isotropic Mat\'ern class of covariance functions, all parameters are microergodic for $d>4$ \cite{And2010}, and only reparameterized parameters obtained from the scale and variance are microergodic for $d \leq 3$ \cite{zhang04inconsistent}. In \cite{ShaKau2013}, the asymptotic normality of the ML estimators for these microergodic parameters is proved, from previous results in \cite{DuZhaMan2009} and \cite{WanLoh2011}.
Finally we remark that, beyond ML, quadratic variation-based estimators have also been extensively studied, under fixed-domain asymptotics (see for instance \cite{istas97quadratic}).

In contrast to ML, CV has received less theoretical attention. Under increasing-domain asymptotics, the consistency and asymptotic normality of a CV estimator is proved in \cite{bachoc14asymptotic}. Also, under increasing-domain asymptotics, it is shown in \cite{bachoc16asymptotic} that this CV estimator asymptotically minimizes the integrated square prediction error. To the best of our knowledge, no fixed-domain asymptotic analysis of CV exists in the literature.

In this paper, we provide a first fixed-domain asymptotic analysis of the CV estimator minimizing the CV logarithmic score, see \cite{rasmussen06gaussian} Equation (5.11) and \cite{zhang10kriging}. We focus on the case of the one-dimensional exponential covariance function, which was historically the first covariance function for which the asymptotic properties of ML were derived \cite{Yin1991}. This covariance function is particularly amenable to theoretical analysis, as its Markovian property yields an explicit (matrix-free) expression of the likelihood function. It turns out that the CV logarithmic score can also be expressed in a matrix-free form, which enables us to prove the strong consistency and asymptotic normality of the corresponding CV estimator. We follow the same general proof architecture as in \cite{Yin1991} for ML, but our proof, and the nature of our results, contain several new elements.

In terms of proofs, the random CV logarithmic score, and its derivatives, have more complicated expressions than for ML. [This is because the CV logarithm score is based on the conditional distributions of the observations, from both their nearest left and right neighbors, while the likelihood function is solely based on the nearest left neighbors. See Lemma \ref{lem:Sn} and Lemma 1 in \cite{Yin1991} for details.] As a consequence, the computations are more involved, and some other tools than in \cite{Yin1991} are needed. In particular, many of our asymptotic approximations rely on Taylor expansions of functions of several variables (where each variable is an interpoint distance going to zero, see the proofs for details). In contrast, only Taylor approximations with one variable are needed in \cite{Yin1991}. In addition, we use central limit theorems for dependent random variables, while only independent variables need to be considered in \cite{Yin1991}.

The nature of our asymptotic normality result also differs from that in \cite{Yin1991}. In this reference, the asymptotic variance does not depend on the triangular array of observation points. On the contrary, in our case, different triangular arrays of observation points can yield different asymptotic variances. We exhibit a lower and an upper bound for these asymptotic variances, and provide examples of triangular arrays reaching them. The lower bound is in fact equal to the asymptotic variance of ML in \cite{Yin1991}. Interestingly, the triangular array given by equispaced observation points attains neither the lower nor the upper bound. It is also pointed out in \cite{bachoc14asymptotic} that equispaced observation points need not provide the smallest asymptotic variance for covariance parameter estimation.

Finally, the fact that the asymptotic variance is larger for CV than for ML is a standard finding in the well-specified case considered here, where the covariance function of the Gaussian process does belong to the parametric family of covariance functions under consideration. In contrasts, as mentioned above, CV has attractive properties compared to ML when this well-specified case does not hold \cite{bachoc13cross,bachoc16asymptotic}.

The rest of the paper is organized as follows. In Section \ref{sec:CV}, we present in more details the setting and the CV estimator under consideration. In Section \ref{sec:consistency}, we give our strong consistency result for this estimator. In Section \ref{sec:AN}, we provide the asymptotic normality result, together with the analysis of the asymptotic variance. In Section \ref{sec:num}, we present numerical experiments, illustrating our theoretical findings. In Section \ref{sec:reg}, we extend the results of Sections \ref{sec:consistency} and \ref{sec:AN} to the case of an unknown mean function.
In Section \ref{sec:ccl}, we give a few concluding remarks. All the proofs are postponed to Section \ref{sec:proofs}.

\section{The context and the cross-validation estimators}\label{sec:CV}

We consider a centered Gaussian process $Y$ on $[0,1]$ with covariance function 
\[
K_0(t_1,t_2) = \sigma_0^2 \exp\{- \theta_0 |t_1-t_2|\}
\]
for some fixed and unknown parameters $\theta_0 > 0$ and $\sigma_0^2 >0$. This process is commonly known as the Ornstein-Uhlenbeck process. It satisfies the following stochastic differential equation, called the Langevin's equation,
\[
dY(t)=-\theta_0 Y(t) dt+\sqrt{2\theta_0}\sigma_0 dB(t),
\]  
where $(B(t))_t$ denotes a standard Brownian motion process. The Ornstein-Uhlenbeck process has been widely used to
model physical, biological, social, and many other phenomena. It also possesses
many useful mathematical properties that simplify the analysis.

We introduce the parametric set of covariance functions $\{ K_{ \theta,\sigma^2}, a \leq \theta \leq A, b \leq \sigma^2 \leq B \}$ for some fixed $0 < a \leq A < + \infty$
and $0 < b \leq B < + \infty$ where 
\[
K_{ \theta,\sigma^2}(t_1,t_2) = \sigma^2 \exp\{- \theta |t_1-t_2|\}.
\]

For any $n \in \mathbb{N}$, we consider a design of observation points $\{s_1,...,s_n\}$. Without loss of generality, we may assume that  $0 = s_1 <...< s_n = 1$. Similarly as in \cite{Yin1991}, there is no need to assume that the sequences of observation points are nested.
We consider the vector of observations at locations $s_1,...,s_n$, $(Y(s_1),\ldots, Y(s_n))'$.   
Now let $\Delta_i \defeq s_i - s_{i-1}$, for $i=2,...,n$, and $y_i \defeq Y(s_i)$, for $i=1,...,n$. For ease of redaction, we do not mention in $s_i$ and $\Delta_i$ the dependency in $n$. We define $\R_{\theta}$ as the variance-covariance matrix of $(y_1,...,y_n)'$ under covariance function $K_{\theta,1}$,
\[
\R_{\theta}\defeq \left(\begin{matrix} 1& e^{-\theta\Delta_2} & \cdots & e^{-\theta\sum\limits_{i=2}^n\Delta_{i}} \\
e^{-\theta\Delta_2}& 1 &  \cdots& e^{-\theta\sum\limits_{i=3}^n\Delta_{i}}\\
\vdots& \vdots & \ddots& \vdots\\
e^{-\theta\sum\limits_{i=2}^n\Delta_{i}}&  
e^{-\theta\sum\limits_{i=3}^n\Delta_{i}} & \cdots & 1\\
\end{matrix}\right).
\]

From \cite{AntZag2010}, we have
\begin{equation} \label{eq:tridiagonal:inverse}
\R_{\theta}^{-1}=\left(\begin{matrix} \frac{1}{1-e^{-2\theta\Delta_2}}&\frac{-e^{-\theta\Delta_2}}{1-e^{-2\theta\Delta_2}}& 0&\cdots&0\\
\frac{-e^{-\theta\Delta_2}}{1-e^{-2\theta\Delta_2}}& \frac{1}{1-e^{-2\theta\Delta_2}}+\frac{e^{-2\theta\Delta_3}}{1-e^{-2\theta\Delta_3}}&\ddots& \ddots &\vdots\\
0 & \ddots & \ddots &  & 0\\
\vdots & \ddots &  & \frac{1}{1-e^{-2\theta\Delta_{n-1}}}+\frac{e^{-2\theta\Delta_n}}{1-e^{-2\theta\Delta_n}}& \frac{-e^{-\theta\Delta_n}}{1-e^{-2\theta\Delta_n}}\\
0& \cdots & 0 & \frac{-e^{-\theta\Delta_n}}{1-e^{-2\theta\Delta_n}}& 
\frac{1}{1-e^{-2\theta\Delta_n}}
\end{matrix} \right).
\end{equation}

We now address the CV estimators of $\theta_0$ and $\sigma^2_0$ considered in \cite{rasmussen06gaussian,zhang10kriging}. Let 
\[
\hat{Y}_{\theta,-i}(s_i) = \mathbb{E}_{ \theta,\sigma^2} ( Y(s_i)|Y(s_{1}),...,Y(s_{i-1}),Y(s_{i+1}),...,Y(s_{n}) ),
\]
where the conditional expectation $\mathbb{E}_{\theta,\sigma^2}$ is calculated assuming that $Y$ is centered and has covariance function $K_{\theta,\sigma^2}$. We remark that $\hat{Y}_{\theta,-i}(s_i)$ does not depend on $\sigma^2$.
We define similarly 
\[
\hat{\sigma}^2_{ \theta,\sigma^2,-i}(s_i) = \Var_{ \theta,\sigma^2} ( Y(s_i)|Y(s_{1}),...,Y(s_{i-1}),Y(s_{i+1}),...,Y(s_{n}) ).
\]
Then, the CV estimators are given by
\[
(\hat{\theta},\hat{\sigma}^2)
\in \underset{a \leq \theta \leq A,b \leq \sigma^2 \leq B}{\argmin}
S_n( \theta,\sigma^2),
\]
where
\begin{equation}\label{eq:defSn}
S_n( \theta,\sigma^2) = \sum_{i=1}^n 
\left[
 \log( \hat{\sigma}^2_{ \theta,\sigma^2,-i}(s_i) ) 
 + \frac{ (y_i -\hat{Y}_{\theta,-i}(s_i) )^2 }{ \hat{\sigma}^2_{ \theta,\sigma^2,-i}(s_i) }
 \right]
\end{equation}
is the logarithmic score. 
The rationale for minimizing the logarithmic score is that $ \log(2 \pi) + \log( \hat{\sigma}^2_{ \theta,\sigma^2,-i}(s_i) ) 
 + \frac{ (y_i -\hat{Y}_{\theta,-i}(s_i) )^2 }{ \hat{\sigma}^2_{ \theta,\sigma^2,-i}(s_i) }$ is equal to $-2$ times the conditional log-likelihood of $y_i$, given $(y_1,...,y_{i-1},y_{i+1},...,y_n)'$, with covariance parameters $\theta,\sigma^2$. The term cross-validation underlines the fact that we consider leave-one-out quantities.

As already known \cite{ibragimov78gaussian,Yin1991,zhang04inconsistent}, it is not possible to consistently estimate simultaneously $\theta_0$ and $\sigma_0^2$ (the ML estimator of $\theta_0$ is a non-degenerate random variable, even if $(Y(t))_{t \in [0,1]}$ is observed continuously \cite{ZhaZim2005}), but it is possible to consistently estimate $\theta_0 \sigma_0^2$. 
As a consequence, we have considered three different cases, as in \cite{Yin1991}. (i) Set $\sigma^2=\sigma_1^2$ in \eqref{eq:defSn} with $\sigma_1^2>0$ being a predetermined constant and consider the CV estimator $\hat{\theta}_1$ of $\theta_1=\theta_0 \sigma_0^2/\sigma_1^2$ that minimizes \eqref{eq:defSn} with $\sigma^2=\sigma_1^2$. (ii)  Set $\theta=\theta_2$ in \eqref{eq:defSn} with $\theta_2>0$ being a predetermined constant and consider the CV estimator $\hat{\sigma}_2^2$ of $\sigma_2^2=\theta_0 \sigma_0^2/\theta_2$ that minimizes \eqref{eq:defSn} with $\theta=\theta_2$.
(iii) Consider the estimator $\hat \theta\hat\sigma^2$ of $\theta_0\sigma_0^2$, where $\hat{\theta}$ and $\hat{\sigma}^2$  are the CV estimators of $\theta_0$ and $\sigma_0^2$.

Ying \cite{Yin1991} considers the ML estimators of $\theta$ and $\sigma^2$ and establishes their consistency and asymptotic normality. We carry out a similar asymptotic analysis for the above CV estimators. More precisely, we prove that $\hat \theta \hat \sigma^2$ (resp. $\hat \theta_1$ and $\hat \sigma_2^2$) converges almost surely to $\theta_0 \sigma_0^2$ (resp. $\theta_1$ and $\sigma_2^2$) in the next section. In section \ref{sec:AN}, we establish that, for a sequence $\tau_n$ which is lower and upper-bounded,
$(\sqrt{n} / [ \theta_1 \tau_n] )(\hat \theta_1-\theta_1)$, $(\sqrt{n} / [ \sigma_2^2 \tau_n] )(\hat \sigma_2^2-\sigma_2^2)$ and 
$(\sqrt{n} / [ \theta_0 \sigma_0^2 \tau_n ] )( \hat \theta \hat \sigma^2-\theta_0 \sigma_0^2)$ all converge in distribution to a standard Gaussian random variable. We remark that the asymptotic variance $\tau_n^2$ depends on how the underlying design points $\{s_1,...,s_n\}$ are chosen. On the contrary, considering  the ML estimators \cite{Yin1991}, the asymptotic variance is the same for any triangular array of design points.

\section{Consistency}\label{sec:consistency}

In this section, we establish the strong consistency of the CV estimator $\hat{\theta}\hat{\sigma}^2$ of $\theta_0\sigma_0^2$ described in the previous section. In that view, we consider $S_n(\theta,\sigma^2)$ defined by \eqref{eq:defSn}. As done in \cite{Yin1991}, we base our analysis on the Markovian property of the Ornstein-Uhlenbeck process in order to handle the fact that, as $n$ increases, the observed sample $(y_1,\ldots, y_n)'$ becomes more and more correlated. We have
\begin{equation} \label{eq:loo:y:pred}
\hat{Y}_{\theta,-i}(s_i)
 =
  - \sum_{\substack{j=1,\ldots,n; \\ j \neq i}} 
  \frac{
  (\R_{\theta}^{-1})_{ij}}{(\R_{\theta}^{-1})_{ii}
  }
  y_j
\end{equation}
and
\[
\hat{\sigma}^2_{ \theta,\sigma^2,-i}(s_i)
=
\frac{\sigma^2}{(\R_{\theta}^{-1})_{ii}},
\]
from \cite{zhang10kriging,bachoc13cross,dubrule83cross}. Then, using Equation \eqref{eq:tridiagonal:inverse}, we get the following lemma after some tedious computations.

\begin{lemma}[Logarithmic score]\label{lem:Sn} With $S_n( \theta,\sigma^2)$ as in \eqref{eq:defSn}, we have
\begin{eqnarray*}
S_n( \theta,\sigma^2) & = & n \log(\sigma^2)+\log( 1-e^{-2 \theta \Delta_2} ) +  \log( 1-e^{-2 \theta \Delta_n} )\\
& & + \frac{ (y_1 - e^{- \theta \Delta_2} y_2 )^2 }{\sigma^2 (1-e^{-2 \theta \Delta_2})}    + \frac{ (y_n - e^{- \theta \Delta_n} y_{n-1} )^2 }{ \sigma^2 (1-e^{-2 \theta \Delta_n}) }
 - \sum_{i=2}^{n-1} \log
\left(
 \frac{1}{
 1-e^{- 2 \theta \Delta_i}
 }
  + \frac{e^{- 2 \theta \Delta_{i+1} }
 }{
 1-e^{- 2 \theta \Delta_{i+1}}
 } 
 \right)
 \\
 &  &  + \frac{1}{\sigma^2} \sum_{i=2}^{n-1} 
 \left[
  \frac{1}{
 1-e^{- 2 \theta \Delta_i}
 }
  + \frac{e^{- 2 \theta \Delta_{i+1} }
 }{
 1-e^{- 2 \theta \Delta_{i+1}}
 } 
 \right]
 \left[
 y_i
 -
 \frac{
   \frac{e^{- \theta \Delta_i}}{
 1-e^{- 2 \theta \Delta_i}
 } y_{i-1}
  + \frac{e^{-\theta \Delta_{i+1} }
 }{
 1-e^{- 2 \theta \Delta_{i+1}}
 } y_{i+1}  
 }{
  \frac{1}{
 1-e^{- 2 \theta \Delta_i}
 }
  + \frac{e^{- 2 \theta \Delta_{i+1} }
 }{
 1-e^{- 2 \theta \Delta_{i+1}}
 } 
 }
 \right]^2.
\end{eqnarray*}
\end{lemma}

Based on Lemma \ref{lem:Sn}, we prove the following theorem in Section \ref{ssec:proof_cons}. 
\begin{theorem}[Consistency]\label{th:consistency} 
Assume that 
\begin{align}\label{ass:delta}
\underset{n\to +\infty}\limsup \underset{i=2,\ldots,n}\max \Delta_i=0.
\end{align}
Let $J = [a,A] \times [b, B]$, where $a$, $A$, $b$ and $B$ are fixed and have been defined in the previous section. Assume that there exists $(\tilde{\theta},\tilde{\sigma}^2)$ in $J$ so that $\tilde{\theta} \tilde{\sigma}^2 = \theta_0 \sigma_0^2$.
Define 
$(\hat{\theta},\hat{\sigma}^2) \in J$ as a solution of 
\begin{align}\label{eq:min_couple}
S_n(\hat{\theta},\hat{\sigma}^2)=\underset{( \theta,\sigma^2)\in J}{\min} S_n( \theta,\sigma^2).
\end{align}
Then $(\hat{\theta},\hat{\sigma}^2)$ exists and 
\begin{align}\label{eq:cv_produit}
\hat{\theta}\hat{\sigma}^2 \overset{a.s.}{\to} \theta_0\sigma_0^2.
\end{align}
In particular, let $\sigma_1^2> 0$ and $\theta_2 > 0$ be predetermined constants satisfying $  \sigma_0^2\theta_0/\sigma_1^2 \in [a,A]$ and
$\sigma_0^2\theta_0/\theta_2 \in [b,B]$. Define $\hat{\theta}_1 \in [a,A]$ and  $\hat{\sigma}_2^2 \in [b,B]$ as solutions of 
\begin{align}\label{eq:min_theta}
S_n(\hat{\theta}_1,\sigma_1^2)=\underset{\theta\in [a,A]}{\min} S_n(\theta,\sigma_1^2)
\end{align}
and
\begin{align}\label{eq:min_sigma}
S_n(\theta_2,\hat{\sigma}_2^2)=\underset{\sigma^2\in [b, B]}{\min} S_n( \theta_2,\sigma^2).
\end{align}
Then $\hat{\theta}_1 \overset{a.s.}{\to} \theta_1\defeq \sigma_0^2\theta_0/\sigma_1^2$ and $\hat{\sigma}_2^2 \overset{a.s.}{\to} \sigma_2^2\defeq \sigma_0^2\theta_0/\theta_2$.
\end{theorem}

\begin{remark}\label{rem:consistency}
It is worth remarking that the asymptotically preponderant terms in Lemma \ref{lem:Sn} are the same as those obtained in the context of ML estimation (see \cite{Yin1991} and Section \ref{ssec:proof_cons} for more details).
\end{remark}

\section{Asymptotic normality}\label{sec:AN}

Once the consistency has been established, the natural question of the convergence speed arises. We address this point in this section. We first provide a central limit result in the following theorem.

\begin{theorem}[Central Limit Theorem]\label{th:asymptotic}
Consider the same notation and assumptions as in Theorem \ref{th:consistency}.
Assume further that either $aB < \theta_0 \sigma_0^2$; $Ab > \theta_0 \sigma_0^2$ or $aB > \theta_0 \sigma_0^2$; $Ab < \theta_0 \sigma_0^2$ hold.
Then the estimators are asymptotically normal. More precisely, we have
\begin{equation}
\frac{\sqrt{n}}{\theta_0\sigma_0^2\tau_{n}}( \hat \theta \hat \sigma^2-\theta_0 \sigma_0^2) \xrightarrow[n\to \infty]{\mathcal D} \mathcal N (0,1).\label{eq:AN}
\end{equation}
Also, when $(\sigma_0^2 \theta_0) / \sigma_1^2 \in (a,A)$ we have
\begin{equation}
\frac{\sqrt{n}}{\theta_1\tau_{n}}(\hat \theta_1 -\theta_1) \xrightarrow[n\to \infty]{\mathcal D} \mathcal N (0,1).\label{eq:AN_1}\\
\end{equation}
Finally, when $(\sigma_0^2 \theta_0) / \theta_2 \in (b,B)$ we have
\begin{equation}
\frac{\sqrt{n}}{\sigma_2^2\tau_{n}}( \hat \sigma_2^2-\sigma_2^2) \xrightarrow[n\to \infty]{\mathcal D} \mathcal N (0,1).\label{eq:AN_2}\\
\end{equation}
The quantity $\tau_n^2$ depends on how the underlying design points $\{s_1,\ldots,s_n\}$ have been chosen. More precisely,
\begin{align} \label{eq:tau:n}
\tau_n^2=   \frac{2}{n}\sum_{i=3}^{n-1} \left[\left(\frac{\Delta_{i+1}}{\Delta_i+\Delta_{i+1}}+\frac{\Delta_{i-1}}{\Delta_i+\Delta_{i-1}}\right)^2
+2 \frac{\Delta_i\Delta_{i+1}}{(\Delta_i+\Delta_{i+1})^2} \right].
\end{align}

\end{theorem}

\begin{remark} \label{rem:assumptions}
The condition $aB < \theta_0 \sigma_0^2$; $Ab > \theta_0 \sigma_0^2$ or $aB > \theta_0 \sigma_0^2$; $Ab < \theta_0 \sigma_0^2$ ensures that $(\partial / \partial \theta) S_n( \hat{\theta},\hat{\sigma}^2 )$ or $(\partial / \partial \sigma^2) S_n( \hat{\theta},\hat{\sigma}^2 )$ will be equal to zero for $n$ large enough almost surely, by applying Theorem \ref{th:consistency}. This is used in the proof of Theorem \ref{th:asymptotic}. A similar assumption is made in \cite{Yin1991}, where the parameter domain for $(\theta,\sigma^2)$ is $(0,\infty) \times [b,B]$ or $[a,A]  \times (0,\infty)$. 
\end{remark}

In the following proposition, we show that the quantity $\tau_n^2$ in Theorem \ref{th:asymptotic} is lower and upper bounded, so that the rate of convergence is always $\sqrt{n}$ in this theorem.

\begin{proposition}\label{prop:encadrement}
We have, for any choice of the triangular array of design points $\{ s_1,...,s_n\}$ satisfying \eqref{ass:delta},
\begin{align}\label{eq:encadrement}
2\leq \underset{n \to \infty } \liminf  \tau^2_n\leq \underset{n \to \infty }\limsup \tau^2_n\leq 4.
\end{align}
\end{proposition}

\begin{remark} \label{rem:designs}
\begin{enumerate}
\item The asymptotic variance of the limiting distribution of $\hat \theta \hat \sigma^2-\theta_0 \sigma_0^2$ can be easily estimated. By the previous proposition, this asymptotic variance is always larger than the one of the ML estimator.
Indeed, with $\hat \theta_{ML}$ and $\hat \sigma^2_{ML}$ the ML estimators of $\theta$ and $\sigma^2$ we have $(\sqrt{n}/[\theta_0 \sigma_0^2])( \hat \theta_{ML} \hat \sigma_{ML}^2-\theta_0 \sigma_0^2 ) \xrightarrow[n\to \infty]{\mathcal D} \mathcal N (0,2)$, see \cite{Yin1991}. This fact is quite expected as ML estimates usually perform best when the covariance model is well-specified, as is the case here.
\item As one can check easily, the regular design $\Delta_i\equiv \frac{1}{n-1}$ for all $i=2,\ldots, n$, does not yield the limiting variance of the  ML estimator. Instead, we have $\tau_n^2 \to_{n \to \infty} 3$ for this design. However, in Proposition \ref{prop2}, we exhibit a particular design realizing the limiting variance of the ML estimator: $\underset{n \to \infty } \lim \tau^2_n=2$. 
\end{enumerate}
\end{remark}

In fact, the bounds in \eqref{eq:encadrement} are sharp as shown in the following proposition.

\begin{proposition}\label{prop2}
(i) Let $\{s_1,\dots,s_n\}$ be such that $s_1 = 0$, for $i=2,...,n-1$,
\begin{equation*}
  \Delta_i=
     \begin{cases}
       (1-\gamma_n)\frac{2}{n}  & \text{if $i$ is even}, \\
        \frac{2 \gamma_n}{n}   & \text{if $i$ is odd},
     \end{cases}
\end{equation*}
where $\gamma_n \in (0,1)$, and $\Delta_n = 1 - \sum_{i=2}^{n-1} \Delta_i$. Then, taking $\gamma_n=1/n$, we get
$\tau_n^2 \underset{n \to \infty} \to  4$. 

(ii) Let $\{s_1,\dots,s_n\}$ and $0<\alpha <1$ be such that $s_1=0$, $\Delta_i = 1/(\fact{i})$  for $i= \lfloor n^\alpha \rfloor+ 1, \dots, n$ and $\Delta_2=\dots=\Delta_{\lfloor n^\alpha \rfloor}\equiv (1-r_n)/( \lfloor n^\alpha \rfloor-1)$ with  $r_n\defeq \displaystyle\sum_{i=\lfloor n^\alpha \rfloor + 1}^{n}\Delta_i$. Then $\displaystyle\sum_{i=2}^{n}\Delta_i =1$ and $\tau_n^2 \underset{n \to \infty} \to  2$. 
\end{proposition}

\begin{remark}
Intuitively, in Proposition \ref{prop2} (ii), $\Delta_{i+1}$ will be much smaller than $\Delta_i$ for most of the indices $i$,  so that the quantities
$\frac{\Delta_{i+1}}{\Delta_i+\Delta_{i+1}}$ and $\frac{\Delta_i\Delta_{i+1}}{(\Delta_i+\Delta_{i+1})^2}$ in \eqref{eq:tau:n} will be negligible. We refer to the proof of Proposition \ref{prop2} for further details.
\end{remark}

\section{Numerical experiments}\label{sec:num}

We illustrate Theorem \ref{th:asymptotic} by a Monte Carlo simulation. We set $\theta_0 = 3$ and $\sigma_0^2 = 1$ and we consider three sample size values, $n=12,50,200$. For the sample size $n=12$, we address three designs $\{s_1,...,s_n  \}$. The first one is the `minimal' design given by Proposition \ref{prop2} (ii) with $\alpha = 0.5$, which asymptotically achieves the minimal estimation variance. The second one is the `regular' design given by $\{s_1,...,s_n \} = \{ 0, 1/(n-1),...,1\}$. The third one is the `maximal' design given by Proposition \ref{prop2} (i) with $\gamma_n = 1/n$, which asymptotically achieves the maximal estimation variance. These three designs are show in Figure \ref{fig:designs}.
For the sample sizes $n=50$ and $n=200$, the `minimal' design is not amenable to numerical computation anymore, as the values of $\Delta_i$ become too small; so that we only address the `regular' and `maximal' designs.

\begin{figure}[h!]
\centering
\begin{tabular}{c}
\includegraphics[height=4.5cm,width=5cm]{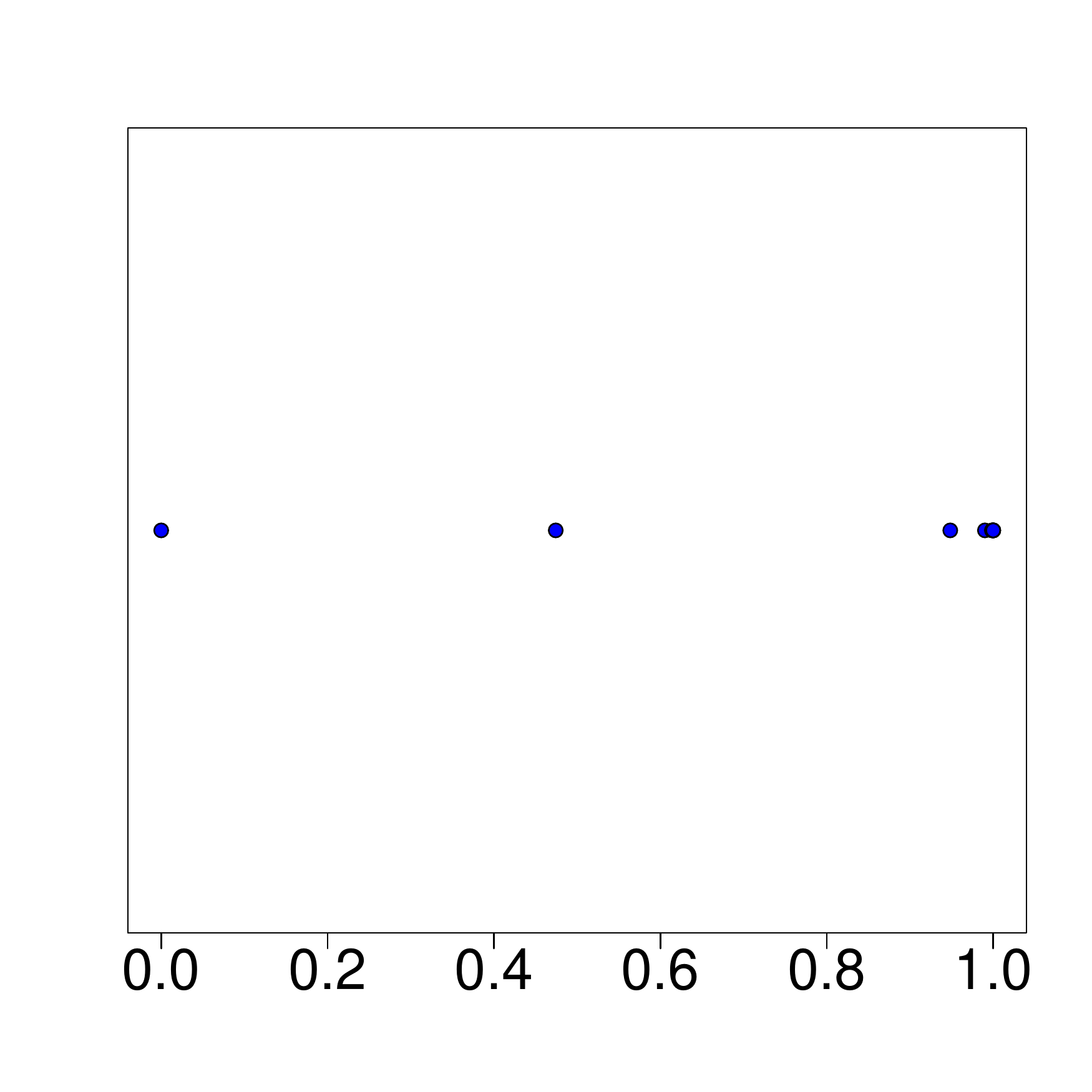}
\includegraphics[height=4.5cm,width=5cm]{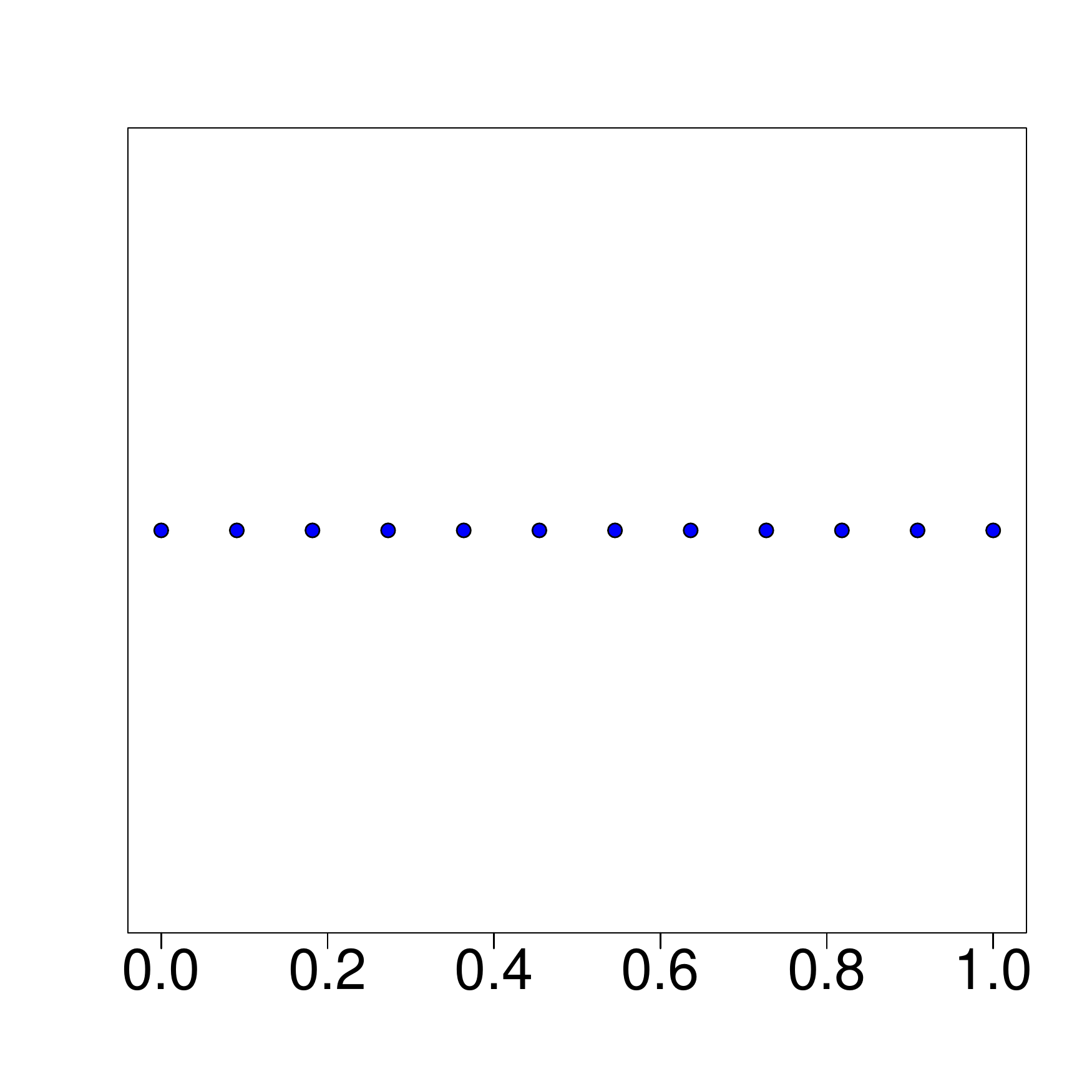}
\includegraphics[height=4.5cm,width=5cm]{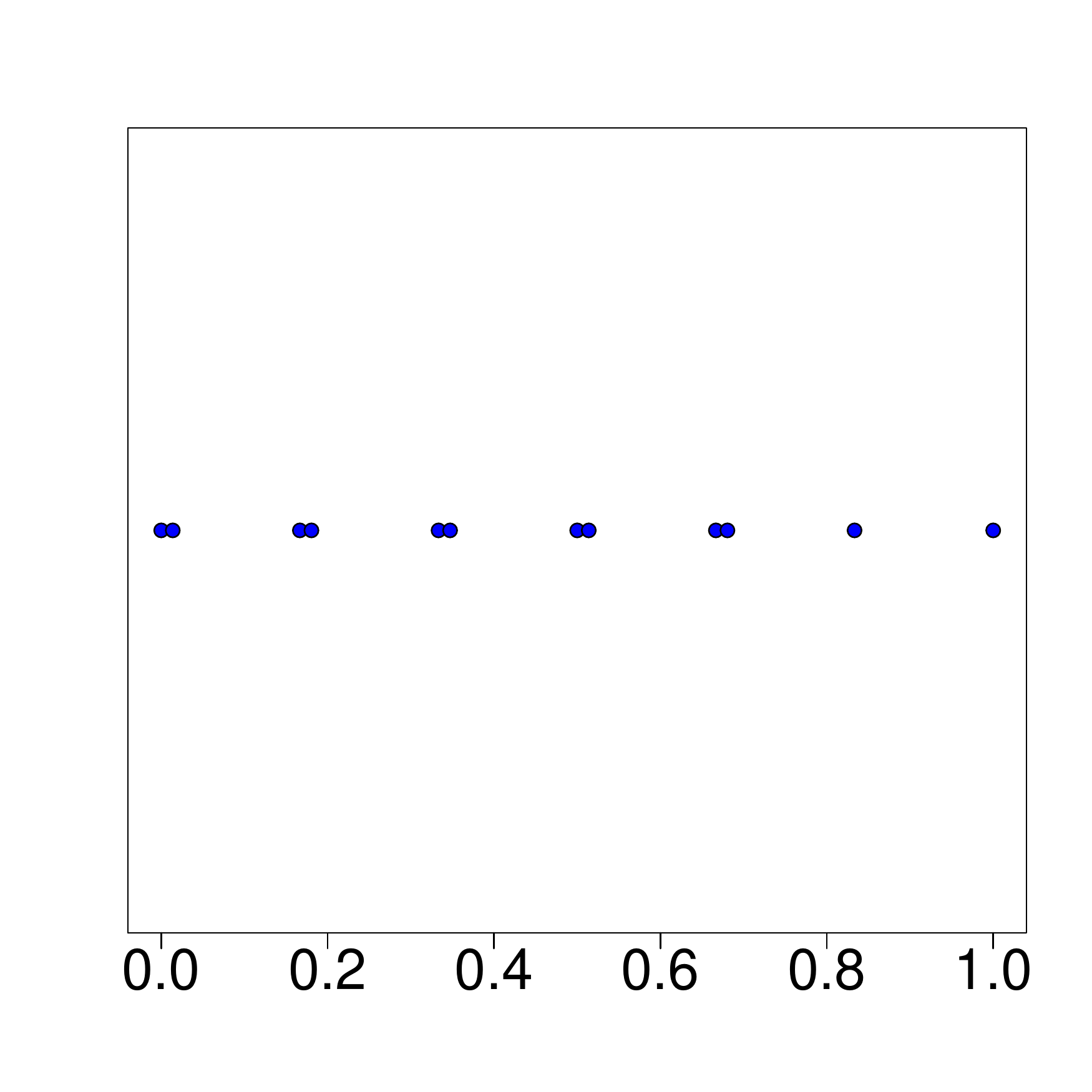} 
\end{tabular}
\caption{Plot of the points $\{s_1,...,s_{n}\}$ for the `minimal', `regular' and `maximal' designs of the numerical experiments for $n=12$ (from left to right). For the `minimal' design, nine points form a dense cluster around one and the asymptotic variance of the CV estimator is $2$ (Proposition \ref{prop2} (i)), for the `regular' design, the asymptotic variance is $3$, and for the `maximal' design, the asymptotic variance is $4$ (Proposition \ref{prop2} (ii)).}
\label{fig:designs}
\end{figure}

For a given configuration of $n$ and a given design $\{s_1,...,s_n \}$, we repeat $N = 2. 000$ data generations and estimations. That is, we independently sample $N$ Gaussian vectors of size $n$ with zero mean vector and covariance matrix $[\sigma_0^2 e^{- \theta_0 |s_i - s_j |}]_{1 \leq i,j \leq n}$. For each of these Gaussian vectors, we compute the CV estimators $\hat{\theta}$ and $\hat{\sigma}^2$, with parameter space $[0.1,10] \times [0.3 , 30]$, so that we consider case \eqref{eq:AN} of Theorem \ref{th:asymptotic}. 
The computation of $\hat{\theta}$ and $\hat{\sigma}^2$ is not challenging since, from Lemma \ref{lem:Sn}, the logarithmic score $S_n(\theta,\sigma^2)$ can be computed quickly, with a $O(n)$ complexity. [For more general covariance functions, the computation of CV or ML criteria is more costly, with a $O(n^3)$ complexity.] The criterion $S_n$ is minimized over $(\theta,\sigma^2)$ by repeating the Broyden–Fletcher–Goldfarb–Shanno (BFGS) algorithm, with several starting points for $(\theta,\sigma^2)$, and by keeping the value of $(\theta,\sigma^2)$ with smallest logarithmic score, over all the repetitions. The \url{R} software was used, with the \url{optim} function. [We remark that, for fixed $\theta$, one could find an explicit expression of $\hat{\sigma}^2(\theta) \in \argmin_{\sigma^2 >0} S_n(\theta,\sigma^2)$ (see also \cite{bachoc13cross} for a different CV criterion). Hence, it would be possible to minimize the profile logarithmic score $\min_{\sigma^2 >0} S_n(\theta,\sigma^2)$ over $\theta$ only. As mentioned earlier, this improvement is not needed here, since the criterion $S_n(\theta,\sigma^2)$ can be computed quickly.]

For the $N$ values of $(\hat{\theta},\hat{\sigma}^2)$, we compute the $N$ values of $( \sqrt{n}/ [\theta_0 \sigma_0^2] ) (\hat{\theta} \hat{\sigma}^2 - \theta_0 \sigma_0^2) $. In Figure \ref{fig:simulation}, we report the histograms of these latter $N$ values, for the seven configurations under consideration. In addition, we report the probability density functions of the seven corresponding Gaussian distributions with mean $0$ and variance $\tau_n^2$, to which the histograms converge when $n \to \infty$, in view of Theorem \ref{th:asymptotic}.

In Figure \ref{fig:simulation}, we observe that, for $n=12$, the asymptotic Gaussian distributions are already reasonable approximations of the empirical histograms. For $n=50$, the asymptotic distributions become very close to the histograms, and for $n=200$ the asymptotic distributions are almost identical to the histograms. Hence, the convergence in distribution of Theorem \ref{th:asymptotic} provides a good approximation of the finite sample situation already for small to moderate $n$. The case $n=12$ illustrates the benefit of the `minimal' design for estimation, as the histogram is most concentrated around zero for this design. Similarly, the value of $\tau_{12}^2$ is the smallest for this design, compared to the `regular' and `maximal' designs. For $n=50$ and $200$, we also observe that the estimation is more accurate for the `regular' design than for the `maximal' design, which also confirms Remark \ref{rem:designs} and Proposition \ref{prop2}.

Finally, we have obtained similar conclusions for the case where either $\theta_0$ or $\sigma_0^2$ is known in the computation of $\hat{\theta}, \hat{\sigma}^2$ (cases of \eqref{eq:AN_1} and \eqref{eq:AN_2}). We do not report the corresponding results for concision.

\begin{figure}[h!]
\centering
\begin{tabular}{c}
\includegraphics[height=4.5cm,width=5cm]{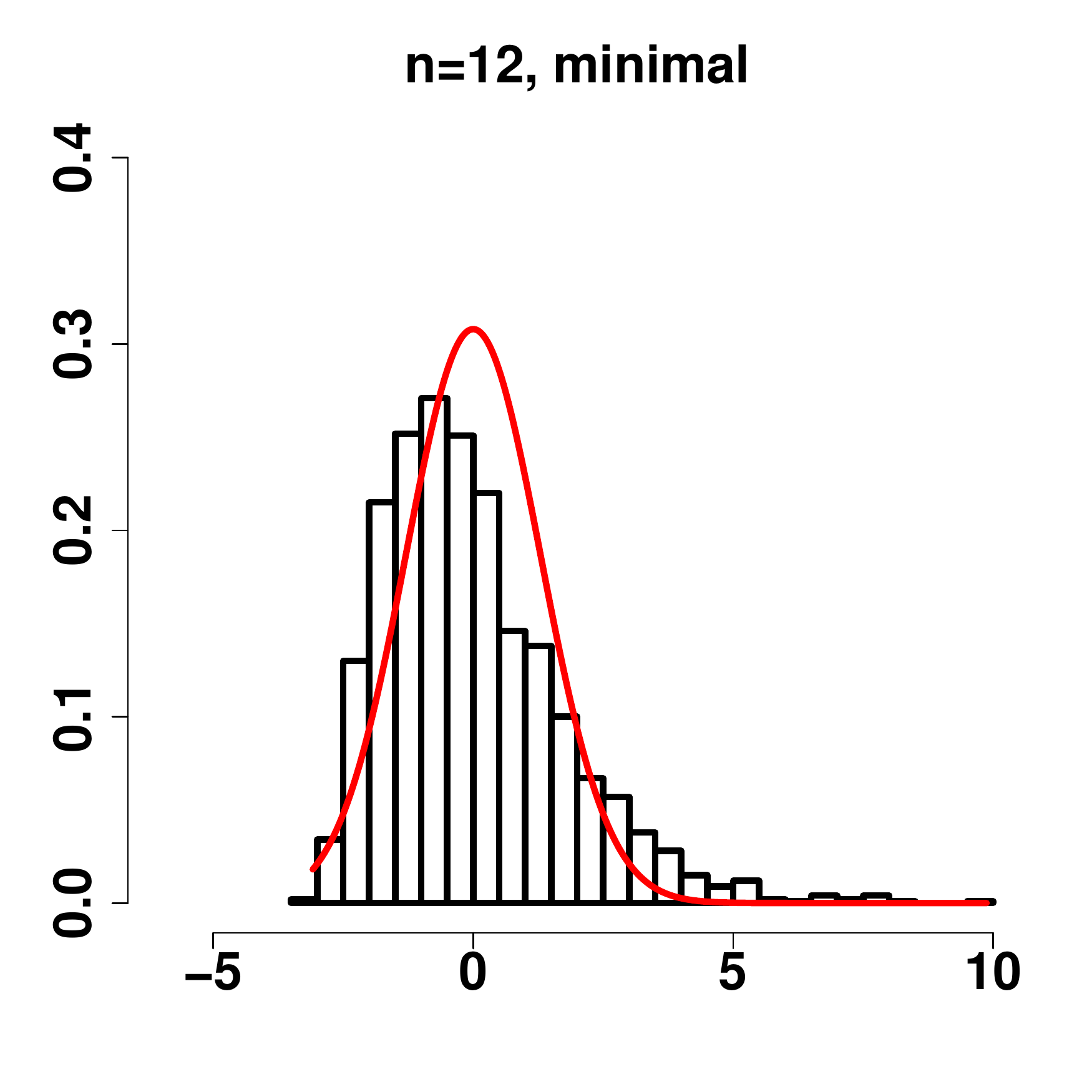}
\includegraphics[height=4.5cm,width=5cm]{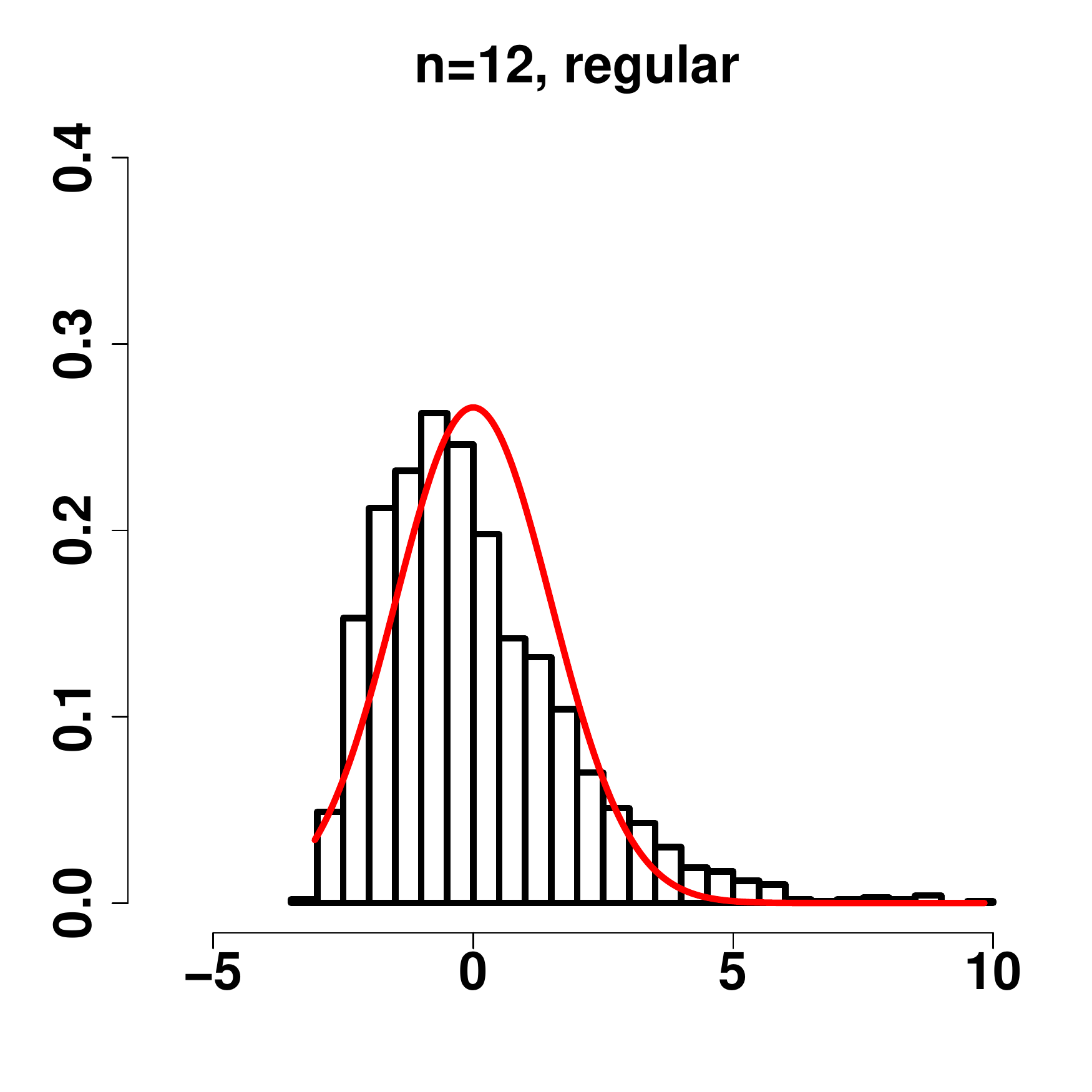}
\includegraphics[height=4.5cm,width=5cm]{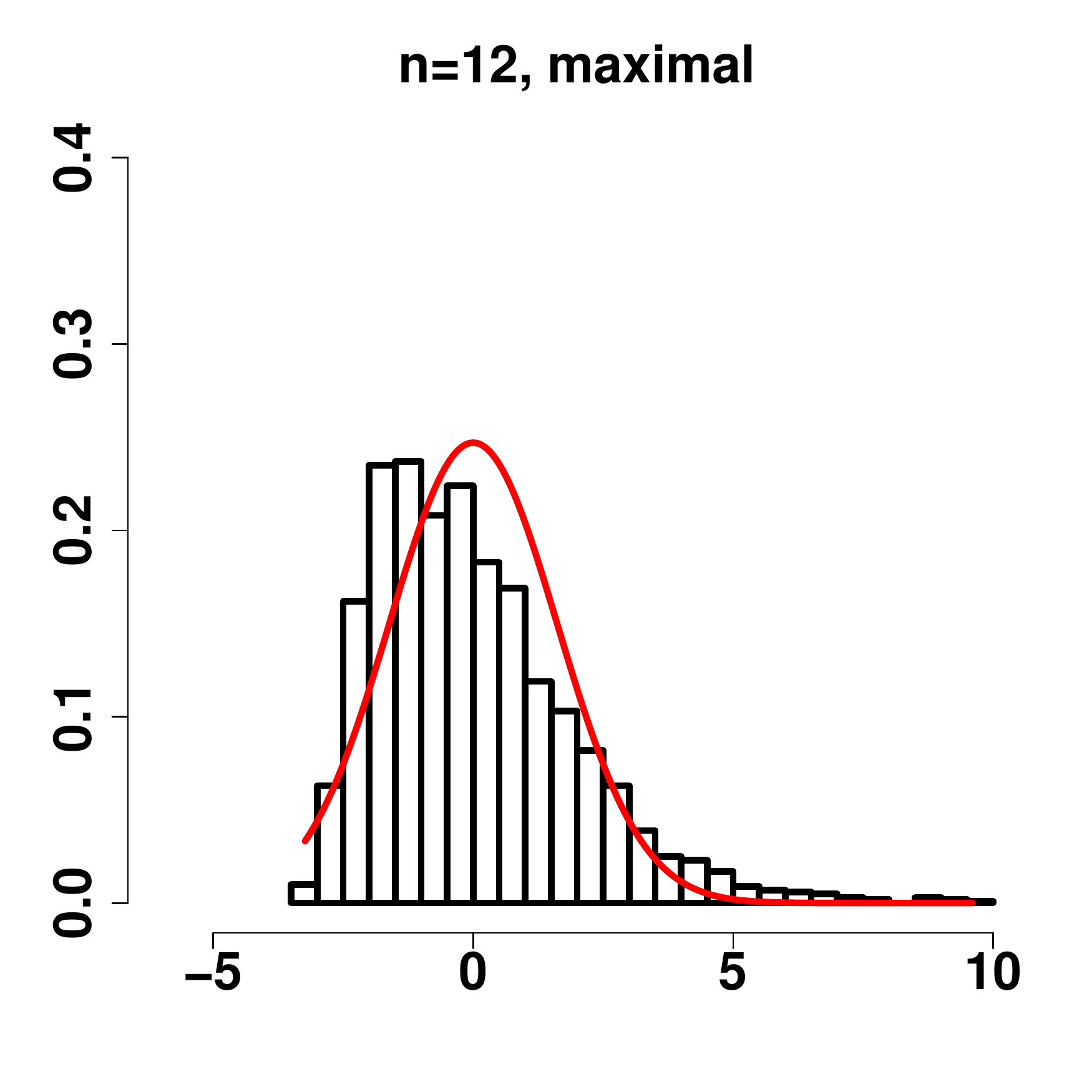} \\
\includegraphics[height=4.5cm,width=5cm]{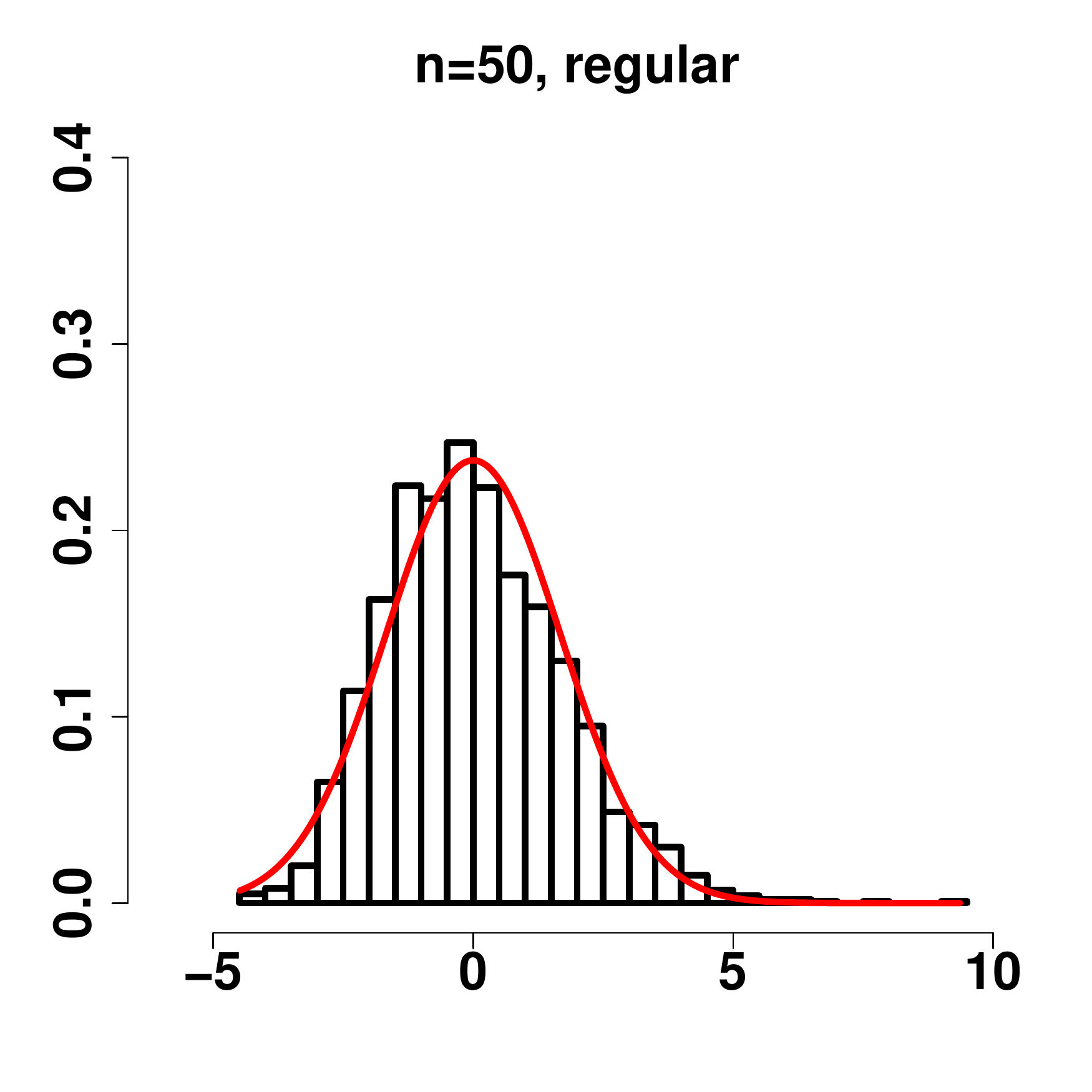}
\includegraphics[height=4.5cm,width=5cm]{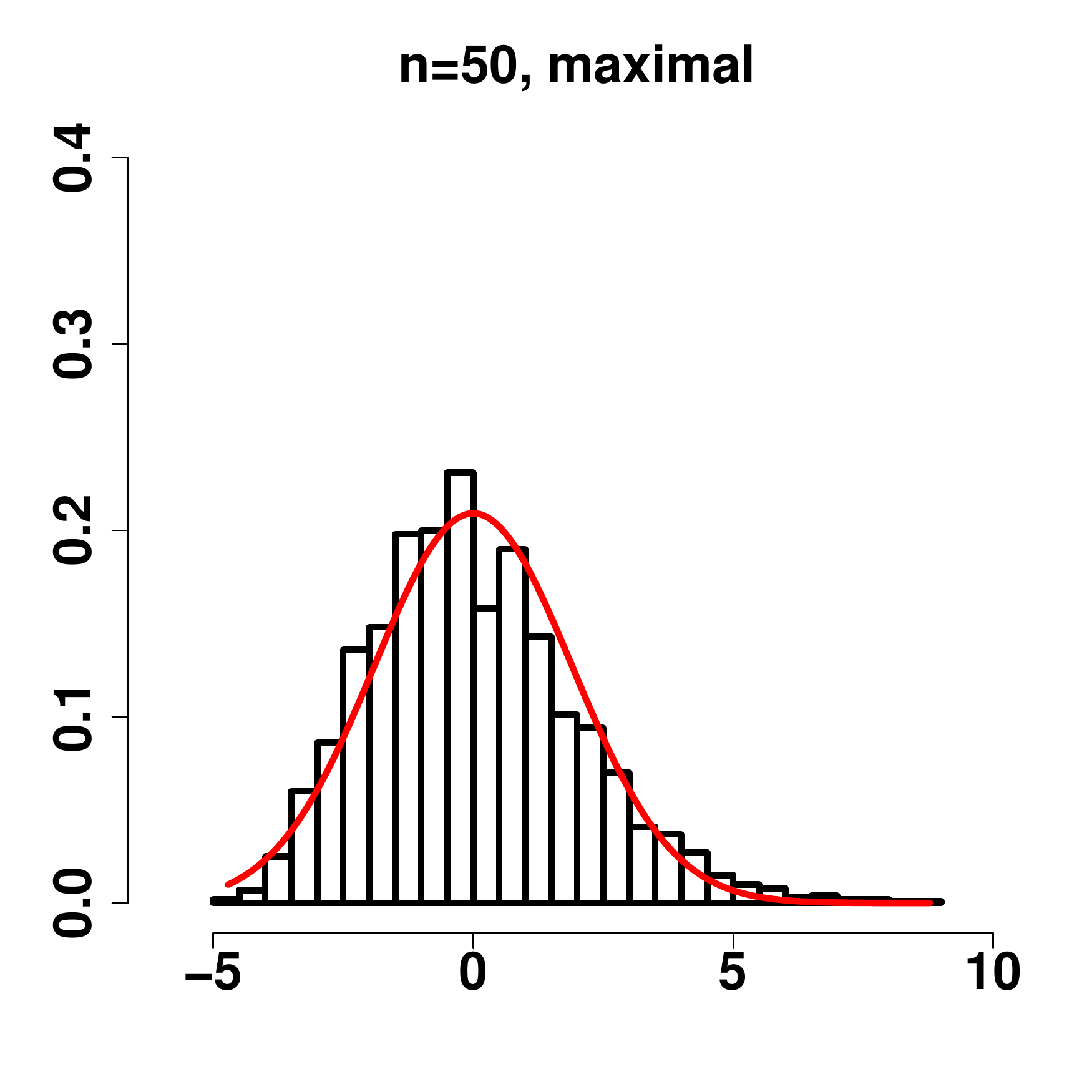} \\
\includegraphics[height=4.5cm,width=5cm]{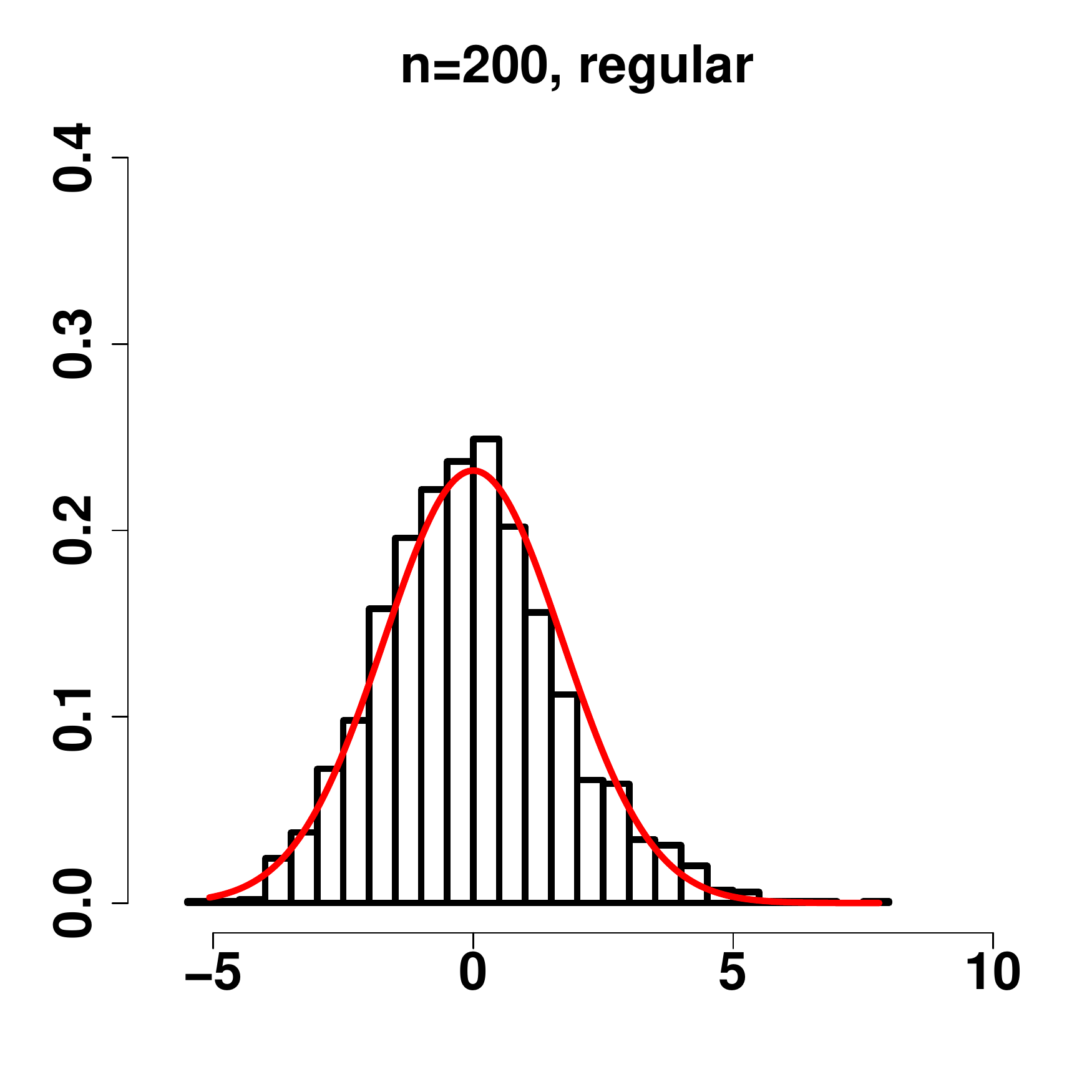} 
\includegraphics[height=4.5cm,width=5cm]{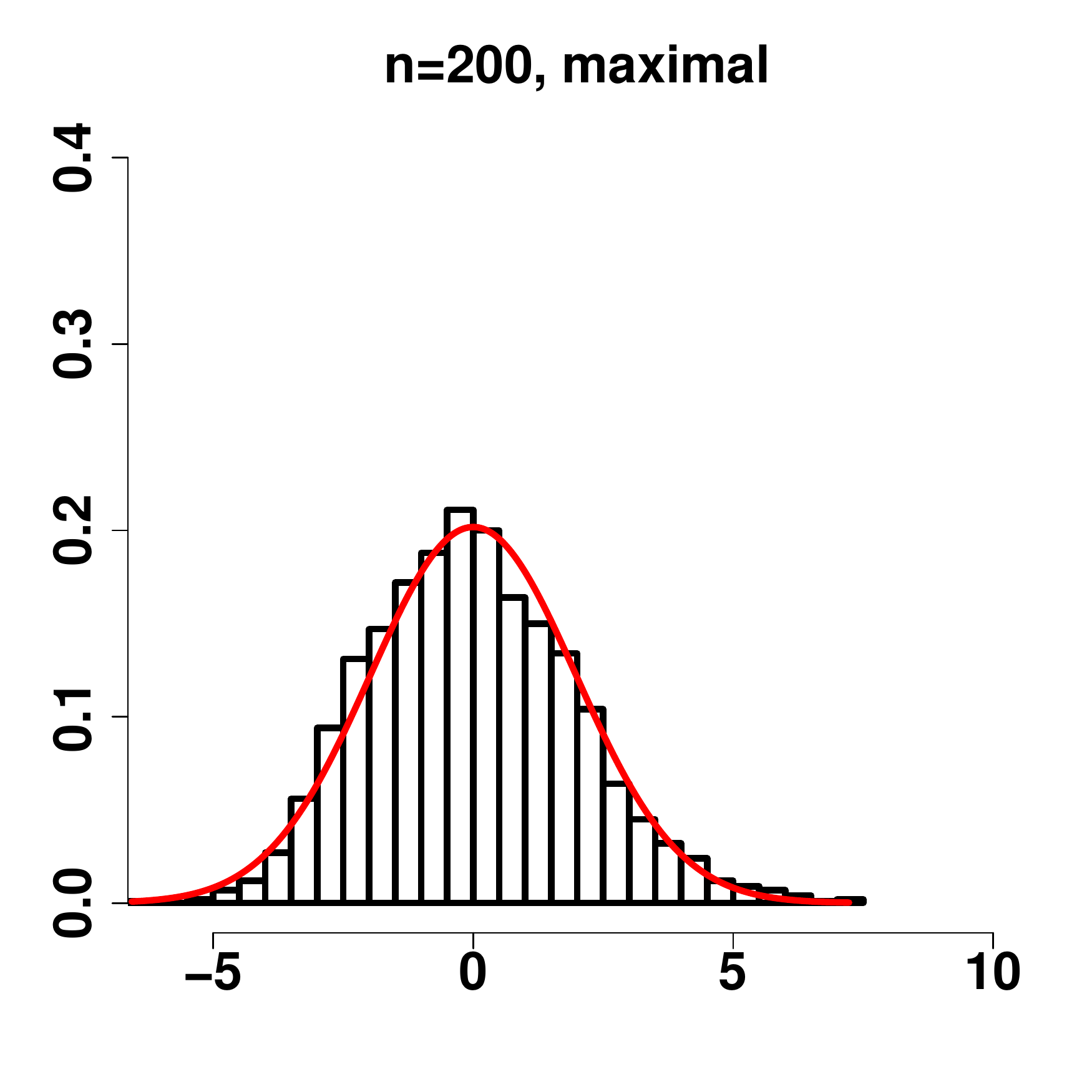} \\
\end{tabular}
\caption{Illustration of Theorem \ref{th:asymptotic}. Histograms of $N=2.000$ independent realizations of $( \sqrt{n}/ [\theta_0 \sigma_0^2]) (\hat{\theta} \hat{\sigma}^2 - \theta_0 \sigma_0^2) $, together with the corresponding asymptotic Gaussian probability density functions with mean $0$ and variances $\tau_n^2$ (red lines). The sample size is $n=12$ (top row), $n=50$ (middle row) and $n=200$ (bottom row). For the top row, the designs are the `minimal' design (left), achieving the smallest asymptotic variance; the `regular' design (middle), with equispaced observation points; and the `maximal' design (right), achieving the largest asymptotic variance. For the middle and bottom rows, the designs are the `regular' design (left) and the `maximal' design (right).}
\label{fig:simulation}
\end{figure}

\section{Extension to regression models} \label{sec:reg}

In this section, we extend Theorems \ref{th:consistency} and \ref{th:asymptotic} to the case of regression models. We assume that, instead of $Y$, we observe the Gaussian process $Z$ defined by $Z(t) = \beta_{01} f_1(t) + ... + \beta_{0p} f_p(t) + Y(t)$. In the definition of $Z$, $\bbeta_0 = (\beta_{01},...,\beta_{0p})' \in \mathbb{R}^p$ is fixed and unknown and, for $k=1,...,p$, $f_k:[0,1] \to \mathbb{R}$ is a known function. Hence, we estimate jointly $(\theta,\sigma^2,\bbeta_0')$ from the observation vector $\z = (z_1,...,z_n)'$, with $z_i = Z(s_i)$. 

Let $\F$ be the $n \times p$ matrix $[f_j(s_i)]_{i=1,...,n,j=1,...,p}$. Then $\hat{\bbeta} = ( \F' \R_{\theta}^{-1} \F)^{-1} \F' \R_{\theta}^{-1} \z$ is the best linear unbiased predictor of $\bbeta_0$ given $\z$, under covariance function $K_{\theta,\sigma^2}$ for all $\sigma^2$, see e.g. \cite{santner03design}.

We now address CV estimation.
Let $\f_i = (f_1(s_i),...,f_p(s_i))'$, let $\z_{-i} = (z_1,...,z_{i-1},z_{i+1},...,z_n)'$, let $\F_{-i}$ be the matrix obtained by removing line $i$ of $\F$, and let $\R_{\theta,-i}$ be the $(n-1) \times (n-1)$ matrix obtained by removing line and column $i$ of $\R_{\theta}$.
Then, for all $\sigma^2$, $\hat{\bbeta}_{-i} = ( \F_{-i}' \R_{\theta,-i}^{-1} \F_{-i})^{-1} \F_{-i}' \R_{\theta,-i}^{-1} \z_{-i}$ is the best linear unbiased predictor of $\bbeta_0$ given $\z_{-i}$, under covariance function $K_{\theta,\sigma^2}$.

We also let $\r_{\theta,-i} = ( K_{\theta,1}(s_i,s_1),..., K_{\theta,1}(s_i,s_{i-1}),K_{\theta,1}(s_i,s_{i+1}),..., K_{\theta,1}(s_i,s_{n}))'$.
Then, from e.g. \cite{santner03design}, 
\begin{equation} \label{eq:loo:pred:with:mean}
\hat{Z}_{\theta,-i}(s_i) = \f_i' \hat{\bbeta}_{-i}  + \r_{\theta,-i}' \R_{\theta,-i}^{-1} ( \z_{-i} - \F_{-i} \hat{\bbeta}_{-i} )
\end{equation}
is the best linear unbiased predictor of $z_i$ given $\z_{-i}$. We let 
\[
\check{\sigma}^2_{\theta,\sigma^2,-i}(s_i) = \mathbb{E}_{\theta,\sigma^2} \left( \left[ 
\hat{Z}_{\theta,-i}(s_i) - Z(s_i)
\right]^2 \right).
\]
Then, the CV estimator of $(\theta,\sigma^2)$ we shall study in this section is
\[
( \check{\theta},\check{\sigma}^2 )
\in \argmin_{a \leq \theta \leq A, b \leq \sigma^2 \leq B} \bar{S}_n(\theta,\sigma^2)
\]
with
\[
\bar{S}_n(\theta,\sigma^2)
=
\sum_{i=1}^n
\left[
 \log( \check{\sigma}^2_{ \theta,\sigma^2,-i}(s_i) ) 
 + \frac{ (z_i -\hat{Z}_{\theta,-i}(s_i) )^2 }{ \check{\sigma}^2_{ \theta,\sigma^2,-i}(s_i) }
 \right].
\]
We remark that \cite{zhang10kriging} suggests to use a similar CV criterion, with the notable difference that $\hat{\bbeta}_{-i}$ is replaced by $\hat{\bbeta}$ in \eqref{eq:loo:pred:with:mean}. The benefit of the CV predictor \eqref{eq:loo:pred:with:mean}, compared to that considered in \cite{zhang10kriging}, is that, in \eqref{eq:loo:pred:with:mean}, no use of $z_i$ is made at all for predicting $z_i$. In \cite{dubrule83cross}, the following relations are shown, extending those of Section \ref{sec:consistency}. We have
\[
\hat{Z}_{\theta,-i}(s_i)
 =
  - \sum_{\substack{j=1,\ldots,n; \\ j \neq i}} 
  \frac{
  (\Q_{\theta}^{-})_{ij}}{(\Q_{\theta}^{-})_{ii}
  }
  z_j
\]
and
\begin{equation} \label{eq:hat:sigma:pred:with:mean}
\check{\sigma}^2_{ \theta,\sigma^2,-i}(s_i)
=
\frac{\sigma^2}{(\Q_{\theta}^{-})_{ii}},
\end{equation}
with $\Q_{\theta}^{-} = \R_{\theta}^{-1} - \R_{\theta}^{-1}\F ( \F' \R_{\theta}^{-1} \F)^{-1} \F' \R_{\theta}^{-1}$.

Based on the two displays above, and again using the explicit matrix inverse in \eqref{eq:tridiagonal:inverse}, we are able to prove the consistency and asymptotic normality of $\check{\theta}\check{\sigma}^2$ where the asymptotic distribution is identical to that of Section \ref{sec:AN}.

\begin{theorem}[Consistency] \label{theorem:consistency:with:mean}
Assume that $\underset{n\to +\infty}\limsup \underset{i=2,\ldots,n}\max \Delta_i=0$ and that there exists $(\tilde{\theta},\tilde{\sigma}^2)$ in $J$ so that $\tilde{\theta} \tilde{\sigma}^2 = \theta_0 \sigma_0^2$. Assume also that $f_1,...,f_p$ are twice continuously differentiable and are linearly independent on $[0,1]$.
Then $(\check{\theta},\check{\sigma}^2)$ exists and 
\begin{align}\label{eq:cv_produit}
\check{\theta}\check{\sigma}^2 \overset{a.s.}{\to} \theta_0\sigma_0^2.
\end{align}
\end{theorem}

\begin{theorem}[Central Limit Theorem] \label{theorem:AN:with:mean}
Assume that the conditions of Theorem \ref{theorem:consistency:with:mean} hold and that $aB > \theta_0 \sigma_0^2$ and $Ab < \theta_0 \sigma_0^2$. Then, with $\tau_n$ as in \eqref{eq:tau:n}, we have
\begin{equation}
\frac{\sqrt{n}}{\theta_0\sigma_0^2\tau_{n}}( \check \theta \check \sigma^2-\theta_0 \sigma_0^2) \xrightarrow[n\to \infty]{\mathcal D} \mathcal N (0,1).\label{eq:AN:with:mean}
\end{equation}
\end{theorem}

In Theorems \ref{theorem:consistency:with:mean} and \ref{theorem:AN:with:mean}, the twice differentiability condition for $f_1,...,f_p$ is mostly technical, and could be replaced by a continuous differentiability condition, at the price of more technical proofs. [We remark that Theorems \ref{theorem:consistency:with:mean} and \ref{theorem:AN:with:mean} apply in particular to polynomial functions $f_1,...,f_p$ which are widely considered, for instance in the framework of computer experiments \cite{santner03design}.]
As remarked in \cite{Yin1991}, when $f_1,...,f_p$ are continuously differentiable, the parameter $\bbeta_0$ is non-microergodic and can not be consistently estimated.

Finally, assume now that $f_1,...,f_p$ satisfy the conditions given in Theorem 3 (ii) in \cite{Yin1991}. Then, one can show from the proof of this theorem that, for any sequence or random variables $(\bar{\theta},\bar{\sigma}^2) \in J$ (and in particular for ( $\check{\theta},\check{\sigma}^2)$), the estimator $\hat{\bbeta} = \hat{\bbeta}(\bar{\theta})$ given above is consistent and asymptotically normal, with asymptotic distribution given in (4.5) in \cite{Yin1991}. In this setting, it would be interesting to study the joint asymptotic distribution of $(\check{\theta},\check{\sigma}^2,\hat{\bbeta})$.

\section{Concluding remarks} \label{sec:ccl}

We have proved the consistency and asymptotic normality of the CV estimator of the microergodic parameter $\theta_0 \sigma_0^2$, based on the logarithmic score. 
While the ML asymptotic variance of $( \sqrt{n}/ [\theta_0 \sigma_0^2] ) (\hat{\theta}_{ML} \hat{\sigma}_{ML}^2 - \theta_0 \sigma_0^2) $ is $2$ for any triangular array of observation points, the corresponding CV asymptotic variance is simply bounded between $2$ and $4$, those bounds being tight. The triangular array we exhibit, achieving the asymptotic variance $2$ for CV, is based on some ratios between interpoint distances (of the form $(s_{i+1} - s_{i})/(s_{j+1} - s_{j})$) going to zero as $n \to \infty$, which makes it too challenging to simulate numerically for large $n$. It would be interesting to find the smallest possible asymptotic variance for CV, when the ratios between interpoint distances are bounded away from zero.

One interesting agenda for future research would be to extend this asymptotic analysis of CV in the other settings where such an analysis was possible for ML.
These settings include the case of measurement errors for the exponential covariance function in dimension one \cite{CheSimYin2000,chang2017mixed},
the case of the separable exponential covariance function in higher dimension \cite{ying93maximum} (consistency and asymptotic normality), of the separable Mat\'ern $3/2$ covariance function \cite{Loh2005} (consistency) and of the Gaussian covariance function \cite{LohLam2000} (consistency). In these references, tractable approximations of the inverse covariance matrices are provided, which could also be exploited in the case of CV. Finally, using techniques which are more spectral in nature, \cite{DuZhaMan2009,ShaKau2013,WanLoh2011} prove central limit theorems for the ML estimation of the microergodic parameter in the isotropic Mat\'ern model. An extension to CV would also be valuable.

\section{Proofs}\label{sec:proofs}

\subsection{Notation and auxiliary results}\label{ssec:aux_nota}

Remind that $\Delta_i= s_i - s_{i-1}$ and introduce $W_{i}\defeq
[y_i - e^{-\theta \Delta_i}y_{i-1}]$ and its normalized version $\overline W_{i}\defeq
[y_i - e^{-\theta \Delta_i}y_{i-1}]/[\sigma^2( 1 - e^{-2\theta \Delta_i})]^{1/2}$ for $i = 2,\ldots, n$ (the dependency in $n$ is not mentioned in $W_{i}$ and $\overline W_{i}$ to lighten notation). When $(\theta,\sigma^2)=(\theta_0,\sigma_0^2)$, the random variables will be denoted $W_{i,0}$ and $\overline W_{i,0}$. By the Markovian
and Gaussian properties of $Y$, it follows that for each $i\geq 2$, $\overline W_{i,0}$ is
independent of $\{y_j,\ j \leq i - 1\}$. Therefore, $\{\overline W_{i,0},\ 2\leq i \leq n\}$ is an i.i.d.\
sequence of random variables having the standard normal distribution $\mathcal{N}(0, 1)$. \\
It is convenient to have short expressions for terms that converge in probability to zero. We follow \cite{van2000asymptotic}.
The notation $o_{\P}(1)$ (respectively $O_{\P}(1)$) stands for a sequence of random variables that converges to zero in probability (resp. is bounded in probability) as $n \to \infty$. More generally, for a sequence of random variables $R_n$,
\beq
X_n&=&o_{\P}(R_n) \quad \textrm{means} \quad X_n=Y_nR_n \quad \textrm{with} \quad Y_n\overset{\P}{\rightarrow}0\\
X_n&=&O_{\P}(R_n) \quad \textrm{means} \quad X_n=Y_nR_n \quad \textrm{with} \quad Y_n=O_{\P}(1).
\eeq 
For deterministic sequences $X_n$ and $R_n$, the stochastic notation reduce to the usual $o$ and $O$.
Throughout the paper, by $K$, we denote a generic constant (i.e.\ $K$ may or may not be the same at different occurrences) that does not depend on $(Y,\theta,\sigma^2,n)$.

We also denote by $\delta_n$ and $\bar{\delta}_n$ two sequences of random variables satisfying
\[
\sup_{(\theta,\sigma^2) \in J} |\delta_n| = O_{\P}(1)
\]
and
\[
\sup_{(\theta,\sigma^2) \in J} | \bar{\delta}_n| = o(n), \mbox{a.s.}
\]
The definition of $\delta_n$ and $\bar{\delta}_n$ may change from one line to the other. Similarly, we denote by $\delta_{i,n}$ a triangular array of deterministic scalars satisfying 
\[
\sup_{ \substack{ n \in \N, i=1,...,n \\ (\theta,\sigma^2) \in J} } |\delta_{i,n}| \leq K.
\]
The definition of $\delta_{i,n}$ may change from one line to the other, and possibly also in different occurrences within the same line.
We also use several times that,
\begin{align}
&\sum_{i=2}^{n-1} \Delta_i =1, \label{ass:delta_sum}\\
&\underset{t\in [0,1]} \sup |Y(t)| <+\infty  ~ ~\mbox{a.s.}\label{ass:y_bounded}
\end{align} 

Before turning to the proof of Theorems \ref{th:consistency} and \ref{th:asymptotic}, we state five auxiliary lemmas that will be required in the sequel.

\begin{lemma}\label{lem:lem1}
\begin{itemize}
\item[(i)]  Let $\lambda_i>0$, $i=1,2$ be fixed. Then as $x\to 0$,
\begin{align*}
\frac{1-e^{-\lambda_1 x}}{1-e^{-\lambda_2 x}}=\frac{\lambda_1}{\lambda_2}+O(x).
\end{align*}
\item[(ii)] Let $\lambda>0$ be fixed. Then as $x,y\to 0$,
\begin{align*}
1+e^{-\lambda y}\frac{1-e^{-\lambda x}}{1-e^{-\lambda y}}= \frac{x+y}{y} (1 + O(x+y)) .
\end{align*}
\end{itemize}
\end{lemma}

\begin{lemma}\label{lem:lem2}
For any constant $\delta > 0$, there exists a constant $\eta> 0$ such that
\begin{align}\label{eq:lem2}
\underset{ \substack{ |x-1|\geq \delta \\ x>0}}{\inf} (x-1-\log x)\geq \eta.
\end{align}
\end{lemma}

\begin{lemma}\label{lem:lem3}
Suppose that for each $n$, the random variables $Z$, $Z_{1,n}$, \ldots, $Z_{n,n}$ are independent and
identically distributed and centered. Suppose also that for
some $t>0$ and $p>0$, $\E[\exp\{t |Z|^p\}] < \infty$. Then for all $\alpha>0$, a.s\,
\begin{align}
\underset{1\leq k\leq n}{\sup} |Z_{k,n}|=o(n^{\alpha})\label{eq:lem31}\\
\sum_{k=2}^n Z_{k,n}=o(n^{(1/2)+\alpha}).\label{eq:lem32}
\end{align}
\end{lemma}

The proof of Lemma \ref{lem:lem1} is direct and these of Lemmas \ref{lem:lem2} and \ref{lem:lem3} can be found in \cite{Yin1991}.\\

\begin{lemma}\label{lem:lem4}
Let for any $i \in \{ 2,...,n-1\}$,
\[
\eta_i = 
\frac{\Delta_{i+1}}{\Delta_i+\Delta_{i+1}}\left(1+ \delta_{i,n} (\Delta_i+\Delta_{i+1})\right)  
\]
and
\[
\tau_i = \frac{\Delta_i}{\Delta_i+\Delta_{i+1}}\left(1+ \delta_{i,n} (\Delta_i+\Delta_{i+1})\right).
\]

Then
\begin{itemize}
\item[(i)] $\underset{( \theta,\sigma^2) \in J}{\sup} \left|\sum_{i=2}^{n-1} \eta_i W_{i,0}y_{i} \right|= O_{\P}(1)$;

\item[(ii)] $\underset{( \theta,\sigma^2) \in J}{\sup} \left|\sum_{i=2}^{n-1} \eta_i W_{i}y_{i} \right|= O_{\P}(1)$;

\item[(iii)] $\underset{( \theta,\sigma^2) \in J}{\sup} \left|\sum_{i=2}^{n-1} \tau_i W_{i+1,0}y_{i-1} \right|= O_{\P}(1)$;

\item[(iv)] $\underset{( \theta,\sigma^2) \in J}{\sup} \left|\sum_{i=2}^{n-1} \tau_i W_{i+1}y_{i-1} \right|= O_{\P}(1)$.

 \end{itemize}
\end{lemma}

\begin{proof}[Proof of Lemma \ref{lem:lem4}] We only prove (iii) and (iv), the proof of (i) and (ii) being similar.

(iii) We have
\begin{align*}
\left| \sum_{i=2}^{n-1} \tau_i W_{i+1,0}y_{i-1} \right| & \leq  \left| \sum_{i=2}^{n-1} \frac{\Delta_i}{\Delta_i+\Delta_{i+1}} W_{i+1,0}y_{i-1} \right| + K \sum_{i=2}^{n-1} \Delta_i  \left| W_{i+1,0}y_{i-1} \right| = :  L_1+L_2.
\end{align*}
Using the fact that $\frac{\Delta_i}{\Delta_i+\Delta_{i+1}}\leq 1$ and that for $i \neq j$, 
\begin{align*}
\mathbb{E}[W_{i+1,0}y_{i-1}W_{j+1,0}y_{j-1}] =0,
\end{align*}
we get 
\begin{align*}
\mathbb{E}[ L_1^2] 
 \leq & 
 \sum_{i=2}^{n-1} \mathbb{E}[W_{i+1,0}^2y_{i-1}^2] 
=
 \sum_{i=2}^{n-1}  \mathbb{E}[y_{i-1}^2]\mathbb{E}[W_{i+1,0}^2]\\
\leq & K \sum_{i=2}^{n-1} (1- e^{-2\theta_0\Delta_{i+1}}) \leq K \sum_{i=2}^{n-1} \Delta_{i+1} 
=  O(1),
\end{align*}
using the fact that $W_{i+1,0}$ and $y_{i-1}$ are independent and a Taylor expansion of $\Delta_{i+1}\mapsto 1- e^{-2\theta_0\Delta_{i+1}}$.
Now, by Cauchy-Schwarz,
\begin{align*}
\E( |L_2| ) \leq & K   \sum_{i=2}^{n-1} \Delta_i \E ( |W_{i+1,0}||y_{i-1}| ) 
\leq
K \sum_{i=2}^{n-1} \Delta_i \sqrt{ \E( W_{i+1,0}^2 ) }
\sqrt{ \E( y_{i-1}^2 ) }
= O(1).
\end{align*}
Hence, since $L_1$ and $L_2$ do not depend on $\theta$ and $\sigma^2$, we have
$\underset{( \theta,\sigma^2) \in J}{\sup} \left|\sum_{i=2}^{n-1} \tau_i W_{i+1,0}y_{i-1}\right| = O_{\P}(1)$.\\

(iv) Now, we use the decomposition 
\begin{align}\label{eq:decomp_W}
W_{i+1}=W_{i+1,0}+(e^{-\theta_0\Delta_{i+1}}-e^{-\theta\Delta_{i+1}})y_i
\end{align} 
to get
\begin{align*}
\left| \sum_{i=2}^{n-1} \tau_i W_{i+1}y_{i-1} \right| & =
\left| \sum_{i=2}^{n-1}\tau_i W_{i+1,0}y_{i-1} + \sum_{i=2}^{n-1} \tau_i  (e^{-\theta_0\Delta_{i+1}}-e^{-\theta\Delta_{i+1}})y_i y_{i-1} \right|\\
&\leq  \left| \sum_{i=2}^{n-1}\tau_i W_{i+1,0}y_{i-1} \right| + \sum_{i=2}^{n-1} \tau_i |e^{-\theta_0\Delta_{i+1}}-e^{-\theta\Delta_{i+1}}| |y_i y_{i-1}|.
\end{align*}
The first sum is $O_{\P}(1)$ by (iii). The second sum is $O_\P(1)$ by using the fact that $\sup_{t \in [0,1]} Y^2(t) < \infty$ a.s. and  
\[
\underset{n \in \N,i=2,...,n-1,( \theta,\sigma^2) \in J}{\sup} \frac{ |\tau_i| |e^{-\theta_0\Delta_{i+1}}-e^{-\theta\Delta_{i+1}}|}{\Delta_{i+1}}  \leq K.
\]
Hence $\underset{( \theta,\sigma^2) \in J}{\sup} \left|\sum_{i=2}^{n-1} \tau_i W_{i+1}y_{i-1}\right| = O_{\P}(1)$.
\end{proof}

We can show, after some tedious but straightforward calculations, the following Taylor expansions.

\begin{lemma}\label{lem:lem5_bis}
Let
\begin{align}
\alpha_i & \defeq (\sigma^2(1-e^{- 2 \theta \Delta_i}))^{-1}, \quad\quad\quad\quad \mbox{for $i=2,...,n$,} \label{def:alpha}\\
\alpha_{i,0} & \defeq (\sigma_0^2(1-e^{- 2 \theta_0  \Delta_i}))^{-1}, \quad\quad\quad\quad \mbox{for $i=2,...,n$,} \label{def:alpha0} \\
q_i & \defeq \frac{\Delta_{i+1}}{\Delta_i+\Delta_{i+1}}+\frac{\Delta_{i-1}}{\Delta_i+\Delta_{i-1}}, \quad \mbox{for $i=3,...,n-1$,} \label{def:q}\\
A_i & \defeq \alpha_i + \alpha_{i+1} e^{-2 \theta \Delta_{i+1}}, \quad\quad\quad\quad \mbox{for $i=2,...,n-1$,} \label{def:A}\\
B_i & \defeq  \frac{1}{A_i}+\frac{e^{-2 \theta \Delta_i}}{A_{i-1}}, \quad\quad\quad\quad\quad\quad \mbox{for $i=3,...,n-1$,} \label{def:B} \\
C_i & \defeq (\alpha_i \alpha_{i+1} e^{-\theta \Delta_{i+1}})/A_i, 
\quad\quad\quad \mbox{for $i=2,...,n-1$.} \label{def:C}
\end{align}
Let $\alpha'_i =  \partial \alpha_i / \partial \theta$, $A'_i =  \partial A_i/ \partial \theta$, $B'_i =  \partial B_i / \partial \theta$ and
$C'_i =  \partial C_i / \partial \theta$. Note that we have
\begin{align}\label{def:C:prime}
C'_i = \frac{e^{-\theta \Delta_{i+1}}}{A_i} \left\{\alpha'_i\alpha_{i+1} +\alpha_i\alpha'_{i+1}-\alpha_i\alpha_{i+1}\Delta_{i+1} - \frac{\alpha_i\alpha_{i+1}A'_i}{A_i}\right\}.
\end{align}
We have, 
\begin{align}  
\underset{\substack{n \in \N, i=2,...,n, \\ (\theta,\sigma^2) \in J}}{\sup}
\left | \alpha_i - \frac{1}{2 \sigma^2 \theta \Delta_i} \right| & \leq K \label{eq:DL:alpha}\\
\underset{\substack{n \in \N, i=2,...,n, \\ (\theta,\sigma^2) \in J}}{\sup}  \left | \alpha'_i - \frac{1}{2 \sigma^2 \theta^2 \Delta_i} \right| & \leq K \label{eq:DL:alpha:prime}\\
 \alpha_i^2 = \frac{1}{4 \sigma^4 \theta^2 \Delta_i^2} + \delta_{i,n}  \frac{1}{\Delta_i} \label{eq:DL:alpha:carre}\\
\underset{\substack{n \in \N, i=2,...,n-1, \\ (\theta,\sigma^2) \in J}}{\sup}  \left | A_i -  \frac{1}{2 \sigma^2 \theta} \left( \frac{1}{\Delta_i } + \frac{1}{ \Delta_{i+1} } \right) \right| & \leq K \label{eq:DL:A}\\
\underset{\substack{n \in \N, i=2,...,n-1, \\ (\theta,\sigma^2) \in J}}{\sup} \left | A'_i - \frac{1}{2 \sigma^2 \theta^2} \left( \frac{1}{\Delta_i } + \frac{1}{ \Delta_{i+1} } \right) \right| & \leq K \label{eq:DL:A:prime}\\
 \frac{A'_i}{A_i} = \frac{1}{\theta} + \delta_{i,n}  \frac{\Delta_i \Delta_{i+1}}{\Delta_i + \Delta_{i+1}}  \label{eq:DL:ratio:A}\\
 \alpha_i B'_i = -\frac{1}{\theta} q_i + \delta_{i,n} (\Delta_{i-1} + \Delta_i + \Delta_{i+1})  \label{eq:DL:alpha:B:prime}\\
 \alpha'_i B_i = \frac{1}{\theta} q_i + \delta_{i,n} (\Delta_{i-1} + \Delta_i + \Delta_{i+1}) \label{eq:DL:alpha:prime:B}\\
\underset{\substack{n \in \N, i=2,...,n-1, \\ (\theta,\sigma^2) \in J}}{\sup}  \left | C'_i - \frac{1}{2\sigma^2\theta^2}\frac{1}{\Delta_i + \Delta_{i+1}} \right| & \leq K. \label{eq:DL:C}
\end{align}
\end{lemma}

\subsection{Proof of Theorem \ref{th:consistency}}\label{ssec:proof_cons}
The existence of $\hat{\theta}$ and $\hat{\sigma}^2$ is a consequence of the fact that $S_n$ is a continuous function defined on a compact set.
Now, it suffices to show \eqref{eq:cv_produit} since $\hat{\theta}_1\to \theta_1$ (resp. $\hat{\sigma}_2^2\to \sigma_2^2$) is a particular case of \eqref{eq:cv_produit} with $b=B=\sigma_1^2$ (resp. $a=A=\theta_2$).
Moreover, in view of \eqref{eq:min_couple}, the result \eqref{eq:cv_produit} holds if we can show that for every $\varepsilon>0$, a.s.
\begin{align}\label{eq:min_couple_2}
\underset{\substack{( \theta,\sigma^2)\in J, \\ |\theta\sigma^2-\widetilde{\theta}\widetilde{\sigma}^2|\geq \varepsilon}}{\inf} \{S_n( \theta,\sigma^2)-S_n(\widetilde{\theta},\widetilde{\sigma}^2)\}\to \infty
\end{align}
where $(\widetilde{\theta},\widetilde{\sigma}^2)\in J$ can be any non-random vector such that $\widetilde{\theta}\widetilde{\sigma}^2=\theta_0\sigma_0^2$.\\

Let us compute the difference $S_n( \theta,\sigma^2)-S_n(\widetilde{\theta},\widetilde{\sigma}^2)$ and determine the preponderant terms. After some computations, we naturally have
\begin{align} \label{eq:y:minus:haty}
y_i-\frac{\alpha_i e^{- \theta \Delta_i} y_{i-1}+ \alpha_{i+1}e^{-\theta \Delta_{i+1}} y_{i+1}}{A_i}=\frac{1}{A_i} \left( \alpha_i W_{i}-\alpha_{i+1} e^{-\theta \Delta_{i+1}}W_{i+1}\right),
\end{align}
where $\alpha_i$ and $A_i$ have already been defined in Equations \eqref{def:alpha} and \eqref{def:A}. Hence, from Lemma \ref{lem:Sn},
\begin{align}
S_n( \theta,\sigma^2)=& n \log(\sigma^2)+\log( 1-e^{-2 \theta \Delta_2} ) +  \log( 1-e^{-2 \theta \Delta_n} )-\sum_{i=2}^{n-1}\log\left(\sigma^2A_i\right)\label{eq:expression:Sn:un} \\
& + \frac{ (y_1 - e^{- \theta \Delta_2} y_2 )^2 }{\sigma^2 (1-e^{-2 \theta \Delta_2})}    + \frac{W_{n}^2 }{ \sigma^2 (1-e^{-2 \theta \Delta_n}) }\label{eq:rv_1} \\
&  +\sum_{i=2}^{n-1} \frac{1}{A_i} \left( \alpha_i W_{i}-\alpha_{i+1} e^{-\theta \Delta_{i+1}}W_{i+1}\right)^2.\label{eq:Sn}
\end{align}

In the following, we prove that the terms in \eqref{eq:rv_1} and those obtained by developing \eqref{eq:Sn}, except one, are $o(n)$ uniformly in $(\theta,\sigma^2)\in J$, a.s. More precisely, we establish the following lemma (see the proof in Section \ref{ssec:proof_lem_cons}).

\begin{lemma}\label{lem5} One has
\begin{align*}
S_n( \theta,\sigma^2)=& n \log(\sigma^2)+\log( 1-e^{-2 \theta \Delta_2} )  +  \log( 1-e^{-2 \theta \Delta_n} )
 -\sum_{i=2}^{n-1}\log\left(\sigma^2A_i\right) \\
& + \sum_{i=3}^{n-1} \alpha_i ^2 B_i W_{i,0}^2 + \bar{\delta}_{n}.
\end{align*}
\end{lemma}

\noindent\\
As a consequence, we find that,
\begin{align} \label{eq:diff:lik:theta:theta:zero}
S_n( \theta,\sigma^2)-S_n(\widetilde{\theta},\widetilde{\sigma}^2)=& n \log \frac{\sigma^2}{\widetilde{\sigma^2}}+ \log\frac{ 1-e^{-2 \theta \Delta_2} }{ 1-e^{-2 \widetilde{\theta} \Delta_2}} +  \log\frac{1-e^{-2 \theta \Delta_n}}{1-e^{-2 \widetilde{\theta} \Delta_n}} 
 -\sum_{i=2}^{n-1}\log\left(A_i/\widetilde A_i\right) \nonumber \\
& + \sum_{i=3}^{n-1} \left(\alpha_i ^2B_i-\widetilde{\alpha}_i ^2\widetilde{B}_i\right) W_{i,0}^2+ \bar{\delta}_{n} 
\end{align}
where $\widetilde{\alpha}_i$, $\widetilde{B}_i$ and $\widetilde{A}_i$ are the analogs of $\alpha_i$, $B_i$ and $A_i$ defined in Equations \eqref{def:alpha}, \eqref{def:A} and \eqref{def:B} with $\theta=\widetilde \theta$ and $\sigma^2=\widetilde \sigma^2$. More precisely, they are naturally defined by $\widetilde{\alpha}_i= (\widetilde{\sigma}^2(1-e^{- 2 \widetilde{\theta} \Delta_i}))^{-1}$, $\widetilde{A}_i= \widetilde{\alpha}_i + \widetilde{\alpha}_{i+1} e^{-\widetilde{\theta} \Delta_{i+1}}$ and $\widetilde{B}_i=\left[\frac{1}{\widetilde{A}_i}+\frac{e^{-2\widetilde{\theta} \Delta_{i}}}{\widetilde{A}_{i-1}}\right]$.\\

Using the fact that $(\widetilde{\theta},\widetilde{\sigma}^2)$ has been chosen such as $\widetilde{\theta}\widetilde{\sigma}^2=\theta_0\sigma_0^2$ and making some more computations, we get the following lemma (see the proof in Section \ref{ssec:proof_lem_cons}).

\begin{lemma}\label{lem6}
Uniformly in $(\theta,\sigma^2)\in J$, a.s.
\begin{align*}
S_n( \theta,\sigma^2)-S_n(\widetilde{\theta},\widetilde{\sigma}^2)
&= n \left[\frac{\theta_0\sigma_0^2}{\theta\sigma^2}-1-\log \frac{\theta_0\sigma_0^2}{\theta\sigma^2}\right]+o(n).
\end{align*}
\end{lemma}

Hence by Lemma \ref{lem6}, a.s.
\begin{align} \label{eq:to:finish:consistency}
\underset{\substack{( \theta,\sigma^2)\in J, \\ |\theta\sigma^2-\widetilde{\theta}\widetilde{\sigma}^2|\geq \varepsilon }}{\inf} \{S_n( \theta,\sigma^2)-S_n(\widetilde{\theta},\widetilde{\sigma}^2)\}\geq \underset{\substack{( \theta,\sigma^2)\in J, \\ |\theta\sigma^2-\widetilde{\theta}\widetilde{\sigma}^2|\geq \varepsilon }}{\inf} n\left[\frac{\theta_0\sigma_0^2}{\theta\sigma^2}-1-\log \frac{\theta_0\sigma_0^2}{\theta\sigma^2}\right] +o(n)
\end{align}
which by Lemma \ref{lem:lem2}, for every $\varepsilon>0$, is strictly positive, for $n$ large enough, a.s. Then the proof of Theorem \ref{th:consistency} is now complete.

\subsection{Proofs of the lemmas of Section \ref{ssec:proof_cons}}\label{ssec:proof_lem_cons}

\begin{proof}[Proof of Lemma \ref{lem5}]

(i) First, we study the terms in \eqref{eq:rv_1}. 
We have, from \eqref{eq:DL:alpha}
\begin{align} \label{eq:insertion:un}
\frac{\left( y_1 - e^{- \theta\Delta_2} y_2  \right)^2}{\sigma^2(1-e^{-2 \theta \Delta_2})}
& \leq K \frac{\left( y_1 - e^{- \theta\Delta_2} y_2  \right)^2}{\sigma_0^2(1-e^{-2 \theta_0 \Delta_2})} \\
& \leq K \frac{\left( y_1 - e^{- \theta_0 \Delta_2} y_2  \right)^2}{\sigma_0^2(1-e^{-2 \theta_0 \Delta_2})} 
+
K \frac{\left( e^{- \theta_0 \Delta_2} - e^{- \theta \Delta_2}   \right)^2}{\sigma_0^2(1-e^{-2 \theta_0 \Delta_2})} y_2^2 \nonumber \\
& \leq K \sup_{2 \leq i \leq n} \bar{W}_{i,0}^2 +
K \Delta_2 \sup_{t \in [0,1]} Y^2(t) \nonumber \\
&  = o(n) ~ ~ \mbox{a.s.}
 \end{align}
from Lemma \ref{lem:lem3}.
The random variable $W_{n}^2/(\sigma^2(1-e^{-2 \theta \Delta_n}))$ can be treated in the same manner leading to the same result.\\

(ii) Second, we turn to the term in \eqref{eq:Sn} that we aim at approximating by a sum of independent random variables. In this goal, we first show the relation
\begin{equation} \label{eq:from:W:to:Wzero}
W_i^2 = W_{i,0}^2 + [ e^{- \theta_0 \Delta_i} - e^{- \theta \Delta_i}]^2 y_{i-1}^2 + 2[ e^{- \theta_0 \Delta_i} - e^{- \theta \Delta_i} ] y_{i-1} W_{i,0}.
\end{equation}

 Hence, by \eqref{eq:from:W:to:Wzero}, one has 
\begin{align}
& \sum_{i=2}^{n-1} \frac{1}{A_i} \left( \alpha_i W_{i}-\alpha_{i+1} e^{-\theta \Delta_{i+1}}W_{i+1}\right)^2\nonumber\\
 = &\sum_{i=2}^{n-1} \frac{\alpha_i ^2}{A_i} W_{i}^2 + \sum_{i=2}^{n-1} \frac{\alpha_{i+1} ^2e^{-2\theta \Delta_{i+1}}}{A_i}W_{i+1}^2 - 2 \sum_{i=2}^{n-1} \frac{\alpha_i \alpha_{i+1} e^{-\theta \Delta_{i+1}}}{A_i} W_{i}W_{i+1}\nonumber\\
= &\sum_{i=3}^{n-1} \alpha_i ^2\left[\frac{1}{A_i}+\frac{e^{-2\theta \Delta_{i}}}{A_{i-1}}\right] W_{i}^2 + \frac{\alpha_2^2}{A_2}W_{2}^2 
+ \frac{\alpha_n^2e^{-2\theta\Delta_n}}{A_{n-1}}W_{n}^2 - 2 \sum_{i=2}^{n-1} C_i W_{i}W_{i+1}\nonumber\\
 = &\sum_{i=3}^{n-1} \alpha_i ^2B_i W_{i,0}^2 + \sum_{i=3}^{n-1} \alpha_i ^2B_i (e^{-\theta_0\Delta_i}-e^{-\theta\Delta_i})^2y_{i-1}^2 + 2 \sum_{i=3}^{n-1} \alpha_i ^2B_i (e^{-\theta_0\Delta_i}-e^{-\theta\Delta_i})W_{i,0}y_{i-1}\nonumber\\
& + \frac{\alpha_2^2}{A_2}W_{2}^2 + \frac{\alpha_n^2e^{-2\theta\Delta_n}}{A_{n-1}}W_{n}^2 -  2 \sum_{i=2}^{n-1} C_i W_{i}W_{i+1} \nonumber\\
\eqdef &\Sigma_1+\Sigma_2+2\Sigma_3 + T_2+T_n-2\Sigma_4, \label{eq:decomp_derniere_somme}
\end{align}

where $\alpha_i$, $A_i$, $B_i$ and $C_i$ have been defined in \eqref{def:alpha}, \eqref{def:A}, \eqref{def:B} and \eqref{def:C}.
We prove that all the previous terms are $o(n)$, uniformly in $\theta$ and $\sigma^2$ a.s, except $\Sigma_1$ that still appears in the expression of $S_n(\theta,\sigma^2)$ in Lemma \ref{lem5}.\\

$\bullet$ \textbf{Term $T_2$}:
For $n$ large enough, since $ \frac{\alpha_2 }{A_2}\leq 1$, we get
\begin{eqnarray*}
|T_2| & = & \left| \frac{\alpha_2^2}{A_2} W_{2}^2 \right| 
 \leq  \alpha_2 W_{2}^2.
\end{eqnarray*}
Hence, we can show
$\underset{( \theta,\sigma^2) \in J}{\sup} |T_2| = o(n)$ a.s. in the same way as for \eqref{eq:insertion:un}.\\

$\bullet$ \textbf{Term $T_n$}: 
For $n$ large enough, since $ \frac{\alpha_n e^{-2 \theta \Delta_n}}{A_{n-1}}\leq 1$, we get
\begin{eqnarray*}
	|T_n| & = & \left| \frac{\alpha_n^2 e^{-2 \theta \Delta_n}}{A_{n-1}} W_{n}^2 \right| 
	\leq  \alpha_n W_{n}^2.
\end{eqnarray*}
Hence, we can show
$\underset{( \theta,\sigma^2) \in J}{\sup} |T_n| = o(n)$ a.s. in the same way as for \eqref{eq:insertion:un}.\\

$\bullet$ \textbf{Term $\Sigma_2$}: The deterministic quantity
$\frac{\alpha_i}{\Delta_i}(e^{-\theta_0\Delta_i}-e^{-\theta\Delta_i})^2$ is bounded for $n$ large enough, uniformly in $(\theta,\sigma^2) \in J$ (trivial inequalities and \eqref{eq:DL:alpha}) while $\frac{\alpha_i}{A_i}\leq 1$ and $\frac{\alpha_i e^{-2\theta\Delta_i}}{A_{i-1}}\leq 1$. Then, we are led to
\begin{align*}
\sum_{i=3}^{n-1} \alpha_i ^2B_i (e^{-\theta_0\Delta_i}-e^{-\theta\Delta_i})^2y_{i-1}^2 &=\sum_{i=3}^{n-1} \Delta_i\alpha_iB_i \frac{\alpha_i}{\Delta_i}(e^{-\theta_0\Delta_i}-e^{-\theta\Delta_i})^2y_{i-1}^2\\
&\leq K \underset{t\in[0,1]}\sup Y(t)^2 \sum_{i=3}^{n-1} \Delta_i =K \underset{t\in[0,1]}\sup Y(t)^2
\end{align*}
from which, by \eqref{ass:y_bounded}, we deduce
$\underset{( \theta,\sigma^2) \in J}{\sup}\Sigma_2 = o(n)$ a.s.\\

$\bullet$ \textbf{Term $\Sigma_3$}: By the Cauchy-Schwarz inequality,
\begin{align} \label{eq:Sigma:trois}
|\Sigma_3|
& \leq \left(\sum_{i=3}^{n-1} \frac{\alpha_i ^4}{\alpha_{i,0}}B_i^2 (e^{-\theta_0\Delta_i}-e^{-\theta\Delta_i})^2\right)^{1/2} \left(\sum_{i=3}^{n-1} y_{i-1}^2 \overline W_{i,0}^2 \right)^{1/2}.
\end{align}
As already mentioned, the deterministic term
\[
\frac{\alpha_i }{\Delta_i} (e^{-\theta_0\Delta_i}-e^{-\theta\Delta_i})^2
\]
is bounded uniformly in $(\theta,\sigma^2) \in J$. Furthermore, $\alpha_i / \alpha_{i,0}$ is bounded uniformly in $(\theta,\sigma^2) \in J$ from \eqref{eq:DL:alpha}. Finally, $\alpha_i^2 B_i^2 \leq K \alpha_i^2 ( 1/\alpha_i )^2 = K$. Hence the first term on the right-hand side of \eqref{eq:Sigma:trois} is bounded uniformly in $(\theta,\sigma^2) \in J$. 
Now Lemma \ref{lem:lem3} yields that $\sum_{i=3}^{n-1}(\overline W_{i,0}^2 - 1) = o(n^{(1/2)+\alpha})$ a.s. for any $\alpha> 0$ leading to $\sum_{i=3}^{n-1}\overline W_{i,0}^2=O(n)$ a.s\, and
\[
\left(\sum_{i=3}^{n-1} y_{i-1}^2 \overline W_{i,0}^2\right)^{1/2}\leq \left(\underset{t\in [0,1]}\sup Y^2(t)\sum_{i=3}^{n-1} \overline W_{i,0}^2\right)^{1/2} =O(n^{1/2})   \ a.s.
\]
As a consequence, $ \underset{( \theta,\sigma^2) \in J}{\sup} | \Sigma_3 | =O(n^{1/2})$ a.s.\ and naturally $\underset{( \theta,\sigma^2) \in J}{\sup} |\Sigma_3 | =o(n)$ a.s.\\

$\bullet$ \textbf{Term $\Sigma_4$}:  Using the trivial equality $ab=ab-a_0b_0+a_0b_0=a_0b_0+a_0(b-b_0)+b(a-a_0)$, one gets 
\begin{align}
& W_{i}W_{i+1}  =  [y_i-e^{- \theta \Delta_i} y_{i-1}][y_{i+1}-e^{-\theta \Delta_{i+1}} y_i] =  [y_i-e^{- \theta_0 \Delta_i} y_{i-1}][y_{i+1}-e^{- \theta_0 \Delta_{i+1}} y_{i}]\nonumber \\
& + [y_i-e^{- \theta_0 \Delta_i} y_{i-1}](e^{- \theta_0 \Delta_{i+1}} -e^{-\theta \Delta_{i+1}} )y_i  +[y_{i+1}-e^{-\theta \Delta_{i+1}} y_i] (e^{ -\theta_0 \Delta_{i}} -e^{-\theta \Delta_{i}} )y_{i-1}\nonumber\\
& = W_{i,0}W_{i+1,0} + (e^{- \theta_0 \Delta_{i+1}} -e^{-\theta \Delta_{i+1}} ) W_{i,0} y_i 
 + (e^{ -\theta_0 \Delta_{i}} -e^{-\theta \Delta_{i}} )W_{i+1} y_{i-1}.\label{eq:decomp_W}
\end{align}
Thus $\Sigma_4$ rewrites

\begin{equation} \label{eq:Sigma:quatre}
\sum_{i=2}^{n-1} C_i W_{i,0} W_{i+1,0}
+ \sum_{i=2}^{n-1} C_i(e^{- \theta_0 \Delta_{i+1}} -e^{-\theta \Delta_{i+1}} )W_{i,0} y_i
+\sum_{i=2}^{n-1} C_i(e^{ -\theta_0 \Delta_{i}} -e^{-\theta \Delta_{i}} )W_{i+1}y_{i-1}.
\end{equation}

We can show that
\[
\frac{C_i}{\alpha_{i,0}^{1/2} \alpha_{i+1,0}^{1/2}}
= 
\frac{\sigma_0^2 \theta_0}{\sigma^2 \theta^2}
\frac{\sqrt{ \Delta_i \Delta_{i+1} }}{\Delta_i + \Delta_{i+1}}
+ \delta_{i,n}
( \Delta_i + \Delta_{i+1} ).
\]
Hence the first random variable of \eqref{eq:Sigma:quatre} rewrites
\begin{align*}
& \underset{(\theta,\sigma^2) \in J}\sup
\left|
\sum_{i=2}^{n-1} C_i W_{i,0} W_{i+1,0}
\right|
 = \underset{(\theta,\sigma^2) \in J}\sup
\left|
\sum_{i=2}^{n-1} \frac{C_i}{\alpha_{i,0}^{1/2} \alpha_{i+1,0}^{1/2}} \overline W_{i,0} \overline W_{i+1,0}
\right|\\
& \leq 
\underset{(\theta,\sigma^2) \in J}\sup
\left|
\sum_{i=2}^{n-1} \frac{\sigma_0^2 \theta_0}{\sigma^2 \theta^2}
\frac{\sqrt{ \Delta_i \Delta_{i+1} }}{\Delta_i + \Delta_{i+1}} \overline W_{i,0} \overline W_{i+1,0}
\right|
 + \sum_{i=2}^{n-1} \delta_{i,n} ( \Delta_i + \Delta_{i+1} ) |\overline W_{i,0} \overline W_{i+1,0}| \\
& \leq K \left| \sum_{i=2}^{n-1} 
\frac{\sqrt{ \Delta_i \Delta_{i+1} }}{\Delta_i + \Delta_{i+1}} \overline W_{i,0} \overline W_{i+1,0} \right|
+\sum_{i=2}^{n-1} \delta_{i,n} ( \Delta_i + \Delta_{i+1} ) |\overline W_{i,0} \overline W_{i+1,0}| \\
& = K \left| \sum_{i=2}^{n-1} 
\frac{\sqrt{ \Delta_i \Delta_{i+1} }}{\Delta_i + \Delta_{i+1}} \overline W_{i,0} \overline W_{i+1,0} \right|
+ \bar{\delta}_n,
\end{align*}
from \eqref{eq:lem31}.
Now, we have 
\begin{align*}
& \left| \sum_{i=2}^{n-1} 
\frac{\sqrt{ \Delta_i \Delta_{i+1} }}{\Delta_i + \Delta_{i+1}} \overline W_{i,0} \overline W_{i+1,0} \right| \\
& \leq 
\left| \sum_{\substack{i=2,...,n-1\\ \mbox{$i$ odd}}} 
\frac{\sqrt{ \Delta_i \Delta_{i+1} }}{\Delta_i + \Delta_{i+1}} \overline W_{i,0} \overline W_{i+1,0} \right|
+
\left| \sum_{\substack{i=2,...,n-1\\ \mbox{$i$ even}}} 
\frac{\sqrt{ \Delta_i \Delta_{i+1} }}{\Delta_i + \Delta_{i+1}} \overline W_{i,0} \overline W_{i+1,0} \right|.
\end{align*}
In each of the two sums above, the summands constitute two triangular arrays of independent random variables. Thus, applying Theorem 2.1 in \cite{hu97strong} with $a_n = n$, each of the two sums is a $o(n)$ a.s.
Hence finally $ \underset{(\theta,\sigma^2) \in J}\sup \left| \sum_{i=2}^{n-1} C_i W_{i,0} W_{i+1,0} \right| = o(n)$ a.s.

Let us now address the second term in \eqref{eq:Sigma:quatre} that by the Cauchy-Schwarz inequality is bounded from above by
\begin{align*}
 \left(\sum_{i=2}^{n-1} \frac{C_i^2}{\alpha_{i,0}} (e^{- \theta_0 \Delta_{i+1}} -e^{-\theta \Delta_{i+1}} )^2 \right)^{1/2} 
\left(\sum_{i=2}^{n-1} y_i^2 \overline W_{i,0}^2 \right)^{1/2}
= o(n) \quad \textrm{a.s.}
\end{align*}
where the last equality comes from similar computations as from the term $\Sigma_3$ above, and from the fact that
\[
\sup_{n \in \mathbb{N}} \sup_{2 \leq i \leq n-1} 
\sup_{(\theta,\sigma^2) \in J} 
\left|
\frac{C_i^2 }{ \alpha_{i,0} \Delta_{i+1} } ( e^{- \theta_0 \Delta_{i+1}} - e^{- \theta \Delta_{i+1} }  )^2
\right|
\leq K.
\]

 The third term in the right-hand side of \eqref{eq:Sigma:quatre} yields  
\begin{align} \label{eq:insertion:trois}
& \sup_{(\theta , \sigma^2) \in J} \left| \sum_{i=2}^{n-1} C_i(e^{ -\theta_0 \Delta_{i}} -e^{-\theta \Delta_{i}} )W_{i+1}y_{i-1} \right |  \leq \nonumber  \\
&  \sup_{(\theta , \sigma^2) \in J} \left| \sum_{i=2}^{n-1} C_i(e^{ -\theta_0 \Delta_{i}} -e^{-\theta \Delta_{i}} )W_{i+1,0}y_{i-1} \right |
+
\sup_{(\theta , \sigma^2) \in J} \left| \sum_{i=2}^{n-1} C_i(e^{ -\theta_0 \Delta_{i}} -e^{-\theta \Delta_{i}} )^2y_i y_{i-1} \right |.
\end{align}
Since trivially $C_i(e^{ -\theta_0 \Delta_{i}} -e^{-\theta \Delta_{i}} ) = \delta_{i,n}$, the second term in \eqref{eq:insertion:trois} is bounded by $K \sum_{i=2}^{n-1} \Delta_i |y_i y_{i-1}| = O(1)$ a.s.

The first term in \eqref{eq:insertion:trois} is bounded by 
\begin{align*}
K \sum_{i=2}^{n-1} | W_{i+1,0} y_{i-1} |
& \leq K \sup_{t \in [0,1]} |Y(t)| \sum_{i=2}^{n-1} | \bar{W}_{i+1,0} | \alpha_{i,0}^{-1/2}  \\
& \leq K \sup_{t \in [0,1]} |Y(t)| 
\sqrt{ \sum_{i=2}^{n-1}  \bar{W}_{i+1,0}^2  }
\sqrt{ \sum_{i=2}^{n-1}  \alpha_{i,0}^{-1} } \\
& = O( \sqrt{n} ) ~ ~ \mbox{a.s.}
\end{align*}
since $\sum_{i=2}^{n-1}  \bar{W}_{i+1,0}^2 = O(n)$ a.s. has been shown when handling $\Sigma_3$ after \eqref{eq:Sigma:trois}. 
 
One may conclude that $\underset{( \theta,\sigma^2) \in J} \sup \Sigma_4=o(n)$ a.s. The proof of Lemma \ref{lem5} is then complete.
\end{proof}

\begin{proof}[Proof of Lemma \ref{lem6}]

We address each of the terms in \eqref{eq:diff:lik:theta:theta:zero}.

$\bullet$ Using Lemma \ref{lem:lem1}(i), we get that
\[
\sup_{\theta \in [a, A]} \left| \log\frac{ 1-e^{-2 \theta \Delta_2} }{ 1-e^{-2 \widetilde{\theta} \Delta_2}} \right| = O(1).
\]
In the same way, $\sup_{\theta \in[a, A]} \left| \log\frac{1-e^{-2 \theta \Delta_n}}{1-e^{-2 \widetilde{\theta} \Delta_n}} \right| =O(1)$.\\

$\bullet$ We have, using Lemmas \ref{lem:lem1} (i) and (ii),

\begin{align*}
\log(A_i/\widetilde{A}_i)
& = \log \left[ \left(\frac{1}{1-e^{- 2\theta \Delta_i}}+\frac{e^{-2\theta \Delta_{i+1}}}{1-e^{- 2 \theta \Delta_{i+1}}}\right)/\left(\frac{1}{1-e^{- 2\widetilde{\theta} \Delta_i}}+\frac{e^{-2\widetilde{\theta} \Delta_{i+1}}}{1-e^{- 2 \widetilde{\theta} \Delta_{i+1}}}\right) \right] \\
& = \log\frac{1-e^{- 2\widetilde{\theta} \Delta_i}}{1-e^{- 2\theta \Delta_i}}
+\log \left[ \left(1+e^{-2\theta \Delta_{i+1}}\frac{1-e^{- 2\theta \Delta_i}}{1-e^{- 2 \theta \Delta_{i+1}}}\right)/ \left(1+e^{-2\widetilde{\theta} \Delta_{i+1}}\frac{1-e^{- 2\widetilde{\theta} \Delta_i}}{1-e^{- 2 \widetilde{\theta}\Delta_{i+1}}}\right) \right] \\
& = \log \frac{\widetilde\theta}{\theta}  + \delta_{i,n} (\Delta_i + \Delta_{i+1}).
\end{align*}
Thus, by summation we have,
\[
\underset{\theta \in [a,A]}\sup \left|
\sum_{i=2}^{n-1} \log  \frac{A_i}{\tilde{A}_i}  - n \log  \frac{\widetilde\theta}{\theta} \right| =  O(1) = o(n). 
\]

$\bullet$ We want to show that 
\[
\underset{(\theta,\sigma^2) \in J}\sup \left| \sum_{i=3}^{n-1} \alpha_i ^2B_i W_{i,0}^2 - n \frac{\theta_0\sigma_0^2}{\theta\sigma^2} \right| = o(n) ~ ~ \mbox{a.s}.
\]
By \eqref{def:B}, one has
\begin{align*}
&\sum_{i=3}^{n-1} \alpha_i ^2B_iW_{i,0}^2
=\sum_{i=3}^{n-1} \frac{\alpha_i ^2}{\alpha_{i,0}}\frac{1}{A_i}\overline W_{i,0}^2+\sum_{i=3}^{n-1} \frac{\alpha_i ^2}{\alpha_{i,0}}\frac{e^{-2\theta \Delta_{i}}}{A_{i-1}}\overline W_{i,0}^2\\
=&\sum_{i=3}^{n-2} \left(\frac{\alpha_i ^2}{\alpha_{i,0}}\frac{1}{A_i}\overline W_{i,0}^2+\frac{\alpha_{i+1} ^2}{\alpha_{i+1,0}}\frac{e^{-2\theta \Delta_{i+1}}}{A_i}\overline W_{i+1,0}^2\right)
+\frac{\alpha_{3} ^2e^{-2\theta \Delta_{3}}}{A_2} W_{3,0}^2
+ \frac{\alpha_{n-1} ^2 }{A_{n-1}}W_{n-1,0}^2.\\
\end{align*}
Then we use \eqref{eq:DL:alpha} to develop $\alpha_i/\alpha_{i,0}$ (respectively $\alpha_{i+1}/\alpha_{i+1,0}$) and Lemma \ref{lem:lem1} (ii) to develop $\alpha_i/A_i$ (respectively $\alpha_{i+1}e^{-2\theta \Delta_{i+1}}/A_i$). We get
\begin{equation}  \label{eq:for:cons:smaller:K:Delta:i}
  \frac{\alpha_i ^2}{\alpha_{i,0}}\frac{1}{A_i} 
  =
   \frac{\theta_0\sigma_0^2}{\theta\sigma^2}  \frac{\Delta_{i+1}}{\Delta_i+\Delta_{i+1}} 
+ \delta_{i,n} (\Delta_i  + \Delta_{i+1})
 \end{equation}
 \begin{equation}  \label{eq:for:cons:smaller:K:Delta:ip1}
 \frac{\alpha_{i+1} ^2}{\alpha_{i+1,0}}\frac{e^{-2\theta \Delta_{i+1}}}{A_i} 
 =
  \frac{\theta_0\sigma_0^2}{\theta\sigma^2}  \frac{\Delta_{i}}{\Delta_i+\Delta_{i+1}} 
+ \delta_{i,n} (\Delta_i  + \Delta_{i+1}).
 \end{equation}
In addition, we easily show, as in \eqref{eq:insertion:un}, that $\underset{(\theta,\sigma^2) \in J}\sup \left| \frac{\alpha_{3} ^2e^{-2\theta \Delta_{3}}}{A_2} W_{3,0}^2 \right| = o(n)$ a.s. and that
$\underset{(\theta,\sigma^2) \in J}\sup \linebreak[1] \left|  \frac{\alpha_{n-1} ^2 }{A_{n-1}}W_{n-1,0}^2 \right| = o(n)$ a.s.
Then,
\begin{align} \label{eq:for:lemma:six:p:sept}
 & \left| \sum_{i=3}^{n-1} \alpha_i ^2 B_i  W_{i,0}^2 -n \frac{\theta_0\sigma_0^2}{\theta\sigma^2} \right|
= 
\left| \sum_{i=3}^{n-2}
\left(
\frac{\alpha_i ^2}{\alpha_{i,0}}\frac{1}{A_i} \overline W_{i,0}^2 + 
\frac{\alpha_{i+1} ^2}{\alpha_{i+1,0}}\frac{e^{-2\theta \Delta_{i+1}}}{A_i} \overline W_{i+1,0}^2
\right)
 -\sum_{i=3}^{n-2} \frac{\theta_0\sigma_0^2}{\theta\sigma^2}
 + \bar{\delta}_{n}  \right| \nonumber
 \\
 = & 
\left|
\sum_{i=3}^{n-2}
\left(
\frac{\alpha_i ^2}{\alpha_{i,0}}\frac{1}{A_i} \overline W_{i,0}^2+ 
\frac{\alpha_{i+1} ^2}{\alpha_{i+1,0}}\frac{e^{-2\theta \Delta_{i+1}}}{A_i} \overline W_{i+1,0}^2
\right)
 -\frac{\theta_0\sigma_0^2}{\theta\sigma^2} \sum_{i=3}^{n-2} \left(\frac{\Delta_{i+1}}{\Delta_i+\Delta_{i+1}}+\frac{\Delta_{i}}{\Delta_i+\Delta_{i+1}}\right)
  + \bar{\delta}_{n}  \right| \nonumber \\
  \leq & \sum_{i=3}^{n-2}
\left|
\frac{\alpha_i ^2}{\alpha_{i,0}}\frac{1}{A_i} 
-
\frac{\theta_0 \sigma_0^2}{\theta^2 \sigma^2}
\frac{\Delta_{i+1}}{\Delta_i+\Delta_{i+1}}
\right|
\overline  W_{i,0}^2
+ 
\sum_{i=3}^{n-2}
\left|
\frac{\alpha_{i+1} ^2}{\alpha_{i+1,0}}\frac{e^{-2\theta \Delta_{i+1}}}{A_i}
-
\frac{\theta_0 \sigma_0^2}{\theta^2 \sigma^2}
\frac{\Delta_{i}}{\Delta_i+\Delta_{i+1}}
\right|
 \overline W_{i+1,0}^2 \nonumber \\
&
+ \frac{\theta_0\sigma_0^2}{\theta\sigma^2} \left|  \sum_{i=3}^{n-2}
\frac{\Delta_{i+1}}{\Delta_i+\Delta_{i+1}}
(  \overline W_{i,0}^2 - 1  ) \right|
+ \frac{\theta_0\sigma_0^2}{\theta\sigma^2} \left|  \sum_{i=3}^{n-2}
\frac{\Delta_{i}}{\Delta_i+\Delta_{i+1}}
(  \overline W_{i+1,0}^2 - 1  ) \right|
  +  | \bar{\delta}_{n} | 
\nonumber  \\
 \leq & K \sum_{i=3}^{n-2} ( \Delta_{i} + \Delta_{i+1}) \overline  W_{i,0}^2
 +
K \sum_{i=3}^{n-2} ( \Delta_{i} + \Delta_{i+1}) \overline  W_{i+1,0}^2
 +  | \bar{\delta}_{n} |  \\
&
+ \frac{\theta_0\sigma_0^2}{\theta\sigma^2} \left|  \sum_{i=3}^{n-2}
\frac{\Delta_{i+1}}{\Delta_i+\Delta_{i+1}}
(  \overline W_{i,0}^2 - 1  ) \right|
+ \frac{\theta_0\sigma_0^2}{\theta\sigma^2} \left|  \sum_{i=3}^{n-2}
\frac{\Delta_{i}}{\Delta_i+\Delta_{i+1}}
(  \overline W_{i+1,0}^2 - 1  ) \right|. \nonumber
\end{align}
Let us show that the terms in the right-hand side of \eqref{eq:for:lemma:six:p:sept} are a.s. $o(n)$. We have
\[
\sum_{i=3}^{n-2} ( \Delta_{i} + \Delta_{i+1}) \overline  W_{i,0}^2
\leq
K \sup_{i=3,...,n-2} \overline  W_{i,0}^2
\leq 1 + \sup_{i=3,...,n-2}  | \overline{W }_{i,0}^2 - 1| 
= o(n),
\]
a.s. from Lemma \ref{lem:lem3}. Similarly
\[
\sum_{i=3}^{n-2} ( \Delta_{i} + \Delta_{i+1}) \overline  W_{i+1,0}^2
= o(n)
\]
a.s. Also, using theorem 2.1 in \cite{hu97strong} with $a_n = n$, we have
\[
\sum_{i=3}^{n-2}
\frac{\Delta_{i+1}}{\Delta_i+\Delta_{i+1}}
(  \overline W_{i,0}^2 - 1  ) = o(n)
\] 
a.s. and
\[
\sum_{i=3}^{n-2}
\frac{\Delta_{i}}{\Delta_i+\Delta_{i+1}}
(  \overline W_{i+1,0}^2 - 1  ) = o(n)
\]
a.s. Hence finally 
\[
\underset{(\theta,\sigma^2) \in J}\sup \left| \sum_{i=3}^{n-1} \alpha_i ^2B_i W_{i,0}^2 - n \frac{\theta_0\sigma_0^2}{\theta\sigma^2} \right| = o(n) ~ ~ \mbox{a.s.}
\]

$\bullet$ We can now conclude the proof. We have
\begin{align*}
S_n(\theta,\sigma^2) - S_n(\tilde{\theta},\tilde{\sigma}^2) 
& = 
n \log \frac{\sigma^2}{\tilde{\sigma}^2}
- n \log \frac{\tilde{\theta}}{\theta} 
+ n \frac{\theta_0 \sigma_0^2}{\theta \sigma^2} 
- n \frac{\theta_0 \sigma_0^2}{\tilde{\theta} \tilde{\sigma^2}} +  \delta_{n}
& = n \left(  \log \frac{\sigma^2 \theta}{\sigma_0^2 \theta_0 } + \frac{\theta_0 \sigma_0^2}{\theta \sigma^2}  
-1  \right) +  \delta_{n},
\end{align*}
by reminding that $\tilde{\theta} \tilde{\sigma^2} = \theta_0 \sigma_0^2$. The proof of Lemma \ref{lem6} is thus complete.
\end{proof}

\subsection{Proof of Theorem \ref{th:asymptotic}}\label{ssec:proof_AN}

Let us first prove \eqref{eq:AN}
in the case $aB < \theta_0 \sigma_0^2$; $Ab > \theta_0 \sigma_0^2$. We shall then discuss the other cases at the end. In that view, let
\begin{align}\label{eq:partial}
\psi ( \theta,\sigma^2)=\frac{\partial}{\partial \theta}S_n( \theta,\sigma^2).
\end{align}
Then from Theorem \ref{th:consistency}, a.s. for $n$ large enough, $\hat{\theta} \in (a,A)$. Thus a.s. for $n$ large enough
$(\hat\theta,\hat\sigma^2)$ satisfies $\psi (\hat\theta,\hat\sigma^2)= 0$. We shall approximate $\psi ( \theta,\sigma^2)$ uniformly in $( \theta,\sigma^2) \in J$.  

Starting from \eqref{eq:expression:Sn:un}, \eqref{eq:rv_1}, \eqref{eq:Sn} and \eqref{eq:decomp_derniere_somme} we can write
\begin{align*}
\psi ( \theta,\sigma^2)=
&\frac{2\Delta_2e^{-2 \theta \Delta_2} }{1-e^{-2 \theta \Delta_2}} + \frac{2\Delta_ne^{-2 \theta \Delta_n} }{1-e^{-2 \theta 
\Delta_n}} +\frac{2\Delta_2e^{- \theta \Delta_2} y_2(y_1-e^{-\theta \Delta_2}y_2)}{\sigma^2(1-e^{-2 \theta \Delta_2})}\\
& - \frac{2\Delta_2e^{-2 \theta \Delta_2} (y_1-e^{-\theta \Delta_2}y_2)^2}{\sigma^2(1-e^{-2 \theta \Delta_2})^2}
 +\frac{2\Delta_ne^{- \theta \Delta_n} y_{n-1}W_{n}}{\sigma^2(1-e^{-2 \theta \Delta_n})}
-  \frac{2\Delta_ne^{-2 \theta \Delta_n}W_{n}^2}{\sigma^2(1-e^{-2 \theta \Delta_n})^2}\\
& - \sum_{i=2}^{n-1}\frac{A'_i}{A_i} + \Sigma'_1 +\Sigma'_2 + 2 \Sigma'_3 + T'_2 +T'_n - 2\Sigma'_4
\end{align*}
where

$\circ$ $\Sigma'_1 = \frac{\partial}{\partial \theta}\Sigma_1= \displaystyle\sum_{i=3}^{n-1}\alpha_iD_i W_{i,0}^2$;

$\circ$ $\Sigma'_2 = \frac{\partial}{\partial \theta}\Sigma_2= \displaystyle\sum_{i=3}^{n-1}\alpha_i \left[D_i(e^{- \theta_0 \Delta_i}-e^{-\theta \Delta_i})  + 2 \alpha_iB_i\Delta_ie^{-\theta \Delta_i} \right]y^2_{i-1}(e^{- \theta_0 \Delta_i}-e^{-\theta \Delta_i}) $;

$\circ$ $\Sigma'_3 = \frac{\partial}{\partial \theta}\Sigma_3= \displaystyle\sum_{i=3}^{n-1}\alpha_i \left[D_i(e^{- \theta_0 \Delta_i}-e^{-\theta \Delta_i})  + \alpha_iB_i\Delta_ie^{-\theta \Delta_i} \right]W_{i,0} y_{i-1}$;

$\circ$ $T'_2 = \frac{\partial}{\partial \theta} T_2= 2\alpha'_2\alpha_2 \displaystyle\frac{W_2^2}{A_2} + \alpha^2_2\left(\frac{2 \Delta_2e^{- \theta \Delta_2}y_1 W_2}{A_2} -  W^2_2\frac{A'_2}{A_2^2} \right)$;

$\circ $ $T'_n = \frac{\partial}{\partial \theta} T_n=  \displaystyle\frac{\alpha_n e^{- 2\theta \Delta_n}}{A_{n-1}}\left[(2\alpha'_n-2\alpha_n\Delta_n-\frac{A'_{n-1}}{A_{n-1}}\alpha_n) W_{n} +   2\alpha_n\Delta_ne^{- \theta \Delta_n}y_{n-1} \right]W_{n}$;\\

$\circ$ $\Sigma'_4 = \frac{\partial}{\partial \theta}\Sigma_4= \displaystyle\sum_{i=2}^{n-1}C'_iW_{i}W_{i+1} +\displaystyle\sum_{i=2}^{n-1}C_i\Delta_ie^{-\theta \Delta_{i}}W_{i+1}y_{i-1}+ \displaystyle\sum_{i=2}^{n-1}C_i\Delta_{i+1}e^{-\theta \Delta_{i+1}}W_{i}y_{i}$

where $C'_i$ is the derivative of $C_i$ w.r.t. $\theta$ defined in \eqref{def:C:prime} and
\begin{align}\label{def:D}
D_i\defeq 2\alpha'_iB_i+ \alpha_i B'_i,\ \quad \text{for $i=3,\ldots,n-1$}.
\end{align}

First, we consider the terms $\Sigma'_1$ and  $\Sigma'_4$ in the following lemma (proved in Section \ref{ssec:proof_lem_AN}).
 
\begin{lemma}\label{lem:sigma_14} We have
\begin{align*}
\Sigma'_1 &= \frac{\theta_0\sigma_0^2}{\theta^2\sigma^2}\sum_{i=3}^{n-1}  q_i \overline W_{i,0}^2 + \delta_n
\end{align*}
and
\begin{align*}
\Sigma'_4 &= \frac{\theta_0\sigma_0^2}{\theta^2\sigma^2} \sum_{i=3}^{n-1} \frac{\sqrt{\Delta_i\Delta_{i+1}}}{\Delta_i+\Delta_{i+1}} \overline W_{i,0} \overline W_{i+1,0}+ \delta_n,
\end{align*}
where $q_i$ and $C'_i$ have been defined in \eqref{def:q} and \eqref{def:C:prime}.
\end{lemma}

Now we prove that the remaining terms in $\psi (\sigma^2,\theta)$ are $O_{\P}(1)$, uniformly in $(\theta,\sigma^2) \in J$, at the exception of $\sum_{i=2}^{n-1} A'_i/A_i$, leading to the following lemma (proved in Section \ref{ssec:proof_lem_AN}).

\begin{lemma}\label{lem:psi}
We obtain 
\begin{align*}
\psi (\theta,\sigma^2)= \frac{\theta_0\sigma^2_0}{\theta^2\sigma^2} \sum_{i=3}^{n-1}\left[ q_i\overline W_{i,0}^2 
-2 \frac{\sqrt{\Delta_i\Delta_{i+1}}}{\Delta_i+\Delta_{i+1}} \overline W_{i,0} \overline W_{i+1,0} \right] -\frac{n}{\theta}+ \delta_n .
\end{align*} 
\end{lemma}

Since $\hat\theta^2\hat\sigma^2\psi (\hat\theta,\hat\sigma^2)=0$ with probability going to $1$, and since we can show that $\sum_{i=3}^{n-1} q_i = n + O(1)$, we have
\begin{align} \label{eq:for:TCL}
n(\hat\theta\hat\sigma^2-\theta_0\sigma^2_0 )
=  \theta_0\sigma^2_0 \sum_{i=3}^{n-1}\left[q_i(\overline W_{i,0}^2 -1) 
-2 \frac{\sqrt{\Delta_i\Delta_{i+1}}}{\Delta_i+\Delta_{i+1}} \overline W_{i,0} \overline W_{i+1,0}\right] +O_{\P}(1).
\end{align}
We want to establish a Central Limit Theorem for $\sqrt{n}(\hat\theta\hat\sigma^2-\theta_0 \sigma^2_0)$. In that view, we define 
$X_{i} :=  q_i(\overline W_{i,0}^2 -1) - 2\frac{\sqrt{\Delta_i\Delta_{i+1}}}{\Delta_i+\Delta_{i+1}} \overline W_{i,0} \overline W_{i+1,0}$ and  we apply Theorem 2.1 in \cite{Neumann13} for weakly dependent variables 
(since $X_{i}$ is not necessarily independent with $X_{i-1}$  and $X_{i+1}$ but is independent with $X_k$ for $|i-k| \geq 2$). \\

Note that we can show easily that $\tau_n^2=\frac{1}{ n}\Var(\sum_{i=3}^{n-1} X_i)$, and assume  
\begin{align*}
\sqrt{n}\frac{(\hat \theta\hat \sigma^2-\theta_0 \sigma_0^2)}{\theta_0 \sigma_0^2\tau_{n}} \overset{\mathcal{D}}{\underset{n \to \infty}{\not\rightarrow}} \mathcal N (0,1).
\end{align*}
By Proposition \ref{prop:encadrement}, we can extract a subsequence $\epsilon_n$ so that $\tau_{\epsilon_n}^2\underset{n \to \infty}{\rightarrow} {\tau^2}$ with $\tau^2 \in [2,4]$ and so that
\begin{align*}
\sqrt{\epsilon_n}\frac{(\hat \theta\hat \sigma^2-\theta_0 \sigma_0^2)}{\theta_0 \sigma_0^2\tau_{\varepsilon_n}}\overset{\mathcal{D}}{\underset{n \to \infty}{\not\rightarrow}} \mathcal N(0,1).
\end{align*}

The triangular array $\left(X_i/\sqrt{\varepsilon_n}\right)_{i=3,\ldots,\varepsilon_n-1}$ satisfies the conditions of \cite[Theorem 2.1]{Neumann13}, thus we obtain
\begin{align*}
\frac{1}{\sqrt{\epsilon_n}} \sum_{i=3}^{\epsilon_n-1} X_i 
\overset{\mathcal{D}}{\underset{n \to \infty}{\rightarrow}}
 \mathcal{N}(0, \tau^2).
\end{align*}
Now, from \eqref{eq:for:TCL},
\begin{align*}
\sqrt{\epsilon_n}\frac{(\hat \theta\hat \sigma^2-\theta_0 \sigma_0^2)}{\theta_0 \sigma_0^2\tau_{\varepsilon_n}}
= \frac{1}{\sqrt{\epsilon_n}} \sum_{i=3}^{\varepsilon_n-1}\frac{X_i}{\tau_{\varepsilon_n}}+o_{\P}(1)
=  \frac{1}{\sqrt{\epsilon_n}} \sum_{i=3}^{\varepsilon_n-1}\frac{X_i}{\tau}+\left(\frac{1}{\tau_{\epsilon_n}}-\frac{1}{\tau}\right)  \frac{1}{\sqrt{\epsilon_n}} \sum_{i=3}^{\varepsilon_n-1}X_i +o_{\P}(1).
\end{align*}
Since $ \frac{1}{\sqrt{\epsilon_n}} \sum_{i=3}^{\varepsilon_n-1}X_i=O_{\P}(1)$ and $\left(\frac{1}{\tau_{\epsilon_n}}-\frac{1}{\tau}\right)=o(1)$, we get 
by Slutsky's lemma 
\begin{align*}
\sqrt{\epsilon_n}\frac{(\hat \theta\hat \sigma^2-\theta_0 \sigma_0^2)}{\theta_0 \sigma_0^2\tau_{\varepsilon_n}}\overset{\mathcal{D}}{\underset{n \to \infty}{\rightarrow}} \mathcal N(0,1),
\end{align*}
which is contradictory and ends the proof of \eqref{eq:AN}.

Now \eqref{eq:AN_1} is under consideration only when $b=B=\sigma_1^2$ and so when $aB < \theta_0 \sigma_0^2$; $Ab > \theta_0 \sigma_0^2$. Thus \eqref{eq:AN_1} is a special case of \eqref{eq:AN}.
Now, when $aB > \theta_0 \sigma_0^2$; $Ab < \theta_0 \sigma_0^2$, we have almost surely for $n$ large enough $(\partial / \partial \sigma^2) S_n( \hat{\theta},\hat{\sigma}^2 ) = 0$, so that the estimator $\hat{\sigma}_2^2$ can be expressed  explicitly, by differentiating the terms in \eqref{eq:expression:Sn:un}, \eqref{eq:rv_1} and \eqref{eq:Sn} w.r.t. $\sigma^2$. Hence, \eqref{eq:AN} can be proved in the case $aB > \theta_0 \sigma_0^2$; $Ab < \theta_0 \sigma_0^2$ by using identical techniques as in the case $aB < \theta_0 \sigma_0^2$; $Ab > \theta_0 \sigma_0^2$. We omit the details to save space. Finally, \eqref{eq:AN_2} is under consideration only when $a=A=\theta_2$ and so when $aB > \theta_0 \sigma_0^2$; $Ab < \theta_0 \sigma_0^2$. Thus \eqref{eq:AN_2} is a special case of \eqref{eq:AN}.

\subsection{Proof of Propositions \ref{prop:encadrement} and \ref{prop2}}\label{ssec:proof:props}

\begin{proof}[Proof of Proposition \ref{prop:encadrement}]
We have
\begin{align*}
\tau_n^2
=  \frac{2}{n}\sum_{i=3}^{n-1} \left[q_i^2+ 2\frac{\Delta_i\Delta_{i+1}}{(\Delta_i+\Delta_{i+1})^2} \right].
\end{align*}

(i) Upper bound for $\tau_n^2$. Let $a_i= \displaystyle\frac{\Delta_{i+1}}{\Delta_i+\Delta_{i+1}}$, note that $\displaystyle\frac{\Delta_{i-1}}{\Delta_i+\Delta_{i-1}}= 1- a_{i-1}$ and $q_i=a_i+1-a_{i-1}$.
First, we have after some trivial computations,
\begin{align} \label{eq:for:upper:bound:variance}
\tau_n^2
= &  \frac{2}{n}\sum_{i=3}^{n-1} [(a_i + 1- a_{i-1})^2 + 2 a_i(1- a_i)] \nonumber \\
= & \frac{2}{n}\sum_{i=3}^{n-1} (1+ 2 a_i -2 a_ia_{i-1})  + o(1)\\
\leq &  2+ 4 m +  o(1), \nonumber
\end{align}
where $m\defeq \displaystyle\frac{1}{n-3}\sum_{i=3}^{n-1}a_i$. 

Also, since for $k=2,...,n-1$, $0 \leq a_{k} \leq 1$, we have $1 + 2(1 - a_{i-1}) a_i \leq 3 - 2 a_{i-1}$. Thus, from \eqref{eq:for:upper:bound:variance},
\begin{align*}
\tau_n^2 \leq \frac{2}{n} \sum_{i=3}^{n-1} (3-2a_{i-1})+o(1)= 6-4m+o(1).
\end{align*}

Finally, $\tau_n^2 \leq \min\left( 2+ 4m, 6-4m \right) + o(1)$. Since $\underset{m\in [0,1]}{\sup} \min \left( 2+ 4m, 6-4m \right) =4$, $\tau_n^2 \leq 4+ o(1)$.\\

(ii) Lower bound for $\tau_n^2$. Note that $ \displaystyle\frac{1}{n}\sum_{i=3}^{n-1} q_i =1+o(1)$.
Since $\displaystyle\frac{\Delta_i\Delta_{i+1}}{(\Delta_i+\Delta_{i+1})^2}  \geq 0$, we get 

\begin{align}\label{eq:lb1}
\tau_n^2 \geq \frac{2}{n}\sum_{i=3}^{n-1}q_i^2 \geq 2\left( \frac{1}{n}\sum_{i=3}^{n-1}q_i\right)^2 + o(1) = 2+o(1).
\end{align}
\end{proof}

\begin{proof}[Proof of Proposition \ref{prop2}] (i) After some computation, we have 
\begin{align*}
\tau_{n}^2= &4\gamma_n^2-4\gamma_n+4 + o(1).
\end{align*}
Since $\gamma_n=1/n$, then  $\tau_n^2 \underset{n \to \infty} \to  4$.\\

(ii) We have
\begin{align*}
\tau_n^2 & \leq  \frac 2n \sum_{i= \lfloor n^{\alpha}\rfloor +2}^{n-1} \left(\left(\frac{1}{i+2}+\frac{i}{i+1}\right)^2+\frac{2(i+1)}{(i+2)^2}\right)+o(1) \leq 2+o(1).
%
%
\end{align*}
As a consequence, this particular design realizes  $\tau_n^2 =2+o(1)$ by \eqref{eq:lb1}.
\end{proof}

\subsection{Proofs of the lemmas of Section \ref{ssec:proof_AN}}\label{ssec:proof_lem_AN}

\begin{proof}[Proof of Lemma \ref{lem:sigma_14}]
(i) From \eqref{eq:DL:alpha:B:prime} and \eqref{eq:DL:alpha:prime:B},
\begin{align*}
\Sigma_1' 
=  & \frac{1}{\theta}\sum_{i=3}^{n-1} \frac{\alpha_i}{\alpha_{i,0}} q_i \overline W_{i,0}^2 + \sum_{i=3}^{n-1} \frac{\alpha_i}{\alpha_{i,0}} \delta_{i,n}(\Delta_{i-1} + \Delta_i + \Delta_{i+1}) \overline W_{i,0}^2\\
\end{align*}
where $q_i$ has been defined in \eqref{def:q}. By \eqref{eq:DL:alpha}, we have 
\[
\frac{\alpha_i}{\alpha_{i,0}}=\frac{\theta_0\sigma_0^2 }{\sigma^2 \theta^2} + \delta_{i,n} \Delta_i.
\]
Moreover, since $\E[\overline W_{i,0}^2]=1$, one clearly has 
\[ 
\sum_{i=3}^{n-1} \frac{\alpha_i}{\alpha_{i,0}} \delta_{i,n} (\Delta_{i-1} + \Delta_i + \Delta_{i+1}) \overline W_{i,0}^2=O_{\P}(1)
\]
that leads to the desired result.

(ii) Now we study the first sum of $\Sigma_4'$ that rewrites $\sum_{i=2}^{n-1} C'_i(M_{i,1}+M_{i,2}+M_{i,3})$ using \eqref{eq:decomp_W} and where
$C'_i$ has been defined in \eqref{def:C:prime} and
\begin{equation*}
\begin{cases}
M_{i,1} \defeq \frac{\overline W_{i,0}\overline W_{i+1,0}}{\alpha_{i,0}^{1/2}\alpha_{i+1,0}^{1/2}}\\
M_{i,2} \defeq (e^{- \theta_0 \Delta_{i+1}} -e^{-\theta \Delta_{i+1}} )W_{i,0}  y_i \\
M_{i,3} \defeq (e^{ -\theta_0 \Delta_{i}} -e^{-\theta \Delta_{i}} )W_{i+1}y_{i-1}.\\
\end{cases}
\end{equation*}

$\bullet$ 
First, we consider $\displaystyle\sum_{i=2}^{n-1}  C'_iM_{i,1}$. By \eqref{eq:DL:C} and  \eqref{eq:DL:alpha} we can show
\[
\frac{C'_i}{\alpha_{i,0}^{1/2}\alpha_{i+1,0}^{1/2}} 
=
\frac{\theta_0 \sigma_0^2}{\theta^2 \sigma^2}
\frac{\sqrt{\Delta_{i} \Delta_{i+1}}}{\Delta_{i} + \Delta_{i+1}}
+ \delta_{i,n} (\Delta_i+\Delta_{i+1}).
\]
Furthermore we have

\begin{align*}
\E\left[ \left| \sum_{i=2}^{n-1} \delta_{i,n} (\Delta_i+\Delta_{i+1}) \overline W_{i,0} \overline W_{i+1,0}  \right|\right] 
\leq& K \sum_{i=2}^{n-1}(\Delta_i+\Delta_{i+1}) \sqrt{\E[\overline W_{i,0}^2 ]\E[\overline W_{i+1,0}^2]}\\
= & K \sum_{i=2}^{n-1}(\Delta_i+\Delta_{i+1}) \leq K.
\end{align*} 
Thus
\begin{align*}
\sum_{i=2}^{n-1} \frac{C'_i}{\alpha_{i,0}^{1/2} \alpha_{i+1,0}^{1/2}} \overline W_{i,0} \overline W_{i+1,0} = \frac{\theta_0\sigma^2_0}{\theta^2\sigma^2}\sum_{i=2}^{n-1}\frac{\sqrt{\Delta_i\Delta_{i+1}}}{\Delta_i+\Delta_{i+1}} \overline W_{i,0} \overline W_{i+1,0}+ \delta_n.
\end{align*} 
Hence 
\[
\sum_{i=2}^{n-1} C'_i M_{i,1} = 
\frac{\theta_0 \sigma_0^2}{\theta^2 \sigma^2} \sum_{i=2}^{n-1} \frac{ \sqrt{\Delta_i \Delta_{i+1} } }{\Delta_i + \Delta_{i+1}} \overline W_{i,0}  \overline W_{i+1,0} + \delta_n.
\]

$\bullet$ Second, one clearly has 
\[
\sum_{i=2}^{n-1}  C'_iM_{i,2}
= \frac{\theta - \theta_0} {2\theta^2\sigma^2} \sum_{i=2}^{n-1} \frac{ \Delta_{i+1}}{\Delta_i+\Delta_{i+1}}(1+\delta_{i,n}(\Delta_i+\Delta_{i+1})) W_{i,0}y_i.
\]
Hence $\underset{( \theta,\sigma^2) \in J}{\sup} \left|\sum_{i=2}^{n-1}  C'_iM_{i,2} \right|= O_{\P}(1)$ by Lemma \ref{lem:lem4} (i).\\

$\bullet$ Third, we get  
\[
 \sum_{i=2}^{n-1}  C'_iM_{i,3}=  \frac{\theta-\theta_0}{2\theta^2\sigma^2} \sum_{i=2}^{n-1} \frac{\Delta_i}{\Delta_i+\Delta_{i+1}} \left(1+\delta_{i,n}(\Delta_i+\Delta_{i+1})\right)W_{i+1}y_{i-1}
 \]
and $\underset{( \theta,\sigma^2) \in J}{\sup} \left|\sum_{i=2}^{n-1}  C'_iM_{i,3} \right|= O_{\P}(1)$ by Lemma \ref{lem:lem4} (iv).\\

(iii) We now consider the second and third sums in $\Sigma'_4$.

Using \eqref{eq:DL:alpha} and \eqref{eq:DL:A}, we can show
\[
C_i\Delta_i e^{-\theta \Delta_{i}}
= 
\frac{1}{2\theta \sigma^2} \frac{ \Delta_i }{ \Delta_i + \Delta_{i+1} } ( 1 + \delta_{i,n}(  \Delta_{i} + \Delta_{i+1}) )
\]
and
\[
C_i\Delta_{i+1} e^{-\theta \Delta_{i+1}}
= 
\frac{1}{2\theta \sigma^2} \frac{ \Delta_{i+1} }{ \Delta_i + \Delta_{i+1} } ( 1 + \delta_{i,n}(  \Delta_{i} + \Delta_{i+1}) ).
\]

Hence by Lemma \ref{lem:lem4} (iv) and (ii), $\underset{( \theta,\sigma^2) \in J}{\sup} \left\vert \displaystyle\sum_{i=2}^{n-1}C_i(\Delta_ie^{-\theta \Delta_{i}}W_{i+1}y_{i-1}+ \Delta_{i+1}e^{-\theta \Delta_{i+1}}W_{i}y_{i}) \right\vert = O_{\P}(1)$.
\end{proof}

\begin{proof}[Proof of Lemma \ref{lem:psi}]
$\bullet$ We have 
\[
\underset{a \leq \theta \leq A}\sup
\left| \displaystyle\frac{2\Delta_2e^{-2 \theta \Delta_2} }{1-e^{-2 \theta \Delta_2}} +\frac{2\Delta_ne^{-2 \theta \Delta_n} }{1-e^{-2 \theta \Delta_n}} \right| =O(1).
\]

$\bullet$ For $n$ large enough,
\begin{eqnarray*}
\underset{(\theta,\sigma^2) \in J}\sup
\left| \frac{2\Delta_2e^{-\theta \Delta_2} y_2[y_1-e^{-\theta \Delta_2}y_2]}{\sigma^2(1-e^{-2 \theta \Delta_2})} \right|
 \leq  K \underset{a \leq \theta \leq A}\sup |y_2||y_1-e^{-\theta \Delta_2}y_2| 
 \leq K \underset{t\in[0,1]}\sup Y(t)^2 =O_{\mathbb{P}}(1).
\end{eqnarray*}

$ \bullet$ 
Using $W_i^2 = W_{i,0}^2 + ( e^{- \theta_0 \Delta_i} - e^{- \theta \Delta_i} ) y_{i-1} W_i + W_{i,0}( e^{- \theta_0 \Delta_i} - e^{- \theta \Delta_i} ) y_{i-1}$ we can easily show
$\underset{( \theta,\sigma^2) \in J}{\sup} \left| \frac{2\Delta_2e^{-2 \theta \Delta_2} (y_1-e^{-\theta \Delta_2}y_2)^2}{\sigma^2(1-e^{-2 \theta \Delta_2})^2}
\right| =O_{\P}(1)$,
$\underset{( \theta,\sigma^2) \in J}{\sup} \left| \displaystyle\frac{2\Delta_ne^{- \theta \Delta_n} y_{n-1}W_{n}}{\sigma^2(1-e^{-2 \theta \Delta_n})} \right| =O_\P(1)$
and 
$\underset{( \theta,\sigma^2) \in J}{\sup} \left| \frac{2\Delta_ne^{-2 \theta \Delta_n} W_{n}^2}{\sigma^2(1-e^{-2 \theta \Delta_n})^2} \right|=O_{\P}(1)$.\\

$\bullet$ \textbf{Term $\Sigma'_2$}: First, using \eqref{eq:DL:alpha:B:prime}, \eqref{eq:DL:alpha:prime:B} and the definition \eqref{def:q} of $q_i$, the deterministic quantity $D_i$ is bounded uniformly in $(\theta,\sigma^2) \in J$, and so is $(e^{-2\theta_0\Delta_i}-e^{-2\theta\Delta_i})^2\alpha_i/\Delta_i$ from \eqref{eq:DL:alpha}.
By \eqref{ass:y_bounded},  we are led to
\begin{align*}
\underset{( \theta,\sigma^2) \in J}{\sup} \left| \displaystyle\sum_{i=3}^{n-1}\Delta_i\frac{\alpha_i}{\Delta_i} (e^{- \theta_0 \Delta_i}-e^{-\theta \Delta_i})^2  y^2_{i-1}D_i \right| = O_{\P}(1).
\end{align*}
Similarly, $ \underset{( \theta,\sigma^2) \in J}{\sup} \left| \displaystyle\sum_{i=3}^{n-1} 2 \alpha_i^2 B_i \Delta_i e^{-\theta \Delta_i} (e^{- \theta_0 \Delta_i}-e^{-\theta \Delta_i}) y^2_{i-1} \right|= O_{\P}(1)$ and thus
$\underset{( \theta,\sigma^2) \in J}{\sup} |\Sigma'_2| = O_{\P}(1)$.\\

$\bullet$ \textbf{Term $\Sigma'_3$}: First, from \eqref{eq:DL:alpha:B:prime}, \eqref{eq:DL:alpha:prime:B} and \eqref{eq:DL:alpha}, we have
\[
\underset{n \in \N, i=2,...,n-1,( \theta,\sigma^2) \in J}{\sup}
\left|
\alpha_i D_i (e^{- \theta_0 \Delta_i} - e^{- \theta \Delta_i}) - \frac{\theta - \theta_0}{2 \sigma^2 \theta^2} q_i
\right| \leq K.
\]
Hence, proceeding as in the proof of Lemma \ref{lem:lem4}, we can show
\begin{align*}
\underset{( \theta,\sigma^2) \in J}{\sup}
\left| 
\displaystyle\sum_{i=3}^{n-1} \alpha_iD_i(e^{- \theta_0 \Delta_i}-e^{-\theta \Delta_i}) W_{i} y_{i-1} \right|
=O_{\P}(1).
\end{align*}

Second, we can show
\[
\underset{n \in \N, i=3,...,n-1,( \theta,\sigma^2) \in J}{\sup}
\left|
\alpha_i^2 B_i \Delta_i e^{- \theta \Delta_i} - \frac{1}{2 \sigma^2 \theta} 
\left(
\frac{\Delta_{i+1}}{\Delta_{i}+\Delta_{i+1}}
+
\frac{\Delta_{i-1}}{\Delta_{i}+\Delta_{i-1}}
\right)
\right| \leq K.
\]
Hence we can show
\[
\underset{( \theta,\sigma^2) \in J}{\sup}
\left| \sum_{i=3}^{n-1} \alpha_i^2 B_i\Delta_i e^{-\theta \Delta_i} W_{i,0} y_{i-1}\right|
= O_\P(1),
\]
as in the proof of Lemma \ref{lem:lem4}.
Hence finally, $\underset{( \theta,\sigma^2) \in J}{\sup} \Sigma'_3 = O_{\P}(1)$.\\

$\bullet$ \textbf{Term $T'_2$}: From (\ref{def:A}), we have 
\[
| T'_2| \leq 2 |\alpha'_2|W_2^2 + |\alpha_2|\Delta_2e^{-\theta\Delta_2} | y_1| +W_2^2 | A'_2|.
\]
We can show 
\[
\underset{( \theta,\sigma^2) \in J}{\sup} (|\alpha'_2|W_2^2 + W_2^2 | A'_2| ) =  O_{\P}(1)
\]
by using $W_2^2 = W_{2,0}^2 + ( e^{- \theta_0 \Delta_2} - e^{- \theta \Delta_2} ) y_1W_2 + W_{2,0}( e^{- \theta_0 \Delta_2} - e^{- \theta \Delta_2} ) y_1$.

Finally, $
\underset{( \theta,\sigma^2) \in J}{\sup} |\alpha_2|\Delta_2e^{-\theta\Delta_2} =  O_{\P}(1)
$, which finally shows $\underset{( \theta,\sigma^2) \in J}{\sup} |T_2'| = O_\P(1)$.

~

$\bullet$ \textbf{Term $T'_n$}: 
Using \eqref{eq:DL:alpha}, \eqref{eq:DL:A} and \eqref{eq:DL:A:prime}, we get
\[
\underset{( \theta,\sigma^2) \in J}{\sup} \left|2\alpha_n\Delta_n
\right|
\leq K \quad \text{and} \quad 
\underset{( \theta,\sigma^2) \in J}{\sup}
| 2\alpha_n\Delta_ne^{- 2\theta \Delta_n}|
\leq K.
\]
Moreover, one has $| \frac{\alpha_n e^{-2\theta\Delta_n}}{A_{n-1}} | \leq 1$.
Finally, we have $\underset{( \theta,\sigma^2) \in J}{\sup} W_n^2 = O_\P(1)$ and
$\underset{( \theta,\sigma^2) \in J}{\sup} | y_{n-1} W_n | = O_\P(1)$. Hence, in order to show $\underset{( \theta,\sigma^2) \in J}{\sup} |T_n'| = O_\P(1)$ it remains to show $\underset{( \theta,\sigma^2) \in J}{\sup} |\alpha'_n W_n^2| = O_\P(1) $
and $\underset{( \theta,\sigma^2) \in J}{\sup} |(A_{n-1}'/A_{n-1}) \alpha_n W_n^2| = O_\P(1) $.
This is shown by using $W_n^2 = W_{n,0}^2 + ( e^{- \theta_0 \Delta_n} - e^{- \theta \Delta_n} ) y_{n-1} W_n + W_{n,0}( e^{- \theta_0 \Delta_n} - e^{- \theta \Delta_n} ) y_{n-1} $.

$\bullet$   \textbf{Term $\displaystyle\sum_{i=2}^{n-1}\frac{A'_i}{A_i}$}: By \eqref{eq:DL:ratio:A},
  
  $$\sum_{i=2}^{n-1}\frac{A'_i}{A_i} = \frac{n}{\theta}+ \delta_n .$$
 
Finally,
\begin{align*}
\psi (\sigma^2,\theta)&= \frac{\theta_0\sigma^2_0}{\theta^2\sigma^2}\sum_{i=3}^{n-1}q_i\overline W_{i,0}^2 -2 \sum_{i=3}^{n-1} \frac{\sqrt{\Delta_i \Delta_{i+1}}}{\Delta_i+ \Delta_{i+1}} \overline W_{i,0} \overline W_{i+1,0} -\frac{n}{\theta}+ \delta_n\\
\end{align*} 
using Lemma \ref{lem:sigma_14}.
\end{proof}

\subsection{Proof of Theorems \ref{theorem:consistency:with:mean} and \ref{theorem:AN:with:mean}} \label{subsection:proof:with:mean}

In this section and the next one, we let $||\boldsymbol{A}||$ denote the largest singular value of a matrix $\boldsymbol{A}$ and $||\boldsymbol{v}||$ denote the Euclidean norm of a vector $||\boldsymbol{v}||$. Finally, for $k \in \{1,...,p\}$, we let $\f^{(k)} = (f_k(s_1),...,f_k(s_n))'$.
We first provide a decomposition of $\bar{S}_n$ in the following lemma.
\begin{lemma} \label{lem:expression:barSn}
We have
\[
\bar{S}_n(\theta,\sigma^2) = S_n(\theta,\sigma^2) -
r_1(\theta) + \frac{r_2(\theta) + 2 r_3(\theta) - r_4(\theta)}{\sigma^2},  
\]
with
\begin{align*}
r_1(\theta)  & =  \sum_{i=1}^n \left[
\log \left(
 (\R_{\theta}^{-1})_{ii}  - \bar{\epsilon}_{-i} 
 \right)
 -
 \log \left(
 (\R_{\theta}^{-1})_{ii} 
 \right)
\right], \\
r_2(\theta)  & = \sum_{i=1}^n 
(\R_{\theta}^{-1})_{ii} \epsilon_{-i}^2,
 \\
r_3(\theta)  & = \sum_{i=1}^n 
(\R_{\theta}^{-1})_{ii} \epsilon_{-i} 
(y_i - \hat{Y}_{\theta,-i}(s_i)),
 \\
 r_4(\theta)  & = \sum_{i=1}^n 
\bar{\epsilon}_{-i}
(y_i - \hat{Y}_{\theta,-i}(s_i) + \epsilon_{-i})^2,
 \\
 \epsilon_{-i} & =  [\f_i' - \r_{\theta,-i}' \R_{\theta,-i}^{-1} \F_{-i} ] (\bbeta_0 - \hat{\bbeta}_{-i}), \\
  \bar{\epsilon}_{-i} & = \e_{i,n}' \R_{\theta}^{-1} \F ( \F' \R_{\theta}^{-1} \F )^{-1} \F' \R_{\theta}^{-1} \e_{i,n},
\end{align*}
where $\e_{i,n}$ is the $i$-th base column vector of $\mathbb{R}^n$. We remark that $r_k(\theta)$ does not depend on $\sigma^2$ for $k=1,...,4$.
\end{lemma}

We now show that, in Lemma \ref{lem:expression:barSn}, the term $S_n$ only, corresponding to the zero-mean case, is preponderant.

\begin{lemma} \label{lem:control:rest}
With the notation of Lemma \ref{lem:expression:barSn}, we have for $k=1,...,4$ 
\[
\sup_{\theta \in [a,A]} | r_k(\theta)  |  = O(1)  ~ \mbox{a.s.}
\]
\end{lemma}

Because of Lemmas \ref{lem:expression:barSn} and \ref{lem:control:rest}, we have a.s.
\[
\sup_{(\theta,\sigma^2) \in J} 
\left|
\bar{S}_n( \theta,\sigma^2 ) 
-
S_n( \theta,\sigma^2 ) 
\right|
= O(1).
\]
Hence Theorem \ref{theorem:consistency:with:mean} follows from \eqref{eq:to:finish:consistency}. 

Also, from Lemmas \ref{lem:Sn}, \ref{lem:expression:barSn} and \ref{lem:control:rest}, we have
\begin{eqnarray*}
\bar{S}_n( \theta,\sigma^2) & = & n \log(\sigma^2)+\log( 1-e^{-2 \theta \Delta_2} ) +  \log( 1-e^{-2 \theta \Delta_n} )  \\
& & - \sum_{i=2}^{n-1} \log
\left(
 \frac{1}{
 1-e^{- 2 \theta \Delta_i}
 }
  + \frac{e^{- 2 \theta \Delta_{i+1} }
 }{
 1-e^{- 2 \theta \Delta_{i+1}}
 } 
 \right) - r_1(\theta)
\\
& & + 
\frac{1}{\sigma^2}
\Bigg \{
 \frac{ (y_1 - e^{- \theta \Delta_2} y_2 )^2 }{ (1-e^{-2 \theta \Delta_2})}    + \frac{ (y_n - e^{- \theta \Delta_n} y_{n-1} )^2 }{  (1-e^{-2 \theta \Delta_n}) }
 \\
 &  &  +  \sum_{i=2}^{n-1} 
 \left[
  \frac{1}{
 1-e^{- 2 \theta \Delta_i}
 }
  + \frac{e^{- 2 \theta \Delta_{i+1} }
 }{
 1-e^{- 2 \theta \Delta_{i+1}}
 } 
 \right]
 \left[
 y_i
 -
 \frac{
   \frac{e^{- \theta \Delta_i}}{
 1-e^{- 2 \theta \Delta_i}
 } y_{i-1}
  + \frac{e^{-\theta \Delta_{i+1} }
 }{
 1-e^{- 2 \theta \Delta_{i+1}}
 } y_{i+1}  
 }{
  \frac{1}{
 1-e^{- 2 \theta \Delta_i}
 }
  + \frac{e^{- 2 \theta \Delta_{i+1} }
 }{
 1-e^{- 2 \theta \Delta_{i+1}}
 } 
 }
 \right]^2 + \delta_n(\theta) \Bigg \},
\end{eqnarray*}
where $\delta_n(\theta)$ does not depend on $\sigma^2$ and satisfies $\sup_{a \leq \theta \leq A} | \delta_n(\theta) | =O_\P(1) $.
Since $\check{\theta} \check{\sigma}^2 \to \theta_0 \sigma_0^2$ a.s. from Theorem \ref{theorem:consistency:with:mean}, and since $aB > \theta_0 \sigma_0^2; Ab < \theta_0 \sigma_0^2$, we have $ [ \partial / \partial \sigma^2 ] \bar{S_n} (\check{\theta} , \check{\sigma}^2 ) = 0 $ for $n$ large enough almost surely. Hence we obtain for $n$ large enough almost surely,

\begin{eqnarray*}
0 & = & \frac{n}{ \check{\sigma}^2 } 
- \frac{1}{\check{\sigma}^4} 
\Bigg \{
 \frac{ (y_1 - e^{- \check{\theta} \Delta_2} y_2 )^2 }{ (1-e^{-2 \check{\theta} \Delta_2})}    + \frac{ (y_n - e^{- \check{\theta} \Delta_n} y_{n-1} )^2 }{  (1-e^{-2 \check{\theta} \Delta_n}) }
 \\
 &  &  +  \sum_{i=2}^{n-1} 
 \left[
  \frac{1}{
 1-e^{- 2 \check{\theta} \Delta_i}
 }
  + \frac{e^{- 2 \check{\theta} \Delta_{i+1} }
 }{
 1-e^{- 2 \check{\theta} \Delta_{i+1}}
 } 
 \right]
 \left[
 y_i
 -
 \frac{
   \frac{e^{- \check{\theta} \Delta_i}}{
 1-e^{- 2 \check{\theta} \Delta_i}
 } y_{i-1}
  + \frac{e^{-\check{\theta} \Delta_{i+1} }
 }{
 1-e^{- 2 \check{\theta} \Delta_{i+1}}
 } y_{i+1}  
 }{
  \frac{1}{
 1-e^{- 2 \check{\theta} \Delta_i}
 }
  + \frac{e^{- 2 \check{\theta} \Delta_{i+1} }
 }{
 1-e^{- 2 \check{\theta} \Delta_{i+1}}
 } 
 }
 \right]^2 + \delta_n(\check{\theta}) \Bigg \}.
\end{eqnarray*}

As noted at the end of the proof of Theorem \ref{th:asymptotic}, using identical techniques as for proving Theorem \ref{th:asymptotic}, one can finish the proof of Theorem \ref{theorem:AN:with:mean}.

\subsection{Proofs of the lemmas in Section \ref{subsection:proof:with:mean}}

\begin{proof}[Proof of Lemma \ref{lem:expression:barSn}]
We have, with  $\y_{-i} = (y_1,...,y_{i-1},y_{i+1},...,y_n)'$,
\begin{eqnarray*}
 Z(s_i) -	\hat{Z}_{\theta,-i}(s_i) 
 &=& y_i + \f_i' \bbeta_0  -\f_i' \hat{\bbeta}_{-i}  - \r_{\theta,-i}' \R_{\theta,-i}^{-1} ( \y_{-i} + \F_{-i} \bbeta_0 -\F_{-i} \hat{\bbeta}_{-i} )\\
  &=&y_i -\r_{\theta,-i}' \R_{\theta,-i}^{-1} \y_{-i} + (\f_i' -\r_{\theta,-i}' \R_{\theta,-i}^{-1} \F_{-i}) ( \bbeta_0-\hat{\bbeta}_{-i})\\
 &=&y_i -\hat Y_{\theta,-i}(s_i) + \epsilon_{-i}.
\end{eqnarray*}
Also 
\[
(\Q_{\theta}^{-})_{ii}=\e_{i,n}'\R_{\theta}^{-1} \e_{i,n} -\e_{i,n}'\R_{\theta}^{-1}\F ( \F' \R_{\theta}^{-1} \F)^{-1} \F' \R_{\theta}^{-1}e_{i,n}
= (\R_{\theta}^{-1})_{ii} - \bar{\epsilon}_{-i}.
\]
Then, from \eqref{eq:hat:sigma:pred:with:mean},

\begin{eqnarray*}
\bar{S}_n(\theta,\sigma^2)
&=& \sum_{i=1}^n\left[\log( \check{\sigma}^2_{ \theta,\sigma^2,-i}(s_i) ) + \frac{ (z_i -\hat{Z}_{\theta,-i}(s_i) )^2 }{ \check{\sigma}^2_{ \theta,\sigma^2,-i}(s_i) }\right]
\\
&=& n\log (\sigma^2) + 
\sum_{i=1}^n \left[ \log \left(\frac{1}{(\Q_{\theta}^{-})_{ii}}\right)\right] +
\frac{1}{\sigma^2}\sum_{i=1}^n (\Q_{\theta}^{-})_{ii}[y_i -\hat Y_{\theta,-i}(s_i) +\epsilon_{-i}]^2
\\
&=& S_n(\theta,\sigma^2) - \sum_{i=1}^n \left[ \log \left((\R_{\theta}^{-1})_{ii} - \bar{\epsilon}_{-i}\right) - \log((\R_{\theta}^{-1})_{ii} )\right] + \frac{1}{\sigma^2} \sum_{i=1}^n 
(\R_{\theta}^{-1})_{ii} \epsilon_{-i}^2 
\\
&&+ \frac{2}{\sigma^2} \sum_{i=1}^n (\R_{\theta}^{-1})_{ii}\epsilon_{-i} (y_i - \hat{Y}_{\theta,-i}(s_i))
- \frac{1}{\sigma^2}  \sum_{i=1}^n \bar{\epsilon}_{-i} (y_i - \hat{Y}_{\theta,-i}(s_i) + \epsilon_{-i})^2
\\
&=& S_n(\theta,\sigma^2) - r_1(\theta) + \frac{r_2(\theta) + 2 r_3(\theta) - r_4(\theta)}{\sigma^2}.
\end{eqnarray*}
\end{proof}

Before proving Lemma \ref{lem:control:rest}, we state and prove some intermediary results. 

\begin{lemma} \label{lem:unt:Rmun:y}
We have
\[
|| \F' \R_{\theta}^{-1} \y || \leq K \sup_{t \in [0,1]} | Y(t) |.
\]
\end{lemma}

\begin{proof}[Proof of Lemma \ref{lem:unt:Rmun:y}]
From \eqref{eq:tridiagonal:inverse} we have, with $k=1,...,p$ and with $\f = (f_1,...,f_n)' = \f^{(k)}$,
\begin{eqnarray*}
[ \F' \R_{\theta}^{-1} \y ]_{k} & = &
\f' \R_{\theta}^{-1} \y  \\
 &=& 
\frac{1}{1-e^{-2\theta\Delta_2}} f_1y_1
+ \frac{1}{1-e^{-2\theta\Delta_n}} f_n y_n 
 - \sum_{i=2}^{n} 
 \frac{e^{-\theta\Delta_i}}{1-e^{-2\theta\Delta_i}}f_i y_{i-1}
- \sum_{i=2}^{n} 
 \frac{e^{-\theta\Delta_i}}{1-e^{-2\theta\Delta_i}}f_{i-1} y_{i} \\
& & + \sum_{i=2}^{n-1} 
\left(
\frac{1}{1-e^{-2\theta\Delta_i}}
+ \frac{e^{-2 \theta \Delta_{i+1}}}{1-e^{-2\theta\Delta_{i+1}}}
\right)
f_i y_i \\
& = & 
\frac{
f_1y_1 - f_2y_1 e^{- \theta \Delta_2}
- f_1y_2 e^{- \theta \Delta_2} + f_2 y_2
}
{
1-e^{-2\theta\Delta_2}
}
\\
& & +
\sum_{i=3}^{n} 
\frac{
f_iy_i - f_iy_{i-1} e^{- \theta \Delta_i}
- f_{i-1} y_{i} e^{- \theta \Delta_i} 
+ e^{- 2 \theta \Delta_i} f_{i-1} y_{i-1}
}
{
1-e^{-2\theta\Delta_i}
} \\
& = &
a_1 + \sum_{i=3}^{n} m_i,
\end{eqnarray*}
say. We have, using \eqref{eq:DL:alpha}, and that $f_k$ is continuously differentiable,
\begin{eqnarray*}
a_1 & = & \left(  \frac{1}{2 \theta \Delta_2} + \delta_{1,n}  \right)
\left(
f_1 y_1 ( 1 - e^{-\theta \Delta_2} ) +
(f_1 - f_2) y_1 e^{- \theta \Delta_2}
+ y_2 f_2(1-e^{- \theta \Delta_2}) + e^{- \theta \Delta_2} y_2 (f_2 - f_1)
\right) \\
& = &
\frac{1}{2} ( f_1 y_1 + y_2 f_2)
+ \frac{f_1-f_2}{2 \theta \Delta_2} (y_1 - y_2) + \delta_{1,n} \Delta_2 \sup_{t \in [0,1]} |Y(t)|
\end{eqnarray*}
and
\begin{eqnarray*}
m_i & = &
 y_i 
 \left(
\frac{
f_i - e^{- \theta \Delta_i} f_{i-1}
}{
1 - e^{-2 \theta \Delta_i}
}
\right)
+
 y_{i-1}
 \left(
\frac{
e^{-2 \theta \Delta_i}f_{i-1}  -
 e^{- \theta \Delta_i} f_{i}
}{
1 - e^{-2 \theta \Delta_i}
} 
\right)
\\
& = &
 y_i 
 \left(
\frac{
f_i - f_{i-1} + f_{i-1} ( 1 - e^{- \theta \Delta_i })
}{
1 - e^{-2 \theta \Delta_i}
}
\right)
+
 y_{i-1}
 \left(
\frac{
f_{i-1} ( e^{- 2 \theta \Delta_i } - e^{- \theta \Delta_i })
+
(f_{i-1} - f_{i}) e^{- \theta \Delta_i}
}{
1 - e^{-2 \theta \Delta_i}
}
\right) \\
& = &
(y_i - y_{i-1}) \frac{f_i - f_{i-1}}{2 \theta \Delta_i}
+ \frac{1}{2} ( y_i - y_{i-1} )  f_{i-1}
+ \delta_{i,n} \Delta_i \sup_{t \in [0,1]} |Y(t)|.
\end{eqnarray*}
Hence, using that $f_k$ is continuously differentiable,
\begin{eqnarray} \label{eq:before:sum:by:part}
\f' \R_{\theta}^{-1} \y
& = &
\frac{1}{2} ( f_1 y_1 + y_2 f_2)
+ \frac{f_1-f_2}{2 \theta \Delta_2} (y_1 - y_2)
+ \delta_{1,n} \Delta_2 \sup_{t \in [0,1]} |Y(t)| \notag \\
& & + 
\sum_{i=3}^{n}
\left(
(y_i - y_{i-1}) \frac{f_i - f_{i-1}}{2 \theta \Delta_i}
+ \frac{1}{2} ( y_i - y_{i-1} )  f_{i-1}
+ \delta_{i,n} \Delta_i \sup_{t \in [0,1]} |Y(t)|
\right) \notag \\
& = &
\delta_{1,n} \sup_{t \in [0,1]} |Y(t)|
+ 
\sum_{i=3}^{n-1}
\left(
\frac{f_i - f_{i-1}}{2 \theta  \Delta_i}
(y_i - y_{i-1})
\right)
+\frac{1}{2}
\sum_{i=3}^{n-1}
\left(
y_i  - y_{i-1} 
\right)  f_{i-1}.
\end{eqnarray}
The absolute value of first sum in \eqref{eq:before:sum:by:part} is equal to, after a summation by part, and using that $f_k$ is twice continuously differentiable,
\begin{eqnarray*}
& & \left |- \sum_{i=3}^{n-1}
y_i
\left(
\frac{f_{i+1} - f_{i}}{2 \theta  \Delta_{i+1}}
-
\frac{f_{i} - f_{i-1}}{2 \theta  \Delta_{i}}
\right)
+ \delta_{n,n} \sup_{t \in [0,1]} |Y(t)|  \right |\\
& \leq & 
 \frac{1}{2 \theta} \left | \sum_{i=3}^{n-1}
y_i
\left(
f'(s_i) + \delta_{i,n} \Delta_{i+1}
-
f'(s_i) + \delta_{i,n} \Delta_{i}
\right)\right |
+ \delta_{n,n} \sup_{t \in [0,1]} |Y(t)| \\
& = & \delta_{n,n} \sup_{t \in [0,1]} |Y(t)|.
\end{eqnarray*}
The second sum in \eqref{eq:before:sum:by:part} is equal to, using that $f_k$ is continuously differentiable,
\[
\sum_{i=3}^{n-1}
\left(
y_i f_{i-1} - y_{i-1} f_{i-2}
+y_{i-1} (f_{i-2} - f_{i-1})
\right)
=
\delta_{n,n} \sup_{t \in [0,1]} |Y(t)|.
\]

Hence 
\[
| \f' \R_{\theta}^{-1} \y | \leq K \sup_{t \in [0,1]} | Y(t) |.
\]
\end{proof}

\begin{lemma} \label{lem:eint:Rmun:un}  
We have, for $k=1,...,p$,
\[
| \e_{i,n}' \R_{\theta}^{-1} \f^{(k)} | 
\leq 
\begin{cases}
K (\Delta_i + \Delta_{i+1}) & \mbox{when $i \not \in \{1,n \}$} \\
K  & \mbox{when $i \in \{1,n \}$}
\end{cases}.
\]
\end{lemma}

\begin{proof}[Proof of Lemma \ref{lem:eint:Rmun:un}]
Let $\f = (f_1,...,f_n)' = \f^{(k)}$. 

i) When $i\notin \{1,n\}$, from \eqref{eq:tridiagonal:inverse},
\begin{eqnarray*}
\e_{i,n}' \R_{\theta}^{-1}\f  &=& 
- \frac{e^{-\theta\Delta_i}}{1-e^{-2\theta\Delta_i}} f_{i-1} +\frac{1}{1-e^{-2\theta\Delta_i}} f_i +\frac{e^{-2\theta\Delta_{i+1}}}{1-e^{-2\theta\Delta_{i+1}}} f_i - \frac{e^{-\theta\Delta_{i+1}}}{1-e^{-2\theta\Delta_{i+1}}} f_{i+1} \\
& = &
\frac{f_i(1 -e^{-\theta\Delta_i}) + e^{-\theta\Delta_i} (f_i - f_{i-1})}{1-e^{-2\theta\Delta_i}} 
+ \frac{f_i (e^{-2\theta\Delta_{i+1}} - e^{- \theta\Delta_{i+1}} ) + (f_i - f_{i+1})e^{-\theta\Delta_{i+1}} }{1-e^{-2\theta\Delta_{i+1}}} \\
& = & \frac{f_i}{2} + \frac{f'(s_i)}{2 \theta} + \delta_{i,n} \Delta_i 
- \frac{f_i}{2} - \frac{f'(s_i)}{2 \theta} + \delta_{i,n} \Delta_{i+1} \\
& = & \delta_{i,n} ( \Delta_i + \Delta_{i+1} ),
\end{eqnarray*}
where we have used \eqref{eq:DL:alpha} and that $f$ is twice continuously differentiable.

ii) Similarly,
\begin{eqnarray*}
 \e_{1,n}' \R_{\theta}^{-1}\f  
 = 
\frac{f(0)}{2} - 
\frac{f'(0)}{2 \theta} + \delta_{1,n} \Delta_2.
\end{eqnarray*}
and
\begin{eqnarray*}
 \e_{n,n}' \R_{\theta}^{-1}\f  
 = 
\frac{f(1)}{2} 
+
\frac{f'(1)}{2 \theta} + \delta_{n,n} \Delta_n.
\end{eqnarray*}
\end{proof}

\begin{lemma} \label{lem:unt:Rmun:un}
We have
\[
 \F' \R_{\theta}^{-1} \F 
= 
\I_f
+\W(\theta)
\]
with $\sup_{ \theta \in [a,A] } || \W(\theta)|| \to_{n \to \infty} 0$ and with $\I_f$ the $p \times p$ matrix defined by
\[
[\I_f]_{kl} =
f_k(0) f_l(0) + \frac{1}{2 \theta} \int_{0}^1  f_k'(t) f_l'(t) dt
+ \frac{\theta}{2} \int_0^1 f_k(t) f_l(t) dt
+ \frac{1}{2} \int_0^1 ( f_k'(t) f_l(t) + f_k(t) f_l'(t)) dt.
\]
Furthermore, $\I_f$ is invertible. 
\end{lemma}

\begin{proof}[Proof of Lemma \ref{lem:unt:Rmun:un}]
Let $k,l \in \{1,...,p\}$ and let $g=f_k$, $h = f_l$, $\g = (g_1,...,g_n)' = \f^{(k)}$ and $\h = (h_1,...,h_n)' = \f^{(l)}$. From \eqref{eq:tridiagonal:inverse}, we have,
\begin{eqnarray} \label{eq:obtaining:If}
[\F' \R_{\theta}^{-1} \F ]_{kl}  & = &
 \g' \R_{\theta}^{-1} \h \notag \\
 & = & \frac{g_1 h_1}{1-e^{-2\theta\Delta_2}} 
+ \frac{g_n h_n}{1-e^{-2\theta\Delta_n}}  
 -\sum_{i=2}^{n} \frac{e^{-\theta\Delta_i}}{1-e^{-2\theta\Delta_i}} g_i h_{i-1} 
 -\sum_{i=2}^{n} \frac{e^{-\theta\Delta_i}}{1-e^{-2\theta\Delta_i}} g_{i-1} h_{i} \notag \\
& & + 
\sum_{i=2}^{n-1} \left( \frac{1}{1-e^{-2\theta\Delta_{i}}} +\frac{e^{-2\theta\Delta_{i+1}}}{1-e^{-2\theta\Delta_{i+1}}} \right) g_i h_i \notag \\
& = & 
\frac{g_1 h_1 -   e^{- \theta \Delta_2} ( g_2 h_1 + g_1 h_2) + g_2 h_2}{1 - e^{- 2 \theta \Delta_2}}
+
\sum_{i=3}^{n} 
\left( 
\frac{ \left( g_i -  e^{-  \theta \Delta_i} g_{i-1} \right) 
\left( h_i -  e^{-  \theta \Delta_i} h_{i-1} \right)
}{
1 - e^{- 2 \theta \Delta_i}}
\right) \notag \\
& = &
\frac{ (g_1 - g_2)(h_1 - h_2) + \theta \Delta_2 (g_2 h_1 + g_1 h_2) + \delta_{1,n} \Delta_2^2 }{2 \theta \Delta_2 + \delta_{1,n} \Delta_2^2} \notag \\
& & +
\sum_{i=3}^{n} 
\frac{1}{1 - e^{- 2 \theta \Delta_i}} 
\Big(
 (g_i - g_{i-1}) (h_i - h_{i-1}) + g_{i-1} (1 - e^{- \theta \Delta_i}) h_{i-1} (1 - e^{- \theta \Delta_i}) 
  \\
 & & + g_{i-1} (1 - e^{- \theta \Delta_i}) (h_i - h_{i-1}) + (g_i - g_{i-1})  h_{i-1} (1 - e^{- \theta \Delta_i}) 
\Big), \notag
\end{eqnarray}
where we have used \eqref{eq:DL:alpha}.
Since $f_k$ and $f_l$ are continuously differentiable we have 
\[
\frac{ (g_1 - g_2)(h_1 - h_2) + \theta \Delta_2 (g_2 h_1 + g_1 h_2) }{2 \theta \Delta_2 + \delta_{1,n} \Delta_2^2}
= g(0) h(0) + \delta_{1,n} \Delta_2.
\]
Also, the element $i$ of the sum in \eqref{eq:obtaining:If} is equal to
\begin{eqnarray*}
& & \frac{1}{2 \theta} \frac{(g_i - g_{i-1})}{\Delta_i}
(h_i - h_{i-1})
+
g_{i-1} h_{i-1} \frac{\theta}{2} \Delta_i
+
\frac{1}{2} g_{i-1} (h_i - h_{i-1})
+
\frac{1}{2} h_{i-1} (g_i - g_{i-1})
+
\delta_{i,n} \Delta_i^2 \\
& = &
\frac{1}{2 \theta} g'(s_i)  h'(s_i) \Delta_i
+
\frac{\theta}{2} g(s_{i-1}) h(s_{i-1})  \Delta_i
+ \frac{1}{2} g(s_i) h'(s_i) \Delta_i 
+ \frac{1}{2} h(s_i) g'(s_i) \Delta_i 
+ \delta_{i,n} \Delta_i^2,
\end{eqnarray*}
since $f$ is twice continuously differentiable.
Hence we have
\begin{eqnarray*}
 \g' \R_{\theta}^{-1} \h
& = &
 g(0) h(0) + \delta_{1,n} \Delta_2 \\
 & &
 + \sum_{i=3}^{n}
 \left[
\Delta_i \left(
 \frac{1}{2 \theta} g'(s_i) h'(s_i)
+
 \frac{\theta}{2} g(s_{i-1}) h(s_{i-1})
+ \frac{1}{2} g'(s_i)h(s_i)
+ \frac{1}{2} g(s_i)h'(s_i)
  \right) +
  \delta_{i,n} \Delta_i^2
  \right] \\
&  = & [\I_f]_{kl} + w(\theta),
\end{eqnarray*}
where $\sup_{\theta \in [a,A]} |w(\theta)| \to_{n \to \infty} 0$,
by dominated convergence. Finally, we have, for $\lambda_1,...,\lambda_p \in \mathbb{R}$,
\[
\sum_{k,l = 1}^p \lambda_k \lambda_l [I_f]_{kl} =
\left(
\sum_{k = 1}^p \lambda_k f_k(0)
\right)^2
+
\frac{1}{2} \int_{0}^1
\left(
\frac{1}{\sqrt{\theta}} 
\sum_{k = 1}^p \lambda_k f_k'(t)
+
\sqrt{\theta} 
\sum_{k = 1}^p \lambda_k f_k(t)
\right)^2 dt.
\]	
Hence, if $\sum_{k,l = 1}^p \lambda_k \lambda_l [I_f]_{kl} =0 $, then $\sum_{k = 1}^p \lambda_k f_k(0)=0$ and for all $t \in [0,1]$, $ \sum_{k = 1}^p \lambda_k f_k'(t) = - \theta \sum_{k = 1}^p \lambda_k f_k(t)$ so that, by linear independence, $\lambda_1=...=\lambda_p = 0$. Hence $\I_f$ is invertible.

\end{proof}

\begin{lemma} \label{lem:y:minus:haty}
We have
\[
\max_{i=1,...,n} \sup_{a \leq \theta \leq A} | y_i - \hat{Y}_{\theta,-i}(s_i) | \leq K \sup_{t \in [0,1]}  | Y(t) |.
\]
\end{lemma}
\begin{proof}[Proof of Lemma \ref{lem:y:minus:haty}]

For $i \not \in \{ 1,...,n\}$, using \eqref{eq:y:minus:haty},
\begin{eqnarray*}
y_i - \hat{Y}_{\theta,-i}(s_i)  
& = & 
\frac{1}{A_i} ( \alpha_i W_i - \alpha_{i+1} e^{- \theta \Delta_{i+1}} W_{i+1} ) \\
& = & 
\delta_{i,n} W_i + \delta_{i,n} W_{i+1} \\
& = & 
\delta_{i,n} \sup_{t \in [0,1]}  | Y(t) |.
\end{eqnarray*} 
Also
\[
y_1 - \hat{Y}_{\theta,-1}(s_1)  
=
y_1 - e^{- \theta \Delta_2} y_2
= \delta_{1,n} \sup_{t \in [0,1]}  | Y(t) |
\]
and similarly $y_n - \hat{Y}_{\theta,-n}(s_n)  
= \delta_{n,n} \sup_{t \in [0,1]}  | Y(t) |$.
\end{proof}

We now prove Lemma \ref{lem:control:rest}.

\begin{proof}[Proof of Lemma \ref{lem:control:rest}]

$\bullet$ \textbf{Term $r_1(\theta) $}:
\begin{eqnarray*}
r_1(\theta) &=& \sum_{i=1}^n \left[ \log \left((\R_{\theta}^{-1})_{ii} - \bar{\epsilon}_{-i} \right) - \log((\R_{\theta}^{-1})_{ii} )\right] = \sum_{i=1}^n  \log \left(1 - \frac{ \bar{\epsilon}_{-i}}{(\R_{\theta}^{-1})_{ii}}\right)
\\
&=& \sum_{i=1}^n  \log \left( 1 - \frac{ (\e_{i,n}'\R_{\theta}^{-1}\F) (\F' \R_{\theta}^{-1} \F)^{-1} (\F' \R_{\theta}^{-1} \e_{i,n}) }{(\R_{\theta}^{-1})_{ii}}\right)
\\
&=& \sum_{i=1}^n  \log \left( 1 - \frac{\delta_{i,n}}{(\R_{\theta}^{-1})_{ii}}\right),
\end{eqnarray*}
from Lemmas \ref{lem:eint:Rmun:un} and \ref{lem:unt:Rmun:un}. 

Now, from \eqref{eq:tridiagonal:inverse}, for $i=2, \dots, n$, we have $\sup_{\theta}  \frac{1}{(\R_{\theta}^{-1})_{ii}} = \delta_{i,n} \Delta_i$ and $\sup_{\theta}  \frac{1}{(\R_{\theta}^{-1})_{11}} = \delta_{1,n} \Delta_2$. Hence
$r_1(\theta) =  \log \left( 1 - \delta_{1,n}\Delta_2\right) + \sum_{i=2}^n  \log \left( 1 - \delta_{i,n}\Delta_i\right)$ so that $\sup_{\theta \in [a,A]} |r_1(\theta)|$ is bounded as $n \to \infty$.\\

$\bullet$ \textbf{Term $r_2(\theta) $}:

For $k \in \mathbb{N}$, let $\f^{(k)}_{-i}$ be obtained by removing component $i$ of $\f^{(k)}$.
We observe that $f_k(s_i) - \r_{\theta,-i}' \R_{\theta,-i}^{-1} \f^{(k)}_{-i}$ can be interpreted as a leave-one-out prediction error for a $n$-dimensional observation vector equal to $\f^{(k)}$. Hence from \eqref{eq:loo:y:pred},
\begin{eqnarray*}
r_2(\theta) &=&\sum_{i=1}^n (\R_{\theta}^{-1})_{ii} \epsilon_{-i}^2 
\\
&=&\sum_{k,l=1}^p 
\left[
\sum_{i=1}^n 
 \frac{1 }{\left[ (\R_{\theta}^{-1})_{ii}\right] }
 (\bbeta_0 - \hat \bbeta_{-i})_k (\bbeta_0 - \hat \bbeta_{-i})_l 
\left( \e_{i,n}' \R_{\theta}^{-1} \f^{(k)} \right)
\left( \e_{i,n}' \R_{\theta}^{-1} \f^{(l)} \right)
\right].
\end{eqnarray*}
For $i \in \{ 1,...,n \}$, one can show that  Lemmas \ref{lem:unt:Rmun:y} and \ref{lem:unt:Rmun:un} remain true with $\F$, $\R_\theta$, $\y$ replaced by $\F_{-i}$, $\R_{\theta,-i}$, $\y_{-i}$. In addition, these (modified) lemmas can be shown to be uniform over $i=1,...,n$. As a consequence, we have, for $i=1, \dots, n$, 
 \begin{equation} \label{eq:bound:hat:beta}
 || \hat \bbeta_{-i}- \bbeta_0 ||
  =
  || (\F_{-i}' \R_{\theta,-i}^{-1} \F_{-i})^{-1}\F_{-i}' \R_{\theta,-i}^{-1}\y_{-i} ||
   \leq K \sup_{t \in [0,1]} | Y(t) |.
 \end{equation}
Also, since from \eqref{eq:tridiagonal:inverse} $\max_{i=1,...,n} \frac{1}{ (\R_{\theta}^{-1})_{ii} } \leq K$, we have 
\[
 \sum_{i=1}^n \frac{1}{(\R_{\theta}^{-1})_{ii} } 
 \left| \e_{i,n}' \R_{\theta}^{-1} \f^{(k)} \right| 
  \left| \e_{i,n}' \R_{\theta}^{-1} \f^{(l)} \right|
  \leq
 K + \sum_{i=2}^{n-1} \delta_{i,n} (\Delta_i + \Delta_{i+1})^2, 
\]
from Lemma \ref{lem:eint:Rmun:un}.
Hence, we have 
\begin{eqnarray*}
r_2(\theta) &\leq & K \sup_{t \in [0,1]}  Y(t)^2 \left( K+\sum_{i=2}^{n-1}(\Delta_i + \Delta_{i+1})^2
\right).
\end{eqnarray*}	

Hence finally,  $ \displaystyle\sup_{\theta \in [a,A]} | r_2(\theta)|$ is a.s. bounded as $n \to \infty$.

$\bullet$ \textbf{Term $r_3(\theta) $}:

\begin{eqnarray*}	
r_3(\theta)  &=&\sum_{i=1}^n (\R_{\theta}^{-1})_{ii}\epsilon_{-i} (y_i - \hat{Y}_{\theta,-i}(s_i)) \\
& =& 
\sum_{k=1}^p \sum_{i=1}^n (\R_{\theta}^{-1})_{ii} (\bbeta_0 - \hat{\bbeta}_{-i})_k [f_k(s_i) - \r_{\theta,-i}' \R_{\theta,-i}^{-1} \f^{(k)}_{-i}] (y_i - \hat{Y}_{\theta,-i}(s_i)).
\end{eqnarray*}	
We make the same observation on $f_k(s_i) - \r_{\theta,-i}' \R_{\theta,-i}^{-1} \f^{(k)}_{-i}$ as for $r_2(\theta)$. Hence we have 
\begin{eqnarray*}	
r_3(\theta)
& = &
\sum_{k=1}^p \sum_{i=1}^n (\bbeta_0 - \hat{\bbeta}_{-i})_k (\e_{i,n}' \R_{\theta}^{-1} \f^{(k)})  (y_i - \hat{Y}_{\theta,-i}(s_i)).
\end{eqnarray*}	
Hence, from \eqref{eq:bound:hat:beta} and Lemmas \ref{lem:eint:Rmun:un} and \ref{lem:y:minus:haty}, we have
\[
| r_3(\theta)  | \leq
\sum_{k=1}^p 
\left(
 \sum_{i=2}^{n-1} \left( \delta_{i,n} (\Delta_i + \Delta_{i+1}) \sup_{t \in [0,1]} Y(t)^2 \right)
+
\delta_{1,n} \sup_{t \in [0,1]} Y(t)^2
\right).
\]
Hence, $r_3(\theta)$ is almost surely bounded as $n \to \infty$.

$\bullet$ \textbf{Term $r_4(\theta) $}:

From Lemmas \ref{lem:eint:Rmun:un} and \ref{lem:unt:Rmun:un}, we have for $i \in \{ 2,...,n-1 \}$, $ \bar{\epsilon}_{-i} \leq K ( \Delta_i + \Delta_{i+1} )^2$ and for $i \in \{ 1,n \}$, $ \bar{\epsilon}_{-i} \leq K $. Hence
\begin{eqnarray*}
|r_4( \theta )|
& \leq & K (y_1 - \hat{Y}_{\theta,-1}(s_1) + \epsilon_{-1} )^2 
+ K (y_n - \hat{Y}_{\theta,-n}(s_n) + \epsilon_{-n} )^2 
\\ 
& & + \sum_{i=2}^{n-1}
(\Delta_i + \Delta_{i+1})^2
(y_i - \hat{Y}_{\theta,-i}(s_i) + \epsilon_{-i} )^2 \\
& \leq &
K \sup_{t \in [0,1]} Y(t)^2
+ K \sum_{i=2}^{n-1}
(\Delta_i + \Delta_{i+1})^2
\sup_{t \in [0,1]} Y(t)^2,
\end{eqnarray*}
using also Lemma \ref{lem:y:minus:haty}, and remarking that the previous treatment of $r_2(\theta)$ implies that $\epsilon_{-i}^2 \leq K \sup_{t \in [0,1]} Y(t)^2$. Hence $r_4(\theta)$ is a.s. bounded as $n \to \infty$.

\end{proof}

\section*{Acknowledgements}

We are grateful to Jean-Marc Azais, for constructive discussions on the topic of this manuscript. This work was partly funded by the ANR project PEPITO.
We thank the associate editor and two referees, whose suggestions lead to a broader scope and an improved presentation of the manuscript.

\bibliographystyle{plain}
\bibliography{Fixed_domain_asymptotics_CV_REVISION_FINAL}
\end{document}